\documentclass[hidelinks]{article}
\usepackage[utf8]{inputenc}
\usepackage[utf8]{inputenc}
\usepackage{mathrsfs}
\usepackage{amssymb}
\usepackage[margin=1in]{geometry}
\usepackage{changepage}
\usepackage{stmaryrd}
\usepackage{xcolor}
\usepackage{bbold}
\usepackage{amsmath,amsthm}
\usepackage{graphicx}
\usepackage{hyperref}
\usepackage{tikz}
\usepackage{soul}
\usepackage{ulem}
  \usepackage{frcursive}

\newtheorem{propo}{Proposition}
\newtheorem{remark}{Remark}
\newtheorem{lemma}{Lemma}

\newtheorem{theorem}{Theorem}
\newtheorem*{theorem*}{Theorem}

\numberwithin{equation}{section}
\newcommand{\vertiii}[1]{{\left\vert\kern-0.25ex\left\vert\kern-0.25ex\left\vert #1 
		\right\vert\kern-0.25ex\right\vert\kern-0.25ex\right\vert}}

\def\aa{\text{{\cursive{a}}}}
\def\bb{\text{{\cursive\slshape{b}}}}
\def\cc{\text{{\cursive{c}}}}
\def\dd{\text{{\cursive{d}}}}

\def\N{\mathbb{N}}
\def\R{\mathbb{R}}

\def\d{\,\mathrm{d}}

\def\I{\mathbb{I}}
\def\fy{\mathfrak{y}}
\def\fz{\mathfrak{z}}

\def\B{\mathbb{B}}
\def\E{\mathbb{E}}

\def\eps{\varepsilon}

\newcommand{\ei}{\textcolor{black}}
\newcommand{\sm}{\textcolor{black}}

\title{\textbf{Weak Error on the densities for the Euler scheme of stable additive SDEs with  Besov drift}}
\author{Mathis Fitoussi\footnote{Laboratoire de Mathématiques et Modélisation d'Evry (LaMME), UMR CNRS 8071, Université d'Evry Val d'Essonne-Paris Saclay, 23 Boulevard de France, 91037 Evry, mathis dot fitoussi at univ-evry dot fr}, Elena Issoglio\footnote{Dipartimento di Matematica Giuseppe Peano, Università di Torino,  Via Carlo Alberto 10, elena.issoglio@unito.it}, Stéphane Menozzi\footnote{Laboratoire de Mathématiques et Modélisation d'Evry (LaMME), UMR CNRS 8071, Université d'Evry Val d'Essonne-Paris Saclay,  23 Boulevard de France, 91037 Evry, stephane dot menozzi at univ-evry dot fr}}
\date{\today}

\begin{document}

\maketitle


\begin{center}
	\textbf{Abstract}
\end{center}
\begin{adjustwidth}{0.8in}{0.8in}
	We are interested in the Euler-Maruyama dicretization of the formal SDE
	\begin{equation}\label{sde-abstract}\notag
		\d X_t = b(t,X_t)\d t + \d Z_t, \qquad X_0 = x \in \R^d,
	\end{equation}
	where $Z_t$ is a symmetric isotropic $d$-dimensional $\alpha$-stable process, $\alpha\in (1,2) $ and the drift $b\in L^r \left([0,T],\B^{\beta}_{p,q}(\R^d,\R^d)\right)$, $\beta<0$ is distributional and the parameters $ $ satisfy some constraints which guarantee weak-well posedness. Defining an appropriate Euler scheme, we show that, denoting $\gamma:=\alpha+2\beta-d/p - \alpha/r-1>0$, the weak error on densities related to this discretization converges at the rate $(\gamma-\varepsilon)/\alpha$ for any $\varepsilon\in (0,\gamma) $.\\
\end{adjustwidth}


\section{Introduction}
    For a fixed finite time horizon $T>0$, we are interested in  the Euler-Maruyama dicretization of the \textit{formal} SDE
	\begin{equation}\label{sde}
	X_t =x+\int_0^t b(s,X_s)\d s + \d Z_s,  \qquad \forall t \in [0,T],
	\end{equation}
	where $Z_t$ is a symmetric isotropic $d$-dimensional $\alpha$-stable process, $\alpha\in (1,2) $ and $b\in L^r \left([0,T],\B^{\beta}_{p,q}(\R^d,\R^d)\right)$, $\beta<0$.
	{\color{black}The (negative regularity) Besov space $\B^{\beta}_{p,q}(\R^d,\R^d)$ will be defined in Subsection \ref{subsec-besov}.}
	 In this pure-jump setting, it was established in \cite{CdRM22} that well-posedness of the generalized martingale problem holds for the generator formally associated with \eqref{sde} under the condition

\begin{equation}\label{serrin}
	\alpha \in \left( 1
	,2 \right), \qquad \beta \in \left( \frac{1-\alpha+\frac{d}{p}+\frac{\alpha}{r}}{2} ,0\right),
\end{equation}
which we assume to hold throughout this paper. Density estimates on the time marginals of \eqref{sde} were obtained in \cite{Fit23}.\\
	
	The goal of this paper is to prove a convergence rate for the weak error on densities associated with an {\textit{appropriate}} Euler scheme for \eqref{sde}. The proof consists in approaching $b(\cdot,\cdot)$ with a sequence $(\mathfrak{b}(\cdot,\cdot,h))_{h\geq 0}$ of bounded H\"older functions, where the mollification parameter $h$ is also the time step of the scheme. 
	\subsection{Definition of the scheme}\label{subsec-def-scheme}
	To introduce the scheme associated with the formal previous SDE \eqref{sde}, one first needs to recall that the precise meaning to be given to the SDE, following \cite{CdRJM22} in the pure-jump setting, inspired by \cite{DD16} in the Brownian setting,  is:
	\begin{equation}\label{DYN_DIFF}
	X_t=x+\int_0^t {\mathfrak b}(s,X_s,\d s)+Z_t,
	\end{equation}
	where for all $(s,z)\in [0,T]\times \R^d,h>0$,
	\begin{equation}
\label{DEF_DRIFT}
\mathfrak{b}(s,z,h) := \int_s^{s+h} \int b(u,y)p_\alpha(u-s,z-y) \d y \d u=\int_s^{s+h} P_{u-s}^\alpha b(u,z)   \d u,
	\end{equation}
$p_\alpha(v,\cdot)$ denoting the density of the $\alpha $-stable driving noise $(Z_v)_{v\geq0}$ at time $v$ and $P^\alpha $ the associated semi-group.  The integral in \eqref{DYN_DIFF} is intended as a {\color{black}nonlinear} Young integral obtained by passing to the limit  in a suitable procedure aimed at reconstructing the drift (see again \cite{CdRJM22}). The resulting drift in \eqref{DYN_DIFF} is, \textit{per se}, a Dirichlet process (as it had already been indicated in the literature, see e.g. \cite{ABM20} and references therein). Importantly, the dynamics in \eqref{DYN_DIFF} also naturally provides a corresponding approximation scheme to be analyzed. Note that, in order to give a precise meaning to the integral appearing in \eqref{DYN_DIFF}, we need the following condition: 
\begin{equation}
	\alpha \in \left( \frac{1+\frac{d}{p}}{1-\frac{1}{r}},2 \right)\qquad \beta \in \left( \frac{1-\alpha+\frac{2d}{p}+\frac{2\alpha}{r}}{2} ,0\right),
\end{equation}
which is more stringent than \eqref{serrin}. Interestingly enough, this condition does not appear \textcolor{black}{elsewhere} in the present work since we only consider the time marginals of the process. \\

	We will use a discretization scheme with $n$ time steps over $[0,T]$, with constant step size $h:=T/n$. For the rest of this paper, we denote, $\forall k \in \{{0},...,n\}, t_k := kh$ and $\forall s>0, \tau_s^h := h \lfloor \frac{s}{h} \rfloor \in (s-h,s]$, which is the last grid point before time $s$. Namely, if $s\in [t_k,t_{k+1}), \tau_s^h = t_k$. \\

We can now define the related Euler scheme $X^h $, starting from $X_0^h=x$, on the time grid as
\begin{equation}\label{euler-scheme-besov_GRID}
		 X_{t_{i+1}}^h = X_{t_i}^h+  \mathfrak{b}(t_i,{X}_{t_i}^h,h)+Z_{t_{i+1}}-Z_{t_i}.
	\end{equation}
	We have precisely
	{\color{black} used the quantity $\mathfrak{b}(t_i,X_{t_i}^h,h)$ defined in \eqref{DEF_DRIFT} as an approximation of the nonlinear Young integral $\int_{t_i}^{t_{i+1}} {\mathfrak b}(s,X^h_s,\d s)$,} 
which served to define the limit dynamics \eqref{DYN_DIFF} for the SDE, with a time argument corresponding to the chosen time step.\\

	Set now for $(s,z)\in [0,T]\backslash {(kh)_{k\in \{0,\cdots,n\}}}\times \R^d$,
	\begin{equation}\label{DEF_BH_SCHEME}
{\mathfrak b}_h(s,z):=P_{s-\tau_s^h}^\alpha b(s,z). 
	\end{equation}
Observe from that definition {\color{black}and using \eqref{DEF_DRIFT}} that, on any time step, the drift also writes as
\begin{equation}
\label{EXPR_DRIFT}
 \mathfrak{b}(t_i,{X}_{t_i}^h,h)=\int_{t_i}^{t_{i+1}} \mathfrak b_h(u,X_{t_i}^h)\d u=\E[\mathfrak{b}_h(U_{i},{X}_{t_i}^h)|{X}_{t_i}^h]h,
 \end{equation}
where  the	$(U_k)_{k\in \N} $ are independent random variables, 
independent as well from the driving noise, s.t. $U_k\overset{({\rm law})}={\mathcal U}([{t_k,t_{k+1}}]) $, i.e. $U_k $ is uniform on the time interval $[t_k,t_{k+1}]$.\\

From a practical viewpoint, the above time and spatial expectations {\color{black} \eqref{EXPR_DRIFT} and \eqref{DEF_BH_SCHEME}, respectively,} will anyhow have to be approximated if one was to fully implement this discretization. These computations are however case-dependent. We mention that a usual way to spare one of these approximations  consists in randomizing the time, namely this amounts to consider $\tilde {\mathfrak{b}}(t_i,{X}_{t_i}^h,h):=\mathfrak{b}_h(U_{i},{X}_{\tau_s^h}^h) $. This approach was successfully carried out for Lebesgue drifts (see \cite{JM211}, \cite{FJM24}) and also allowed in the spatial Hölder setting to achieve the somehow expected convergence rates without any requirements on the time regularity (see \cite{FM24}).\\

Anyhow, in the current singular setting it seems difficult to benefit from such an effect in the sense that without any additional time integration we do not have controls on the approximate drift norm. This can be  seen e.g. in \eqref{CTR_PONCTUEL_BH} below or in the proof of the sensitivity analysis involving the local transitions (see proof of control \eqref{besov-estimate-gammah-stable-PERTURB_DRIFT}).\\

The representation \textcolor{black}{\eqref{EXPR_DRIFT}} naturally suggests to extend the dynamics of the scheme in continuous time as follows 
	\begin{equation}\label{euler-scheme-besov}
		 X_{t}^h = X_{\tau_t^h}^h+  \mathfrak{b}(\tau_t^h,{X}_{\tau_t^h}^h,t-\tau_t^h)+Z_{t}-Z_{\tau_t^h}=X_{\tau_t^h}^h+ \int_{\tau_t^h}^t \mathfrak{b}_h(s,{X}_{\tau_t^h}^h)\d s+Z_t-Z_{\tau_t^h},
	\end{equation}
which gives an  extension in integral form which is  similar to the dynamics of Euler schemes involving non-singular drifts, i.e. it is an Itô type process and the approximate drift appears through a usual time integral. 

	\subsection{Euler Scheme - State of the art}

	Estimates on the weak error involving a suitably smooth test function have been studied for a long time. In the Brownian setting, we can mention, the seminal works of \cite{TT90} (smooth coefficients and test function) and \cite{MP91} (non-degeneracy and H\"older coefficients). Going to density (i.e. taking a dirac mass as test function) requires some additional non-degeneracy of the noise. We can refer to  \cite{BT96} in a suitable H\"ormander setting and \cite{KM02} in a non-degenerate case, which deal with smooth coefficients as well as \cite{KM10} for stable-driven SDEs with smooth coefficients. The key tool to derive these results resides in studying the smoothness of the backward Kolmogorov equation with the corresponding test function.\\
	
	When the coefficients are smooth, the aforementioned works prove that the weak error rate of the Euler scheme is of order 1 with respect to the time step $h$. When the drift and the (possibly non-trivial) diffusion coefficients are $\eta$-H\"older, in the Brownian setting, the expected rate falls to $h^{\frac{\eta}{2}}$ (see \cite{MP91} for smooth test functions and \cite{KM17} for the error on densities). To the best of the authors' knowledge, this has not yet been proven in the pure-jump setting $\alpha<2$, but an associated rate of order $\eta / \alpha$ would be expected.\\
	
	For a more detailed review of those topics, we refer to the introduction of \cite{FM24}.\\
	
	Recently, a series of works considered the Euler approximation of SDEs with stable additive noise and low-regularity drifts and a randomization in the time variable for the scheme. The first work in this direction goes back to \cite{BJ20}, which addressed the case of a Brownian SDE with bounded drift, achieving a convergence rate of order $1/2$ up to a logarithmic factor for the total variation error. The ideas introduced therein have been generalised in \cite{JM211} in order to handle Lebesgue drifts in $L^q_t-L^p_x$ under the Krylov-R\"ockner condition $2/q + d/p <1$. The achieved rate on densities then writes $(1-2/q-d/p)/{2}$, which corresponds to the margin in the Krylov-R\"ockner condition multiplied by the self-similarity index of the noise. This has been extended to the strictly stable case in \cite{FJM24} under the condition $\alpha/q + d/p <\alpha-1$, achieving the rate $(\alpha-1-\alpha/q-d/p)/\alpha$, although in this setting, the above condition only ensures \textit{weak} well-posedness of the underlying SDE (see \cite{XZ20} for the conditions leading to strong well-posedness). Eventually, in a H\"older setting, it was derived in \cite{FM24} that for bounded $\eta$-H\"older (in space) coefficients with $\alpha\in (1,2]$, the convergence rate writes $(\alpha+\eta-1)/\alpha$. Keeping in mind that weak well-posedness holds for $\alpha+\eta-1>0$ (see \cite{CZZ21}), this rate again corresponds to the margin multiplied by the self-similarity index of the noise. All those works rely on first deriving heat kernel estimates for the diffusion and the scheme in order to bypass the lack of regularity of the drift. In the present work, we manage to apply this approach to handle Besov drifts through the previously introduced scheme, achieving, up to some $\eps>0$ (which is intrisic to the dsitributional setting), a rate which corresponds to the margin appearing in the condition required for weak well-posedness multiplied by the self-similarity index of the noise, thus emphasizing the robustness of the approach.\\
	
	Nevertheless, another approach has proven fruitful when handling SDEs with (possibly fractional) Brownian noise, mainly to derive strong error rates, leading, surprisingly, to better convergence rates using the stochastic sewing lemma (see \cite{Le20}). In the Krylov-R\"ockner setting {\color{black}for the drift}, the strong error rate derived in \cite{LL22} is $1/2$ up to a logarithmic factor. For the weak error, an approach involving the stochastic sewing lemma has been successfully applied by \cite{Hol22} for H\"older drifts, leading to the rate $(\eta+1)/2$, up to some $\eps>0$ (which is intrisic to the sewing lemma), thus achieving similar rates to those discussed in the previous paragraph in an analog setting. \textcolor{black}{We can as well mention the work by Hao \textit{et al.} \cite{hao:le:ling:24} who obtained order $1/2$ for a semi-explicit scheme involving a spatial convolution in the Krylov-Röckner setting, investigating weaker norms than the supremum of the difference of the densities}.
	
	{\color{black}Regarding numerical schemes for SDEs with distributional drifts} let us also mention {\color{black} \cite{dAGI, CIP23} in a Brownian scalar setting and \cite{GHR25} in a multi-dimensional fractional Brownian setting, who derived convergence rates} for the strong error. The results of the current paper are, to the best of our knowledge, the first ones concerning multi-dimensional SDEs with distributional drifts in the strictly stable setting. \textcolor{black}{Let us as well mention that the restriction to the strictly stable setting is mainly due to the lack in the literature of corresponding heat kernel estimates in the Brownian driven case. We can mention the work by Perkowski and Van Zuijlen \cite{PvZ22} providing estimates of this type for $b\in B_{\infty,1}^{\beta}$.
	 We anyhow stress that the results and the proofs we present in this work could somehow readily extend to the Brownian setting provided the appropriate heat kernel estimates work}.

	 \subsection{Driving noise and related density properties}\label{subsec-noise}
	Let us denote by $\mathcal{L}^\alpha$ the generator of the driving noise $Z$. When $\alpha \in (1,2)$, in whole generality, the generator of a symmetric stable process writes, $\forall \phi \in C_0^\infty (\R^d,\R)$ (smooth compactly supported functions),
	\begin{align*}
		\mathcal{L}^\alpha \phi (x)&= \mathrm{p.v.} \int_{\R^d} \left[ \phi(x+z) - \phi(x)\right]\nu(\d z)\\
		&=\mathrm{p.v.}\int_{\R_+}\int_{\mathbb{S}^{d-1}}\left[ \phi(x+\rho \xi) - \phi(x)\right]\mu(\d \xi)\frac{\d \rho}{\rho^{1+\alpha}}
	\end{align*}
	(see \cite{Sat99} for the polar decomposition of the stable L\'evy measure) where $\mu$ is a symmetric measure on the unit sphere $\mathbb{S}^{d-1}$. We will here restrict to the case where $\mu  =m$  the Lebesgue measure on the sphere but it is very likely that the analysis below can be extended to the case where
	 $\mu$ is symmetric and $\exists \kappa \geq 1 : \forall \lambda \in \R^d$,
	\begin{equation*}
C^{-1} m(\d \xi) \leq \mu (\d \xi) \leq C m(\d \xi),
	\end{equation*} 
 i.e. it is equivalent to the Lebesgue measure on the sphere. Indeed, in that setting
 Watanabe (see \cite{Wa07}, Theorem 1.5) and Kolokoltsov (\cite{Kol00}, Propositions 2.1--2.5) showed that if $C^{-1} m(\d \xi) \leq \mu (\d \xi) \leq C m(\d \xi)$, the following estimates hold: denoting $p_\alpha(v,\cdot)$ the density of the noise at time $v$, there exists a constant $C$ depending only on $\alpha,d$, s.t. $\forall v\in \R_+^*, z\in \R^d$,
    \begin{equation}\label{ARONSON_STABLE}
    C^{-1}v^{-\frac{d}{\alpha}}\left( 1+ \frac{|z|}{v^{\frac{1}{\alpha}}} \right)^{-(d+\alpha)}\leq p_\alpha(v,z)\leq Cv^{-\frac{d}{\alpha}}\left( 1+ \frac{|z|}{v^{\frac{1}{\alpha}}} \right)^{-(d+\alpha)}.
    \end{equation}
 
 	Note that, additionally to the previous non-degeneracy condition, in order to have the estimates on the derivatives of $p_\alpha$  appearing in Lemma \ref{lemma-stable-sensitivities}, some smoothness is required on the Lebesgue density of $\mu$.

	On the other hand let us mention that the sole non-degeneracy condition  
		\begin{equation*}
		\kappa^{-1} |\lambda|^\alpha \leq \int_{\mathbb{S}^{d-1}} |\lambda \cdot \xi|^\alpha \mu(\d \xi) \leq \kappa |\lambda|^\alpha,
	\end{equation*} 
	does not allow to derive \textit{global} heat kernel estimates for the noise density. In \cite{Wa07}, Watanabe investigates the behavior of the density of an $\alpha$-stable process in terms of properties fulfilled by the support of its spectral measure $\mu$. From this work, we know that whenever the measure $\mu$ is not equivalent to the Lebesgue measure $m$ on the unit sphere, accurate estimates on the density of the stable process can be delicate to obtain. \\

	Let us now introduce
	\begin{equation}
		\bar{p}_{\alpha} (v,z) :=    C_\alpha {v^{-\frac{d}{\alpha}}} \left(1+\frac{|z|}{v^{\frac{1}{\alpha}}} \right)^{-(d+\alpha)} \qquad\qquad v>0,z\in \R^d,\label{DEF_P_BAR}
	\end{equation}
	where $C_\alpha$ is chosen so that $\forall v>0, \int \bar {p}_{\alpha} (v,y) \d y = 1$. Observe as well that from the definition in \eqref{DEF_P_BAR} we readily have {\color{black}from \eqref{ARONSON_STABLE}} the following important properties:
	
	\begin{lemma}[Convolution properties and spatial moment for $\bar p_\alpha$]\textcolor{white}{a\\}
		\begin{itemize}
			\item (Approximate) convolution property: there exists ${\mathfrak c}\ge 1$ s.t. for all $(u,v)\in (\R_+^*)^2,\   (x,y)\in (\R^d)^2$,
			\begin{equation}
				\label{APPR_CONV_PROP}
				\int_{\R^d}  \bar p_{\alpha}(u,z-x) \bar p_{\alpha}(v,y-z) \d z\le {\mathfrak c}\bar   p_{\alpha}(u+v,y-x). 
			\end{equation}
			\item Time-scale for the spatial moments: for all $0\le \delta<\alpha, \alpha\in (1,2) $, there exists $C_{\alpha,\delta} $ s.t.
			\begin{equation}
				\label{spatial-moments}
				\int_{\R^d} |z|^\delta\bar  p_\alpha(v,z)\d z\le C_{\alpha,\delta} v^{\frac{\delta}{\alpha}}.
			\end{equation}
		\end{itemize}
	\end{lemma}
%

From now on, we will use assume w.l.o.g that $0<  T<1$. The results below can be extended to an abitrary fixed $T>0$ through a simple iteration procedure. For two quantities $A$ and $B$ the symbol $A\lesssim B $ {\color{black} is used} whenever there exists a constant $C:=C({d},b,\alpha)$ s.t. $A\le CB $. Namely,
\begin{equation}\label{DEF_LESSSIM}
	A\lesssim B \Longleftrightarrow \exists C:=C({d},b,\alpha),\ A\le CB.
\end{equation}
We will also use the notation $A \asymp B$ whenever $A\lesssim B$ and $B\lesssim A$. \textcolor{black}{Also, for any $\ell\in [1,\infty]$, $\ell '$ will always denote its conjugate exponent, i.e. $\frac{1}{\ell}+\frac{1}{\ell '} = 1$.}

	\begin{lemma}[Stable sensitivities - Estimates on the $\alpha $-stable kernel {\color{black}${p}_\alpha$}]\label{lemma-stable-sensitivities}
	For each multi-index  $\zeta$ with length $|\zeta|\leq {2}$,  for all $0<u\leq u'<+\infty$, $(z,z')\in (\R^d)^2$, \ei{and $\delta\in \{0,1\}$,}
	\begin{itemize} 
		\item Spatial derivatives: 
		\begin{align}\label{derivatives-palpha}
			\left|\partial_u^\delta \nabla_z^\zeta p_{\alpha} (u,z)\right| \lesssim \frac{1}{u^{\delta + \frac{|\zeta|}{\alpha}}}\bar  p_\alpha (u,z).
		\end{align}
		\item Time H\"older regularity: for all $\theta\in [0,1] $,
		\begin{align}\label{holder-time-palpha}
			\left|\partial_u^\delta\nabla_z^\zeta p_{\alpha} (u,z)-\partial_u^\delta\nabla_z^\zeta p_{\alpha} (u',z)\right| \lesssim \frac{|u-u'|^\theta}{u^{\delta + \theta+\frac{|\zeta|}{\alpha}}} \left( \bar p_\alpha (u,z)+ \bar p_\alpha (u',z)\right).
		\end{align}
		\item Spatial H\"older regularity: for all $\theta \in [0,1]$,
		\begin{align}\label{holder-space-palpha}
			\left|\partial_u^\delta\nabla_z^\zeta p_{\alpha} (u,z)-\partial_u^\delta\nabla_z^\zeta p_{\alpha} (u,z')\right| \lesssim \left(\frac{|z-z'|^\theta}{u^{\frac{\theta}{\alpha}}} \wedge 1\right)\frac{1}{u^{\delta + \frac{|\zeta|}{\alpha}}}\left(\bar  p_\alpha (u,z)+ \bar  p_\alpha (u,z')\right).
		\end{align}
		\item Convolution {\color{black}(for $\bar{p}_\alpha$)}: for all $x,y \in (\R^d)^2$, $\forall 0 \leq s \leq u \leq t$, $\forall \ell\geq 1$,
		\begin{equation}\label{p-q-convo}
			\Vert \bar{p}_\alpha (t-u,\cdot-y) \bar{p}_\alpha (u-s,x-\cdot)\Vert_{L^{\ell'}} \lesssim \left[ \frac{1}{(t-u)^{\frac{d}{\alpha  \ell}}} +\frac{1}{(u-s)^{\frac{d}{\alpha  \ell}}}\right] \bar{p}_\alpha (t-s,x-y).
		\end{equation}
		\item Besov norm: for all $\vartheta \in \R_+$, $(\ell,m)\in [1,+\infty]^2$
		\begin{equation}
			\label{besov-norm-stable-kernel} \Vert p_\alpha (t,\cdot)\Vert_{ \B_{\ell,m}^{\vartheta}}  \lesssim t^{-\frac{\vartheta}{\alpha}-\frac{d}{\alpha \ell'}}
		\end{equation}
	\end{itemize}
\end{lemma}
The controls of Lemma \ref{lemma-stable-sensitivities} are somehow standard. Up to equation \eqref{p-q-convo}, a proof can be found e.g.\ in \cite{FJM24}. \textcolor{black}{Importantly, note that those controls are valid both for $\bar{p}_\alpha$ and $p_\alpha$. For \eqref{besov-norm-stable-kernel} we can also refer to Lemma 12 in \cite{CdRM22}.} 
\sm{
\begin{remark}[About pointwise bounds in the diagonal regime]\label{REM_DIAG_REGIME_DENS}
It is plain, from the expression \eqref{DEF_P_BAR} to derive that, for all $z,z'\in \R^d,\ u>0 $, as soon as $|z'-z|\le u^{\frac 1\alpha} $, which corresponds to the \textit{diagonal regime}, in the sense that the difference between the spatial points considered is smaller than the typical time scale, then
$$ \bar  p_\alpha (u,z')\lesssim \bar  p_\alpha (u,z).$$
This can be established following  the lines of the proof of Lemma \ref{Lemma_TRANS_SCHEME} below.
\end{remark}
}
	
    \subsection{Main result}\label{subsec-main-result}
  {\color{black} 
  Let $X_t^{0,x}$ be the canonical process associated to the solution of the generalized martingale problem with initial condition $x$ at time 0 under assumption \eqref{serrin}.
  Then, in \cite{Fit23} it was shown that $X_t^{0,x}$ admits a density which is denoted by $\Gamma(0,x,t,\cdot)$. Similarly, for the associated scheme \eqref{euler-scheme-besov} we denote by $\Gamma^h(t_i, x, t, \cdot)$ the density of $X^h_t$ for $t\in[t_i, t_{i+1})$  conditioned on the value $x$ at time $t_i$. The existence of such density and its properties are collected in 
 {Section \ref{SSC:duhamel}} below. 
  Recall that the ``gap to singularity'' is defined as 
\begin{equation}\label{eq:gamma}
\gamma:= \alpha+2\beta -\frac dp - \frac \alpha r-1
\end{equation}
and it is positive by assumption \eqref{serrin} on the parameters.
The main result of the paper is the following theorem:
	\begin{theorem}[Convergence Rate for the stable-driven Euler scheme with Besov drift]\label{thm-besov}
Let \eqref{serrin} holds and let the drift  $\mathfrak b$ be an element of $ L^r\left([0,T],\B^{\beta}_{p,q}(\R^d,\R^d)\right)$ for some $r\in[1, \infty]$.
	Denoting by $\Gamma $ and $ \Gamma^h$ the respective densities of the SDE \eqref{sde} and its Euler scheme defined in \eqref{euler-scheme-besov}, for all $\eps\in \textcolor{black}{>0}$
	there exists a constant $C:=C({d},b,\alpha,T,\eps,\textcolor{black}{\beta})<\infty$ s.t. for all $h=T/n$ with $n\in\N^*$, and all $t\in(0,T]$, $x,y\in \R^d $,
	\begin{align}
		| \Gamma^h(0,x,t,y)-\Gamma(0,x,t,y)| &\leq C h^{\frac{\gamma-\eps}{\alpha}}\sm{\bar p_\alpha(t,y-x)},\label{BD_THM-besov}
	\end{align}
	where $\gamma=\alpha +2\beta-\frac{d}{p}-\frac{\alpha}{r}-1 >0$ is  the ``gap to singularity" in the Besov case.
\end{theorem}

\begin{remark}[About the convergence rate]\label{REM_AFTER_THM}
Let us briefly discusss the convergence rates. For simplicity assume that $p=r
=\infty$. If $\alpha=2-\varepsilon,\ \varepsilon>0$, then for $\beta >(-1+\varepsilon)/2$ the approximation scheme is convergent and will give a convergence rate of order $(1-\varepsilon)/2+\beta $ which goes to 0 when $\beta $ decreases to the limit threshold $-1/2$.  On the other hand, for small $\beta $, e.g. $\beta=-\varepsilon/2 $ meaning that we are close to the usual function setting, then the convergence rate is $1/2-\varepsilon $ and there is somehow a continuity result w.r.t. to the rate $1/2$ obtained  in \cite{JM211} for an $L^\infty $ drift. Of course, the smaller the $\alpha $, the less singular the corresponding drift. If $\alpha$ goes to close to $1 $, the worst drift that can be handled is barely distributional, i.e. $\beta $ will be close to 0, with associated very small convergence rate.
\textcolor{black}{We also mention, in the current distributional setting, the recent paper by Hao and Wu \cite{hao:wu:26} in which  a similar convergence rate is obtained for the considered case  $p=r=\infty $ therein and for the total variation distance between the laws. Importantly, the semi-group setting, instead of the densities here, allows to consider more general stable driving {noises}, with possibly singular non degenerate spectral measures including e.g. cylindrical ones.}
\end{remark}
\begin{remark}[About test functions.]
We observe that the main result implies that the classical weak rate for a class of bounded functions is also $\frac{\gamma-\varepsilon}{\alpha}$. Indeed for a bounded measurable function $F$  and denoting by $\mathbb E_{0,x}[F(Y_t)]$ the expectation of a stochastic process $Y$ at time $t$ evaluated in $F$ given that $Y_0=x$, one has 
\begin{align*}
\big |\mathbb E_{0,x}[F(X^h_t)] - & \mathbb E_{0,x}[F(X_t)] \big |   = \big |  \int F(y) (\Gamma^h (0,x,t,y)- \Gamma(0,x,t,y)) dy  \big |\\
&\leq   \|F\|_\infty \int p_\alpha (0,x,t,y) C  h^{\frac{\gamma-\varepsilon}\alpha} dy 
\leq C'  h^{\frac{\gamma-\varepsilon}\alpha}.
\end{align*}
Let us point out that, for the moment we are not able to take advantage of some additional smoothness of the \textit{test function} $F$ in this negative regularity setting. This was somehow obtained in the Lebesgue setting, see \cite{hao:le:ling:24}, \cite{hao:le:ling:26} and \cite{jour:meno:26}  who address as well the Hölder case.
\end{remark}

{
\begin{remark}[About the driving noise]
We already mentioned the work \cite{hao:wu:26} which handles the error in total variation for a wider class of $\alpha $-stable noises with $\alpha\in (1,2) $. One can naturally wonder whether our results extend or not to the Brownian setting, i.e. for $\alpha=2 $. The main ingredient needed concerns heat kernel estimates on the diffusion and its related scheme. In the Brownian setting with Besov drifts, for $p=r=\infty $, such controls have been obtained for the diffusion in \cite{meno:pagl:26}. It seems reasonable, see e.g. Remark 1 of the quoted reference, that corresponding heat kernel estimates can be obtained for general indexes $p,q,\beta,r $ under the more general condition on $\beta $ appearing in \eqref{serrin}. Hence, we believe that similar controls should as well hold true for the scheme and that the error analysis developed here should extend to the Gaussian driven case, under \eqref{serrin} with $\alpha=2$, with final convergence rate as in \eqref{BD_THM-besov} with $\alpha= 2$.  
\end{remark}
}

To establish the main result, we will proceed through a comparison of the Duhamel representation of the densities given in Proposition \ref{prop-main-estimates-D} below.  Actually, because of the distributional nature of the drift involved, we will need to control the $\rho>\beta $-Hölder norm of the difference of the densities, which will somehow be the quantity to be \textit{gronwallized}. Additionally to the errors which already appeared in the previous related works with less singular drifts, e.g. \cite{JM211}, \cite{FJM24}, \cite{FM24}, which were related to sensitivities of heat kernels (of the driving noise and some forward time sensitivity of the diffusion/scheme), there will be a new term to be analyzed involving the approximation of the drift, term $\Delta_3 $ in the error decomposition \eqref{DECOUP_ERR}. Handling this new term will again involve heat kernel controls of the underlying objects. Indeed, since the natural idea consists in estimating the difference of the drifts in some function space with the biggest possible negative regularity, in order to have the biggest possible convergence rate, this in turn requires to have the corresponding positive regularity index for the other terms appearing which are precisely the heat kernels.

The paper is organized as follows. We recall in Section \ref{SEC_BESOV_AND_DRIFT} some useful results on Besov spaces and controls related to the approximated drift
$\mathfrak b_h $ of the Euler scheme. We state in Section \ref{TOOLS_HK_AND_MISC_CONV} some key tools and technical results for the error analysis: Duhamel representations of the densities and convolution estimates in Besov norms. Section \ref{SEC_PROOF_THM} is dedicated to the proof of the main theorem. Eventually, we prove in Section \ref{sec-proofs-lemmas} the main technical results stated in Section \ref{TOOLS_HK_AND_MISC_CONV}, while Appendix \ref{APP_HK} provides a proof of the estimates  stated in the above Proposition \ref{THE_PROP} for the Euler scheme.


	\section{About Besov spaces and related controls on the mollified drift} \label{SEC_BESOV_AND_DRIFT}

\subsection{Definition and related properties}\label{subsec-besov}

We first recall that \textcolor{black}{denoting by ${\mathcal S'}( \R^d) $ the dual space of the Schwartz class ${\mathcal S}( \R^d) $}, for $\ell,m\in (0,+\infty] $, $\vartheta\in \R $, the {\color{black}inhomogeneous} Besov space $\B^\vartheta_{\ell,m}$  can be characterized {\color{black}(see \cite{Tri06})}  with
\[
{\color{black} \B^\vartheta_{\ell,m}(\R^d)} = \B^\vartheta_{\ell,m}=\left\{f\in\textcolor{black}{\mathcal S'}( \R^d)\,:\,\Vert f\Vert_{\B^\vartheta_{\ell,m}}:=\|\mathcal F^{-1}(\phi\mathcal F(f))\|_{L^\ell}+\mathcal T_{\ell,m}^\vartheta(f)<\infty\right\},
\]
\begin{align}
	\mathcal T_{\ell,m}^\vartheta(f):=&\left\{
	\begin{aligned}
		&
		\left(\int_0^1\,\frac{\d v}{v}v^{(n-\vartheta/{\alpha})m}\Vert\partial^n_v p_\alpha(v,\cdot)\star f\Vert^m_{L^\ell}\right)^{\frac 1m}\,\text{ for }\,1\le m<\infty,\\
		&\sup_{v\in(0,1]}\left\{v^{(n-\vartheta/{\alpha})}\Vert\partial^n_v p_\alpha(v,\cdot)\star f\Vert_{L^\ell}\right\}\,\text{for}\,m=\infty,
	\end{aligned}
	\right. 
	\label{HEAT_CAR}
\end{align}
with $\star$ denoting the spatial convolution, $n$ being any non-negative integer (strictly) greater than $\vartheta/{\alpha}$, the function $\phi$ being a $\mathcal C^\infty_0$-function (infinitely differentiable function with compact support) such that $\phi(0)\neq 0$, and $p_\alpha(v,\cdot)$ denoting \textcolor{black}{as above} the density function at time $v$ of the $d$-dimensional isotropic stable process. 
{\color{black}The term $\mathcal T_{\ell,m}^\vartheta(f)$ appearing in the norm is called ``thermic part'' of the  Besov norm, and this characterization of the norm and related Besov spaces is \textcolor{black}{often} called the ``thermic characterization''.}\\

For our analysis we will rely on the following important inequalities:
\begin{itemize}
\item \textcolor{black}{Embeddings between Lebesgue and $B^0_{\ell,m}$-spaces (\cite[Prop. 2.1]{Sawano18}):
\begin{equation}
\forall 1 \le \ell \le \infty,\qquad B^0_{\ell,1}\hookrightarrow L^\ell \hookrightarrow B_{\ell,\infty}^0.\label{EMBEDDING}
\end{equation}
}
\item Product rule: for all $\vartheta \in \R$, $(\ell,m)\in [1,+\infty]^2$ and $\rho>\max\Big(\vartheta,-\vartheta\Big)$, $\forall (f,g)\in \B_{\infty,\infty}^\rho \times \B_{\ell,m}^\vartheta $,
	\begin{equation}\label{PR}
		\|f \cdot g\|_{ \B_{\ell,m}^\vartheta} \le \|f \|_{\B_{\infty,\infty}^\rho} \|g\|_{ \B_{\ell,m}^\vartheta}.
	\end{equation}
	See Theorem 4.37 in \cite{Sawano18} for a proof.
	\item Duality inequality: for all $\vartheta \in \R$, $(\ell,m)\in [1,+\infty]^2$, with $m'$ and $\ell'$ respective conjugates of $m$ and $\ell$, and $(f,g)\in \B_{\ell,m}^\vartheta \times \B_{\ell',m'}^{-\vartheta}$,
	\begin{equation}\label{dual-ineq}
		\textcolor{black}{|\langle f,g\rangle_{B_{\ell,m}^{\vartheta},B_{\ell',m'}^{-\vartheta}}|:=}\left| \int f(y)g(y) \mathrm{d}y \right| \leq \Vert f \Vert_{\B_{\ell,m}^\vartheta} \Vert g \Vert_{\B_{\ell',m'}^{-\vartheta}},
	\end{equation}
where the duality pairing is from now on denoted in integral form for notational convenience.
	We refer to Proposition 6.6 in \cite{lema:02} for a proof.
	
	\item Young inequality: for all $\vartheta \in \R$, $(\ell,m)\in [1,+\infty]^2$, for any $\delta\in \R$ and for $(\ell_1,\ell_2)\in [1,\infty]^2$ and $(m_1,m_2)\in (0,\infty]^2$ such that 
	$$1+\frac{1}{\ell} = \frac{1}{\ell_1} + \frac{1}{\ell_2} \qquad  \text{and}\qquad  \frac{1}{m}\leq \frac{1}{m_1}+\frac{1}{m_2},$$
	there exists $C$ such that, for $f\in \B_{\ell_1,m_1}^{\vartheta - \delta}$ and $g\in \B_{\ell_2,m_2}^\delta$,
	\begin{equation}
		\label{young-besov} \Vert f\star g \Vert_{\B_{\ell,m}^\vartheta} \leq C \Vert f \Vert_{\B_{\ell_1,m_1}^{\vartheta - \delta}} \Vert g \Vert_{\B_{\ell_2,m_2}^\delta}.
	\end{equation}
	See Theorem 2.2 in \cite{KS21} for a proof (or \cite{Sawano18}).
\end{itemize}

{\color{black}For $r\in[1,\infty]$ we denote by $L^r ([0,T], \B^\vartheta_{\ell,m} (\R^d))$ the space of  $L^r$-functions of time with values in $\B^\vartheta_{\ell,m} (\R^d)$. The respective norm is denoted by $\| \cdot \|_{L^r-\B^\vartheta_{\ell,m}}$, that is, for $f\in L^r ([0,T], \B^\vartheta_{\ell,m} (\R^d))$ then 
\[
\| f \|_{L^r-\B^\vartheta_{\ell,m}} := \left( \int_0^T \| f(t) \|^r_{\B^\vartheta_{\ell,m}} \d r \right)^{\frac1r}.
\] }

\subsection{Controls for the mollified drift}

Let us now state some important properties of the chosen approximate drift \textcolor{black}{$ \mathfrak b_h$ defined in \eqref{DEF_BH_SCHEME}}.
\begin{lemma}[Useful bounds for $\mathfrak{b}_h$] \label{lemma-regularity-mollified-b} There exists $C\geq 1$ s.t. for all $h>0$ and all $(s,z)\in [0,T] \times \R^d$, $s\neq \tau_s^h $,
	\begin{itemize}
		\item Pointwise control {\color{black}(uniform in $z$)}
		\begin{equation}\label{CTR_PONCTUEL_BH}
			|\mathfrak b_h(s,z)|\le C (s-\tau_s^h)^{-\frac d{\alpha p}+\frac\beta\alpha}\|b(s,\cdot)\|_{\B_{p,q}^\beta}.
		\end{equation}
		\item Time-integrated pointwise control {\color{black}(uniform in $z$)}
		\begin{equation}\label{CTR_PONCTUEL_BH_INT}
			\left|\int_{\tau_s^h}^s\mathfrak b_h(u,z)\d u \right|\le C (s-\tau_s^h)^{\frac{\gamma}{\alpha}+\frac{1-\beta}{\alpha}}\|b\|_{L^r-\B_{p,q}^\beta}.
		\end{equation} 
		\item Spatial H\"older modulus of the integrated drift\\
		For all $(z,z')\in (\R^d)^2 $, $\zeta\in [-\beta,\alpha-1+\beta-d/p - \alpha/r) $,
		\begin{equation}\label{CTR_PONCTUEL_BH_INT_HOLDER} 
			\left|\int_{\tau_s^h}^s\Big(\mathfrak b_h(u,z)-\mathfrak b_h(u,z')\Big)\d u \right|\le  C|z-z'|^\zeta h^{\frac{\gamma}{\alpha}+\frac{1-\beta-\zeta}{\alpha}}\|b\|_{L^r-\B_{p,q}^\beta}.
		\end{equation}
		\item Besov norm of the mollified drift
		\begin{equation}
			\label{CTR_BESOV_BH}
			\|\mathfrak b_h(s,\cdot)\|_{\B_{p,q}^\beta}\le \|b(s,\cdot)\|_{\B_{p,q}^\beta}.
		\end{equation}
	\end{itemize}
\end{lemma}
\begin{proof}
	Let us recall the definition of $\mathfrak{b}_h$:
	\begin{equation}
		\mathfrak{b}_h(s,{\color{black}z}) := P^\alpha_{s-\tau_s^h}b(s,{\color{black}z}) = {\color{black}(} p_\alpha (s-\tau_s^h,\cdot)\star b(s,\cdot){\color{black}) (z)}.
	\end{equation}
{\color{black}Using the embedding of $L^\infty$ in $\B^0_{\infty, \textcolor{black}{1}}$, see  \eqref{EMBEDDING},
then \eqref{young-besov} and \eqref{besov-norm-stable-kernel}
we have,
\begin{align*}
|\mathfrak b_h(s,z)| = & | (p_\alpha (s-\tau_s^h,\cdot)\star b(s,\cdot))(z)| 
\leq \| p_\alpha (s-\tau_s^h,\cdot)\star b(s,\cdot)\|_{\B^0_{\infty, \textcolor{black}{1}}}\\
& \leq \| p_\alpha (s-\tau_s^h)\|_{\B^{-\beta}_{p',q' }} \|b(s)\|_{\B^{\beta}_{p,q }}
 \leq  (s-\tau_s^h)^{\frac\beta\alpha - \frac d {\alpha p} } \|b(s)\|_{\B^{\beta}_{p,q }},
\end{align*}
which is \eqref{CTR_PONCTUEL_BH} as wanted. 
}	
 Using now \eqref{CTR_PONCTUEL_BH} for the time integral and the H\"older inequality we obtain 
	\begin{align*}
		\left|\int_{\tau_s^h}^s\mathfrak b_h(u,z)\d u\right|& 
		\le \int_{\tau_s^h}^{s} (u-\tau_s^h)^{\frac \beta\alpha-\frac{d}{p \alpha}}\|b(u,\cdot)\|_{\B_{p,q}^\beta}\d u 
		\le \|b\|_{L^r-\B_{p,q}^\beta} \Big(\int_{\tau_s^h}^{s} \frac{\d u}{(u-\tau_s^h)^{r'(\frac{-\beta}\alpha+\frac{d}{p\alpha})}}\Big)^{\frac 1{r'}}\\
		&\le C\|b\|_{L^r-\B_{p,q}^\beta}(s-\tau_s^h)^{1-\frac 1r-\frac{d}{\alpha p}+\frac{\beta}\alpha}.\notag
	\end{align*}
	This proves \eqref{CTR_PONCTUEL_BH_INT}. Similarly, for $\zeta\in [-\beta,\alpha-1+\beta-d/p - \alpha/r) $, 
	using again 
	the Young inequality \eqref{young-besov} \textcolor{black}{and \eqref{besov-norm-stable-kernel}}, we have
	\begin{align*}
		&\left|\int_{\tau_s^h}^s\big(\mathfrak b_h(u,z)-\mathfrak b_h(u,z')\big)\d u\right| \\
		\le &|z-z'|^\zeta  \int_{\tau_s^h}^s\big\|\mathfrak b_h(u,\cdot)\big\|_{\B_{\infty,\infty}^\zeta}\d u \\
		\le& |z-z'|^\zeta\int_{\tau_s^h}^{s} \|p_{\alpha}(u-\tau_s^h,\cdot)\|_{\B_{p',q'}^{\zeta-\beta}}\|b(u,\cdot)\|_{\B_{p,q}^\beta}\d u\le |z-z'|^\zeta \int_{\tau_s^h}^{s} (u-\tau_s^h)^{\frac \beta\alpha-\frac \zeta\alpha-\frac{d}{p \alpha}}\|b(u,\cdot)\|_{\B_{p,q}^\beta}\d u\\
		\le & |z-z'|^\zeta\|b\|_{L^r-\B_{p,q}^\beta} \Big(\int_{\tau_s^h}^{s} \frac{\d u}{(u-\tau_s^h)^{r'(\frac{-\beta}\alpha+\frac{\zeta}\alpha+\frac{d}{p\alpha})}}\Big)^{\frac 1{r'}}\le C|z-z'|^\zeta\|b\|_{L^r-\B_{p,q}^\beta}(s-\tau_s^h)^{1-\frac 1r-\frac{d}{\alpha p}+\frac{\beta}\alpha-\frac \zeta\alpha}\\
		\le &\textcolor{black}{C|z-z'|^\zeta\|b\|_{L^r-\B_{p,q}^\beta}h^{\frac\gamma\alpha+\frac{1-\beta- \zeta}\alpha}}.
	\end{align*}
	This proves \eqref{CTR_PONCTUEL_BH_INT_HOLDER}. Eventually, \eqref{CTR_BESOV_BH} follows from the Young inequality \eqref{young-besov} and the fact that the $\B_{1,1}^0$ norm of the stable kernel is uniformly bounded.
\end{proof}

Let us also state {\color{black}and prove} the following lemma, which indicates that the deviation induced by the mollified drift over a single time step can be neglected at the scale of the noise:
\begin{lemma}[The approximate singular drift in the density of the driving noise]
	\label{Lemma_TRANS_SCHEME}
	There exists $C$ s.t.\ for $0\le s<t\le T $ s.t. $t-v\ge s-\tau_s^h $ with $v\in \{s, \tau_s^h\}$ and $(z,z')\in (\R^d)^2 $, and for all $k,\ |k|\le 2 $ {\color{black} we have}
	\begin{align}\label{DRIFT_TO_NEGLECT}
		\left|\nabla_z^k p_\alpha\left(t-v,z-\int_{\tau_s^h}^s \mathfrak b_h(u,z')\d u\right)\right|\le\frac{ C}{(t-v)^{\frac{|k|}\alpha}}\bar p_\alpha(t-v,z).
	\end{align}
\end{lemma}
{
The argument $v=\tau_s^h$ will naturally appear after conditioning w.r.t $X_{\tau_s^h}^h $ in the Duhamel expansion of the density \eqref{duhamel-scheme} of the scheme in Proposition \ref{prop-main-estimates-D}. This will be the most frequent use of the lemma.
On the other hand, the argument  $v=s$ is e.g. used when handling contributions associated with the first time step (see eq. \eqref{delta2-supnorm}  below).}
\begin{proof}
	Write from \eqref{derivatives-palpha},
	\begin{align*}
		\left|\nabla_z^k p_\alpha\left(t-v,z-\int_{\tau_s^h}^s \mathfrak b_h(u,z')\d u\right)\right|\le&\frac{ C}{(t-v)^{\frac{|k|}\alpha}}\bar p_\alpha\left(t-v,z-\int_{\tau_s^h}^s \mathfrak b_h(u,z')\d u\right)\\
		\le& \frac{C}{(t-v)^{\frac{|k|}\alpha+\frac{d}{\alpha}}}\frac{1}{\left(2-\frac{|\int_{\tau_s^h}^s \mathfrak b_h(u,z')\d u|}{(t-v)^{\frac 1\alpha}}+\frac{|z|}{(t-v)^{\frac 1\alpha}}\right)^{d+\alpha}}\\
		\le &\frac{C}{(t-v)^{\frac{|k|}\alpha+\frac{d}{\alpha}}}\frac{1}{\left(2-\frac{(s-\tau_s^h)^{\frac1\alpha+\frac{\gamma-\beta}\alpha}}{(t-v)^{\frac 1\alpha}}+\frac{|z|}{(t-v)^{\frac 1\alpha}}\right)^{d+\alpha}}\\
		\le & \frac{C}{(t-v)^{\frac{|k|}\alpha+\frac{d}{\alpha}}}\frac{1}{\left(1+\frac{|z|}{(t-v)^{\frac 1\alpha}}\right)^{d+\alpha}}\le \frac{C}{(t-v)^{\frac{|k|}\alpha}}\bar p_\alpha(t-v,z), 
	\end{align*}
	for $h$ sufficiently small, using \eqref{CTR_PONCTUEL_BH_INT} for the last but one inequality and up to a modification of $C$ from line to line. This proves \eqref{DRIFT_TO_NEGLECT}.
\end{proof}

\section{Tools for the proof of Theorem \ref{thm-besov}}\label{TOOLS_HK_AND_MISC_CONV}
{\color{black}In this section we collect some tools and technical lemmata that are useful to prove the main result.}

\subsection{Duhamel expansion of the densities and heat kernel-like estimates} \label{SSC:duhamel}
{\color{black}Let us introduce some notation first. Given a stochastic process $(Y_r)$ in $\R^d$, a measurable  function $f:[0,T]\times\R^d\to\R^{d'}$ (for some $d,d'$), $0\leq s\leq r$ and $x\in\R^d$, we denote the conditional expectation knowing $Y_s=x$ by $\E_{s,x}\left[f(r,Y_r)\right]:= \E\left[f(r,Y_r) \vert Y_s=x\right] $.}

        \begin{propo}[Duhamel representations for the densities of the SDE and the Euler scheme]\label{prop-main-estimates-D}
		The density $\Gamma(s,x,t,\cdot) $ of the unique weak solution to Equation \eqref{sde} starting from $x$ at time $s\in [0,T)$ admits \textcolor{black}{the} following  Duhamel representation: for all $\textcolor{black}{t\in (s, T],\  y\in \R^d }$,
		\begin{align}
			\Gamma(s,x,t,y)
			= p_\alpha(t-s,y-x)-\int_{s}^{ t}\E_{s,x}\left[b(r,X_r)\cdot\nabla_y  p_\alpha(t-r,y-X_r)\right]\d r.\label{duhamel-Diff}
		\end{align} 
Similarly, for $k\in \llbracket 0,n-1\rrbracket ,\ t\in (t_k,T] $, the density of $X_t^h$ admits, conditionally to $X_{t_k}^h=x$, a transition density $\Gamma^h(t_k,x,t,\cdot) $, which enjoys a Duhamel type representation: for all $y\in \R^d $,
		\begin{align}
			\Gamma^h(t_k,x,t,y)
			= p_\alpha(t-t_k,y-x)-\int_{t_k}^{ t}\E_{t_k,x}\left[\mathfrak b_h(r,X^h_{\tau_r^h})\cdot\nabla_y  p_\alpha(t-r,y-X^h_r)\right]\d r.\label{duhamel-scheme}
		\end{align} 
	\end{propo}
\begin{proof}
For \eqref{duhamel-Diff}, this is a consequence of the mollification procedure considered in \cite{Fit23} (see Section 4 therein).

{\color{black}As for \eqref{duhamel-scheme}, first of all we prove using} \eqref{CTR_PONCTUEL_BH}, that, for any $t>0$, the   {\color{black}scheme \eqref{euler-scheme-besov}}
admits a density which we denote $ \Gamma^h$. Indeed, the scheme can be viewed as the \textcolor{black}{Euler discretization in space} associated with the solution to the SDE
\begin{equation}
	\d X_t^{D,h}=\mathfrak{b}_h(t,X_t^{D,h})\d t + \d Z_t,
\end{equation}
which has $L^r-L^\infty$ drift, albeit with $L^r-L^\infty$ norm depending on $h$.
As a consequence of \cite{FJM24} \textcolor{black}{and \eqref{EXPR_DRIFT}}, it follows that $\Gamma^h$ enjoys the Duhamel-type representation \eqref{duhamel-scheme} at the discretization times. For the sake of completeness we provide a proof for \eqref{duhamel-scheme} in Section \ref{PROOF_FOR_DENS_SCHEME}. 
\end{proof}

 We are now ready to state some useful controls for the densities $\Gamma$ and $\Gamma^h$.}
\begin{propo}[Heat kernel estimates for the densities]\label{prop-HK}\label{THE_PROP}\phantom{BOUH}
	\begin{itemize}
		\item Heat kernel bound for the density of the Euler scheme: for all $\rho \in (-\beta,\gamma-\beta)$ there exists $C\textcolor{black}{=C(d,b,\alpha)}$ such that for all $(x,y,y')\in (\R^d)^3$, $t>0$, \textcolor{black}{$|y-y'|\lesssim t^{\frac 1\alpha}$},
		\begin{align}
		\Gamma^h (0,x,t,y) &\leq C \bar p_{\alpha}(t,y-x),\label{aronson-gammah}\\
		|\Gamma^h (0,x,t,y')-\Gamma^h (0,x,t,y)|&\le C\frac{|y-y'|^\rho}{t^{\frac \rho\alpha}} \bar p_\alpha(t,y-x).\label{holder-forward-gammah}
		\end{align}
		{Actually, these bounds follow from the proof of the following inequality:}
		\begin{equation}\label{ineq-density-scheme}
						\left\Vert \frac{\Gamma^h (0,x,t,\cdot)}{\bar{p}_{\alpha}(t,\cdot-x)} \right\Vert_{\B_{\infty,\infty}^\rho} \leq C(1+t^{-\frac \rho\alpha}).
		\end{equation}
		\item Heat kernel and Sensitivity bounds for the density of the SDE: for all $\rho \in (-\beta,\gamma-\beta)$, there exists $C\textcolor{black}{=C(d,b,\alpha)}$ such that for all $(x,y,y')\in (\R^d)^3$, $t\in (0,T]$, \textcolor{black}{$|y-y'|\lesssim t^{\frac 1\alpha}$},
				\begin{align}
					\Gamma (0,x,t,y) &\leq C \bar p_\alpha(t,y-x),\label{aronson-gamma}\\
				|\Gamma (0,x,t,y')-\Gamma (0,x,t,y)|&\le C\frac{|y-y'|^\rho}{t^{\frac \rho\alpha}}\bar p_\alpha(t,y-x).\label{holder-forward-gamma}
			\end{align}
Consequently, in terms of Besov spaces,
		\begin{equation}\label{ineq-density-diff}
			\left\Vert \frac{\Gamma (0,x,t,\cdot)}{\bar{p}_{\alpha}(t,\cdot-x)} \right\Vert_{\B_{\infty,\infty}^\rho} \leq C(1+t^{-\frac \rho\alpha}).
		\end{equation}
Moreover, it holds that \textcolor{black}{for all $\eps>0$ meant to be small}
, $t'\in (t,T]$ such that  $|t-t'|\le t/2$, 
		\begin{equation}\label{holder-time-gamma}
			\left\Vert \frac{\Gamma (0,x,t,\cdot)-\Gamma (0,x,t',\cdot)}{\bar{p}_{\alpha}(t',\cdot-x)} \right\Vert_{\B_{\infty,\infty}^\rho} \leq C \frac{(t'-t)^\frac{\gamma-\eps}{\alpha}}{t^\frac{\gamma-\eps+\rho}{\alpha}} .
		\end{equation}
	\end{itemize}
\end{propo}
The bound \eqref{ineq-density-diff} was obtained in \cite{Fit23}, \eqref{holder-time-gamma} would follow from the same lines but it is detailed for {\color{black}completeness} in Appendix \ref{HK_TIME_APP}. The corresponding bound \eqref{ineq-density-scheme} for the Euler scheme is also proved in Appendix \ref{PROOF_FOR_DENS_SCHEME}.
\textcolor{black}{Actually, \eqref{aronson-gammah} and \eqref{holder-forward-gammah} are a byproduct of the arguments used to prove \eqref{ineq-density-scheme}. These points are discussed as well in Appendix \ref{PROOF_FOR_DENS_SCHEME}.}

\subsection{\color{black}Auxiliary estimates for the density of stable processes}

The following estimates will be needed for the error analysis below.

	\begin{lemma}[Besov estimates for ${\bar p}_{\alpha}$]\label{lemma-besov-estimates-palpha} Let $\beta<0$.
		\begin{itemize}
			\item $\forall 0< s  < t$, $\forall (x,y)\in (\R^d)^2$, $\forall \zeta \in (-\beta,1]$, $\forall k \in \{ 0,1\}$,
			\begin{align}\label{besov-estimate-stable}
				\Vert&\bar {p}_{\alpha} (s,x-\cdot) \nabla^k_y p_{\alpha} (t-s,y-\cdot) \Vert_{\B^{-\beta}_{p',q'}} \lesssim 
				\frac{\bar{p}_{\alpha} (t,x-y)}{(t-s)^{\frac{k}{\alpha}}} t^{\frac{\beta}{\alpha}}\left[ \frac{1}{s^{\frac{d }{\alpha  p}}}+\frac{1}{(t-s)^{\frac{d }{\alpha  p}}} \right] \left[
				\frac{t^{\frac{\zeta}{\alpha}}}{s^{\frac{\zeta }{\alpha}}}+\frac{t^{\frac{\zeta}{\alpha}}}{(t-s)^{\frac{\zeta }{\alpha}}}  \right] .
			\end{align}
			\item $\forall 0< s  < t$, $\forall (x,y)\in (\R^d)^2$, $\textcolor{black}{j\in \{0,1-\textcolor{black}{\frac \varepsilon \gamma}\}}$, $\forall \zeta \in (-\beta+\textcolor{black}{j\gamma},\textcolor{black}{-\beta+\gamma})$\footnote{\textcolor{black}{Pay attention that the constraint on $\zeta $ here follows from the available regularity of the density $\Gamma $ in its forward spatial variable. The threshold 1 is reachable for the $\bar p_\alpha $. This is not the case when the density $\Gamma $ is involved (see \eqref{ineq-density-diff}). 
			}
			}, $\forall k \in \{ 0,1\}$,
			\begin{align}\label{besov-estimate-gamma-and-stable}
				\Vert& \Gamma(0,x,s,\cdot) \nabla^k_y p_{\alpha} (t-s,y-\cdot) \Vert_{\B^{-\beta+j\gamma}_{p',q'}} \lesssim 
				\frac{\bar{p}_{\alpha} (t,x-y)}{(t-s)^{\frac{k}{\alpha}}} t^{\frac{\beta\textcolor{black}{-j\gamma}}{\alpha}}\left[ \frac{1}{s^{\frac{d }{\alpha  p}}}+\frac{1}{(t-s)^{\frac{d }{\alpha  p}}} \right] \left[
				  \frac{t^{\frac{\zeta}{\alpha}}}{s^{\frac{\zeta }{\alpha}}}+\frac{t^{\frac{\zeta}{\alpha}}}{(t-s)^{\frac{\zeta }{\alpha}}}  \right] .
			\end{align}
			
				\item  $\forall 0< s  < t$, $\forall (x,y,w)\in (\R^d)^3$, $\forall \zeta \in (-\beta,1]$,
	\begin{align}\label{big-lemma-2}
		&\left\Vert \bar{p}_{\alpha} (s,x-\cdot)\left[\frac{\nabla p_\alpha (t-s,w-\cdot)}{\bar{p}_{\alpha} (t,w-x)} -\frac{\nabla p_\alpha (t-s,y-\cdot)}{\bar{p}_{\alpha} (t,y-x)} \right] \right\Vert_{ \B_{p',q'}^{-\beta}} \nonumber \\& \qquad\lesssim \frac{|w-y|^\zeta }{(t-s)^{\frac{\zeta+1}{\alpha}}} t^{\frac{\beta}{\alpha}} \left[\frac{1}{s^{\frac{d }{\alpha  p}}}+ \frac{1}{(t-s)^{\frac{d }{\alpha  p}}} \right] \left[
		 \frac{t^{\frac{\zeta}{\alpha}}}{s^{\frac{\zeta }{\alpha}}}+\frac{t^{\frac{\zeta}{\alpha}}}{(t-s)^{\frac{\zeta }{\alpha}}}  \right] .
\end{align}
			\item $\forall 0< s  < t$, $\forall (x,y,y')\in (\R^d)^3$, s.t. $|y'-y|\le t^{\frac 1\alpha} $, $\textcolor{black}{j\in \{0,1-\textcolor{black}{\frac \varepsilon\gamma}\}}$, $\forall \zeta \in (-\beta+\textcolor{black}{j\gamma} ,\textcolor{black}{-\beta+\gamma})\footnote{Here as well the restriction on $\zeta $ comes from the regularity of $\Gamma $ in its forward spatial variable.},\rho \in (-\beta+\textcolor{black}{j\gamma},1]$, $\forall k \in \{ 0,1\}$, 
			\begin{align}\label{besov-estimate-gamma-and-stable_NO_NORM_Holder}
			&\Vert \Gamma(0,x,s,\cdot) (\nabla^k_y p_{\alpha} (t-s,y-\cdot)-\nabla^k_y p_{\alpha} (t-s,y'-\cdot)) \Vert_{\B^{-\beta\textcolor{black}{+j\gamma}}_{p',q'}} \notag\\
			\lesssim& \bar{p}_{\alpha} (t,x-y)\frac{|y-y'|^\rho}{(t-s)^{\frac k\alpha+\frac\rho\alpha}} t^{\frac{\beta\textcolor{black}{-j\gamma}}{\alpha}}\left[ \frac{1}{s^{\frac{d }{\alpha  p}}}+\frac{1}{(t-s)^{\frac{d }{\alpha  p}}} \right] \left[
			\frac{t^{\frac{\zeta}{\alpha}}}{s^{\frac{\zeta }{\alpha}}}+\frac{t^{\frac{\zeta}{\alpha}}}{(t-s)^{\frac{\zeta }{\alpha}}}  \right] .
			\end{align}			
			
			\item  $ \forall h\leq s \leq \tau_t^h-h$, $r\in \textcolor{black}{[}\tau_s^h,\tau_s^h+h)$, $\forall (x,y)\in (\R^d)^2$, $\forall \zeta \in (-\beta,\textcolor{black}{-\beta+\gamma)}$, $\forall \delta\in [0,1) $, 
			\begin{align}\label{besov-estimate-gammah-stable-sensi-holder-time}
				&\left\Vert \Gamma^h(0,x,\tau_s^h,\cdot)\left[\nabla_y p_{\alpha} (t-s,y-\cdot) -\nabla_y p_{\alpha} (t-r,y-\cdot) \right] \right\Vert_{ \B_{p',q'}^{-\beta}} \nonumber \\& \qquad\lesssim (s-r)^\delta\frac{\bar p_{\alpha}(t,y-x) }{(t-s)^{\frac{1}{\alpha}+\delta}} t^{\frac{\beta}{\alpha}}\left[\frac{1}{s^{\frac{d }{\alpha  p}}}+ \frac{1}{(t-s)^{\frac{d }{\alpha  p}}} \right] \left[
				\frac{t^{\frac{\zeta}{\alpha}}}{s^{\frac{\zeta }{\alpha}}}+\frac{t^{\frac{\zeta}{\alpha}}}{(t-s)^{\frac{\zeta }{\alpha}}}  \right] .
			\end{align}
			
			\item  $ \forall h\leq s \leq \tau_t^h-h$, $r\in \textcolor{black}{[}\tau_s^h,\tau_s^h+h)$, $\forall (x,y,y')\in (\R^d)^3$, $\forall \zeta,\rho \in (-\beta,1]$, $|y-y'|\le t^{\frac 1\alpha} $, 
			$\forall \theta \in\{0,1\} $,
			\begin{align}\label{besov-estimate-stable-derivees_temps_esp_sensi_holder_esp}
				&\left\Vert \bar {p}_{{\alpha}} (\tau_s^h,x-\cdot)\left[\partial_t^\theta\nabla_y p_{\alpha} (t-r,y-\cdot) -\partial_t^\theta\nabla_y p_{\alpha} (t-r,y'-\cdot) \right] \right\Vert_{ \B_{p',q'}^{-\beta}} \nonumber \\& \qquad\lesssim |y-y'|^\rho \frac{\bar p_{\alpha}(t,y-x) }{(t-s)^{\frac{1}{\alpha}+\theta+\frac{\rho}{\alpha}}} t^{\frac{\beta}{\alpha}}\left[\frac{1}{s^{\frac{d }{\alpha  p}}}+ \frac{1}{(t-s)^{\frac{d }{\alpha  p}}} \right] \left[
				 \frac{t^{\frac{\zeta}{\alpha}}}{s^{\frac{\zeta }{\alpha}}}+\frac{t^{\frac{\zeta}{\alpha}}}{(t-s)^{\frac{\zeta }{\alpha}}}  \right] .
			\end{align}
			
			\item  $ \forall h\leq s \leq \tau_t^h-h$, $r\in (\tau_s^h,\tau_s^h+h)$, $\forall (x,y,y')\in (\R^d)^3$, $\forall \zeta \in \textcolor{black}{(-\beta,-\beta+\gamma)},\rho \in (-\beta,1]$, $|y-y'|\le t^{\frac 1\alpha} $, $\forall \theta \in\{0,1\} $, \\
			\begin{align}\label{besov-estimate-gammah-stable-derivees_temps_esp_sensi_holder_esp}
				&\left\Vert \Gamma^h(0,x,\tau_s^h,\cdot)\left[\partial_t^\theta\nabla_y p_{\alpha} (t-\sm{r} ,y-\cdot) -\partial_t^\theta\nabla_y p_{\alpha} (t-\textcolor{black}{r},y'-\cdot) \right] \right\Vert_{ \B_{p',q'}^{-\beta}} \nonumber \\& \qquad\lesssim |y-y'|^\rho \frac{\bar p_{\alpha}(t,y-x) }{(t-s)^{\frac{1}{\alpha}+\theta+\frac{\rho}{\alpha}}} t^{\frac{\beta}{\alpha}}\left[\frac{1}{s^{\frac{d }{\alpha  p}}}+ \frac{1}{(t-s)^{\frac{d }{\alpha  p}}} \right] \left[
				 \frac{t^{\frac{\zeta}{\alpha}}}{s^{\frac{\zeta }{\alpha}}}+\frac{t^{\frac{\zeta}{\alpha}}}{(t-s)^{\frac{\zeta }{\alpha}}}  \right] .
			\end{align}
			
			\item $\forall h\leq s \leq \tau_t^h-h$,  $\forall (x,y)\in (\R^d)^2$, $\forall \zeta \in (-\beta,\textcolor{black}{-\beta+\gamma})$,
			\begin{align}\label{besov-estimate-gammah-stable-PERTURB_DRIFT}
				&\left\| \Gamma^h(0,x,\tau_s^h,\cdot)\Big(\nabla_y p_\alpha(t-\tau_s^h,y-\cdot) -\nabla_y  p_\alpha (t-\tau_s^h,y-(\cdot+ \int_{\tau_s^h}^{s}\mathfrak b_h(u,\cdot)\d u)) \Big)\right \|_{\B_{p',q'}^{-\beta}}\notag\\
				& \qquad\lesssim  \bar p_{\alpha}(t,y-x)  \frac{h^{\frac{\gamma-\beta}{\alpha}}   
				}{(t-\tau_s^h)^{\frac 1 \alpha}}t^{\frac{\beta}{\alpha}}\left[\frac{1}{(\tau_s^h)^{\frac{d}{\alpha p}}} + \frac{1}{(t-\tau_s^h)^{\frac{d}{\alpha p}}}\right]\left[
				\frac{t^{\frac \zeta\alpha}}{(\tau_s^h)^\frac\zeta\alpha}+\frac{t^{\frac \zeta\alpha}}{(t-\tau_s^h)^\frac\zeta\alpha} \right].
			\end{align}
			\item $ \forall h\leq s \leq \tau_t^h-h$,  $\forall (x,y,y')\in (\R^d)^3,\ |y-y'|\le t^{\frac 1\alpha}$, $\forall \zeta\in \textcolor{black}{(-\beta,-\beta+\gamma)},\rho \in (-\beta,1], \forall \lambda\in [0,1]$,
			\begin{align}\label{besov-estimate-gammah-stable-PERTURB_DRIFT_SENSI_HOLDER}
				&\left\| \Gamma^h(0,x,\tau_s^h,\cdot)\Big(\nabla_y^2  p_\alpha (t-\tau_s^h,y-(\cdot+ \lambda \int_{\tau_s^h}^{s}\mathfrak b_h(u,\cdot)\d u))- \nabla_y^2  p_\alpha (t-\tau_s^h,y'-(\cdot+ \lambda \int_{\tau_s^h}^{s}\mathfrak b_h(u,\cdot)\d u))\Big)\right.\notag\\
				&\qquad\left.\times\int_{\tau_s^h}^{s}\mathfrak b_h(u,\cdot)\d u \right \|_{\B_{p',q'}^{-\beta}}\notag\\
				& \qquad\lesssim  |y-y'|^\rho \bar p_{\alpha}(t,y-x)  \frac{h^{\frac{\gamma-\beta}{\alpha}}    
				} {(t-\tau_s^h)^{\frac 1 \alpha+\frac\rho\alpha}}t^{\frac{\beta}{\alpha}}\left[\frac{1}{(\tau_s^h)^{\frac{d}{\alpha p}}} + \frac{1}{(t-\tau_s^h)^{\frac{d}{\alpha p}}}\right]\left[ \frac{t^{\frac \zeta\alpha}}{(\tau_s^h)^\frac\zeta\alpha}+\frac{t^{\frac \zeta\alpha}}{(t-\tau_s^h)^\frac\zeta\alpha} \right].
			\end{align}
	\end{itemize}
	\end{lemma}
\textcolor{black}{The bounds \eqref{besov-estimate-stable}, \eqref{big-lemma-2} 
can be found in Lemma 3 of \cite{Fit23}.
For the sake of completeness and in order to be self contained we provide a proof of \eqref{big-lemma-2}  below in Section \ref{PROOF_BIG_LEMME_2}}.
Equation \eqref{besov-estimate-gamma-and-stable} relies on the same proof as \eqref{big-lemma-2}. The other estimates are proved in Section \ref{subsec-lemma-besov}, and the approach therein would also readily give the previously mentioned bounds.\\	

\sm{\begin{remark}[About the time arguments in some of the previous inequalities]\label{REM_WITH_DISCR}\phantom{GRR}\\
Let us mention that, when $s\in [h,t) $, then the controls \eqref{besov-estimate-stable} and \eqref{besov-estimate-gamma-and-stable} still hold with the respective l.h.s. replaced by
$ \Vert\bar {p}_{\alpha} (\tau_s^h,x-\cdot) \nabla^k_y p_{\alpha} (t-s,y-\cdot) \Vert_{\B^{-\beta}_{p',q'}}$ and $\Vert\Gamma(0,x,\tau_s^h,\cdot) \nabla^k_y p_{\alpha} (t-s,y-\cdot) \Vert_{\B^{-\beta+j\gamma}_{p',q'}} $. This will be thoroughly used in the proofs below and is a simple consequence of the fact that,
 on the considered time interval, $s\asymp \tau_s^h$ 
    and the very same proof would lead to a r.h.s. with $\bar p_\alpha(t-s+\tau_s^h,y-x)\le \bar p_\alpha(t,y-x) $. Similarly, the resulting time singularities w.r.t. the first time argument would be in $(\tau_s^h)^{-\vartheta},\ \vartheta>0 $, which can as well be bounded by $s^{-\vartheta} $ again because of the equivalence.   
\end{remark}
}

\subsection{Singular integrals}
{\color{black}
Let us recall the definition of the Beta function, that is $B(1-\aa,1-\bb):= \int_0^1 s^{-\aa} (1-s)^{-\bb} \d s$, which converges  if $\aa<1, \bb<1$. By a change of variables it is easy to see that for $0\leq u<t$ we have
\begin{equation}\label{eq:Beta}
\int_u^t \frac{1}{(s-u)^{\aa}}\frac{1}{(t-s)^{\bb}} \d s =\int_0^1 \frac{1}{[(t-u)s']^\aa} \frac{1}{[(t-u)(1-s')]^\bb} (t-u) \d s' = (t-u)^{1-\aa-\bb} B(1-\aa, 1-\bb).
\end{equation}
We now use this to prove a lemma on three singular integrals, quantities which  are recurring in the computations below. 

\begin{lemma}\label{lm:sing_int}
Let $0<v<t$, $r'\geq 1$ and $\aa,\bb,\cc,\dd \in \mathbb R$. Let $r$ be the conjugate of $r'$, that is $\frac1r+ \frac1{r'}=1$.
\begin{itemize}
\item  For  $r'(\aa+\cc+\dd)<1$ and $r'(\bb+\cc+\dd)<1$ we have
\begin{equation}\label{eq:Beta2-5}
\left (\int_0^t  \frac{1}{s^{\aa r'}}\frac{1}{(t-s)^{\bb r'}} \left [\frac{1}{s^{\cc}} + \frac{1}{(t-s)^{\cc}}  \right]^{r'}  \left [\frac{1}{s^{\dd}} + \frac{1}{(t-s)^{\dd}}  \right]^{r'} \d s \right)^{\frac1{r'}}
\lesssim t^{1-\frac1{r} -\aa-\bb-\cc-\dd}.
\end{equation} 

\item For $r'(\aa+\bb)<1$  we have
\begin{equation}\label{eq:Beta1}
\left (\int_0^v  \frac{1}{s^{\aa r'}}  \left [\frac{1}{s^{\bb}} + \frac{1}{t^{\bb}}  \right]^{r'} \d s \right)^{\frac1{r'}} 
\lesssim v^{1-\frac1{r}-\aa-\bb} + t^{-\bb } v^{1-\frac1{r} -\aa},
\end{equation} 
and 
\begin{equation}\label{eq:Beta6}
\left (\int_v^t  \frac{1}{(t-s)^{\aa r'}}  \left [\frac{1}{t^{\bb}} + \frac{1}{(t-s)^{\bb}}  \right]^{r'} \d s \right)^{\frac1{r'}} 
\lesssim (t-v)^{1-\frac1{r}-\aa-\bb} + t^{-\bb } (t-v)^{1-\frac1{r} -\aa}.
\end{equation} 
\end{itemize}
\end{lemma}  
\begin{proof}
The proofs of \eqref{eq:Beta1} and \eqref{eq:Beta6} follow by using the simple inequality (for positive real numbers $x,y$) $(x+y)^{r'} \leq 2^{r'-1}(x^{r'}+ y^{r'})$ and then computing the integrals, which converge exactly under the conditions stated. 

We now prove \eqref{eq:Beta2-5}. Using the same inequality as before we get
\begin{align*}
&\left (\int_0^t   \frac{1}{s^{\aa r'}}\frac{1}{(t-s)^{\bb r'}} \left [\frac{1}{s^{\cc}} + \frac{1}{(t-s)^{\cc}}  \right]^{r'}  \left [\frac{1}{s^{\dd}} + \frac{1}{(t-s)^{\dd}}  \right]^{r'} \d s \right)^{\frac1{r'}}\\
& \qquad \lesssim
\left (\int_0^t  \frac{1}{s^{(\aa+\cc+\dd)r'}}\frac{1}{(t-s)^{\bb r'}}  \d s \right)^{\frac1{r'}}
+ \left (\int_0^t  \frac{1}{s^{(\aa+\cc)r'}}\frac{1}{(t-s)^{(\bb+\dd)r'}}  \d s \right)^{\frac1{r'}}\\ 
& \quad \qquad+ \left (\int_0^t  \frac{1}{s^{(\aa+\dd)r'}}\frac{1}{(t-s)^{(\bb+\cc)r'}}  \d s \right)^{\frac1{r'}}
+ \left (\int_0^t  \frac{1}{s^{\aa r'}}\frac{1}{(t-s)^{(\bb+\cc+\dd)r'}}  \d s \right)^{\frac1{r'}}\\
&  \qquad =
\left ( t^{1- (\aa+\cc+\dd)r' -\bb r'} B(1- (\aa+\cc+\dd)r'), 1-\bb r' ) \right)^{\frac1{r'}}\\
& \quad \qquad
+ \left ( t^{1- (\aa+\cc)r' -(\bb+\dd)r'} B(1- (\aa+\cc)r'), 1-(\bb+\dd)r' ) \right)^{\frac1{r'}}\\ 
& \quad \qquad
+ \left ( t^{1- (\aa+\dd)r' -(\bb+\cc)r'} B(1- (\aa+\dd)r'), 1-(\bb+\cc)r' ) \right)^{\frac1{r'}}\\
& \quad \qquad
+ \left ( t^{1- \aa r' -(\bb+\cc+\dd)r'} B(1- \aa r', 1-(\bb+\cc+\dd)r' ) \right)^{\frac1{r'}}\\
&  \qquad \lesssim  t^{\frac1{r'} -\aa-\bb-\cc-\dd},
\end{align*}
having used \eqref{eq:Beta} for the equality. Clearly the conditions $r'(\aa+\cc+\dd)<1$ and $r'(\bb+\cc+\dd)<1$ imply that all integrals converge. 
\end{proof}
}

\section{Proof of Theorem \ref{thm-besov}}\label{SEC_PROOF_THM}
{\color{black} 
This whole section is devoted to the proof of Theorem \ref{thm-besov} and the details are worked out on four subsections. Before heading straight to the technical calculations and estimates, we illustrate the overall strategy.
The convergence rate that we want to prove can be rewritten as 
\[
f(t,x, \cdot):=\frac{|\Gamma^h(0,x,t,\cdot)-\Gamma(0,x,t,\cdot)|}{\bar p_\alpha (t,\cdot-x)}\leq C h^{\frac{\gamma-\varepsilon}{\alpha}}.
\]
We start by decomposing in Section \ref{ssc:decomposition} the error $ \Gamma^h(0,x,t,y)-\Gamma(0,x,t,y)$  in \sm{several} terms using the Duhamel representation for the densities, see \eqref{DECOUP_ERR} and Figure 1. \sm{All the considered terms, but one, are associated with some time or space sensitivities, mollification error, or some small time interval. The remaining term is somehow associated with a Gronwall type procedure and will make an appropriate Besov norm of $f$ appear}. We then  look for bounds of the difference $ \Gamma^h(0,x,t,y)-\Gamma(0,x,t,y)$ that are uniform in $y$, and we do so in Section \ref{subsec-proof-main-thm-supnorm} by bounding each term in the decomposition separately. \sm{Due to product rules associated with a distribution with regularity index $\beta<0 $, the Gronwall type term, $\Delta_4$ in the expansion \eqref{DECOUP_ERR} below},  can \sm{naturally}
be bounded by a quantity depending on the Besov norm of $f(t,x,\cdot)$ in the space $\B_{\infty,\infty}^\rho$ for some $\rho\in(-\beta,1)$. Notice also that, for $\rho$ \sm{in the indicated range}, one has 
\begin{align}\label{eq:Gamma-Gammah}
\left \|f(t,x,\cdot)\right\|_{\B_{\infty,\infty}^\rho}
\asymp &\left \| f(t,x,\cdot) \right\|_{L^\infty}+\mathcal H_\rho\left(f(t,x,\cdot)\right),
\end{align}
where the  $\rho$-H\"older modulus $\mathcal H_\rho(\cdot)$ is defined by
\[
\mathcal H_\rho \left ( f(t, x, \cdot) \right)
	:=  \sup_{y\neq y'\in \R^d} \frac{\left| f(t, x,y) - f(t, x,y')\right|}{|y-y'|^\rho}.
\]
\sm{In order to be as close as possible to the threshold imposed by the product rules, which will in turn lead to the largest convergence rate, we set from now on}
\begin{equation}
\label{CHOICE_FOR_RHO}
\rho:=-\beta+\frac \varepsilon2,
\end{equation}
where $\varepsilon>0 $ can be chosen as small as desired. \sm{In order to close} the proof argument, we must  also control the H\"older modulus of $f(t,x,\cdot) $. To this end, write: 
\begin{align}
&\left|f(t,x,y)-f(t,x,y')\right|\notag\\
\sm{\le} &\left|\frac{(\Gamma-\Gamma^h)(0,t,x,y)-(\Gamma-\Gamma^h)(0,t,x,y')}{\bar p_\alpha(t,y-x)}\right|+\left|\Big[\frac{1}{\bar p_\alpha(t,y'-x)}-\frac{1}{\bar p_\alpha(t,y-x)} \Big](\Gamma-\Gamma^h)(0,t,x,y')\right|\notag\\
\lesssim &\left|\frac{(\Gamma-\Gamma^h)(0,t,x,y)-(\Gamma-\Gamma^h)(0,t,x,y')}{\bar p_\alpha(t,y-x)}\right|+\frac{|y-y'|^\rho}{t^{\frac\rho\alpha}}\frac{\left|(\Gamma-\Gamma^h)(0,t,x,y')\right|}{\bar p_\alpha(t,y'-x)}\notag\\
\lesssim &\left|\frac{(\Gamma-\Gamma^h)(0,t,x,y)-(\Gamma-\Gamma^h)(0,t,x,y')}{\bar p_\alpha(t,y-x)}\right|+\frac{|y-y'|^\rho}{t^{\frac\rho\alpha}}\left\|\frac{(\Gamma-\Gamma^h)(0,t,x,\cdot)}{\bar p_\alpha(t,\cdot-x)}\right\|_{L^\infty},
\label{THE_DECOUP_FOR_HOLDER_MOD_E}
\end{align}
having used Lemma \ref{lemma-stable-sensitivities} for the first inequality. The second term in  \eqref{THE_DECOUP_FOR_HOLDER_MOD_E} can be bounded using the $L^\infty$ bounds already discussed, while 
the first term in \eqref{THE_DECOUP_FOR_HOLDER_MOD_E}, or in fact of the difference $ (\Gamma-\Gamma^h)(0,t,x,y)-(\Gamma-\Gamma^h)(0,t,x,y')$,
is the content of Section \ref{subsec-proof-main-thm-holder-mod}.  
\sm{The previous r.h.s. \eqref{THE_DECOUP_FOR_HOLDER_MOD_E} exhibits a time singularity. This leads  us to define, from \eqref{eq:Gamma-Gammah}, the normalized} quantity 
\begin{align}\label{DEF_GH}
g_{h,\rho}(t):=\left \|\frac{(\Gamma-\Gamma^h)(0,x,t,\cdot)}{\bar p_\alpha(t,\cdot-x)}\right\|_{L^\infty}+t^{\frac\rho\alpha}\mathcal H_\rho\left(\frac{(\Gamma-\Gamma^h)(0,x,t,\cdot)}{\bar p_\alpha(t,\cdot-x)}\right).
\end{align}
Thanks to the heat kernel estimates \eqref{aronson-gamma}, \eqref{holder-forward-gamma} and \eqref{aronson-gammah}, \eqref{holder-forward-gammah}, we  know that $g_{h,\rho}(t)$ is finite. We mention that the additional normalization in $t^{\frac \rho\alpha} $ for the H\"older modulus is the natural one associated with the spatial $\rho $-H\"older modulus of continuity of the stable heat kernel. 

\vspace{10pt}
Thus, the overall strategy of the proof will be to decompose the error (Section \ref{ssc:decomposition}) and  control both its $L^\infty$-norm (Section \ref{subsec-proof-main-thm-supnorm}) and its H\"older modulus (Section \ref{subsec-proof-main-thm-holder-mod}) with a power of the discretization step $h$ plus something depending on $g_{h,\rho}(t)$. This in turn will allow us to  perform a circular argument (Section \ref{sec:circular}), 
{assuming w.l.o.g. that $T$ is small enough}, by plugging those bounds in \eqref{DEF_GH} and to conclude since $f(t,x,y)\lesssim g_{h, \rho}(t)$. 

{Once the result holds for a small enough time $T_0$, it is direct to extend it to an arbitrary time $t$ not necessarily small, by using the Chapman-Kolmogorov decomposition. Indeed, assuming w.l.o.g. that  $t=mT_0,\ m>1 $ (which is always possible up to a modification of $T_0$), and setting $ T_{0}^{(i)}=iT_0,\ i\in \{0,\cdots,m\}$, $z_0=x, z_m=y $, one writes (with the convention that $\prod_{i=j}^{j-1}=1,\ j\in \N$):
\begin{align*}
(\Gamma^h-\Gamma)(0,x,t,y)=&\int_{(\R^d)^{m-1}} \Big[\prod_{i=0}^{m-1}\Gamma^h(T_0^{(i)},z_i,T_{0}^{(i+1)},z_{i+1})-\prod_{i=0}^{m-1}\Gamma(T_0^{(i)},z_i,T_{0}^{(i+1)},z_{i+1})\Big]\prod_{i=1}^{m-1} dz_i\\
=&\sum_{\ell=1}^m\int_{(\R^d)^{m-1}} \Big[\prod_{i=0}^{\ell-2}\Gamma(T_0^{(i)},z_i,T_{0}^{(i+1)},z_{i+1})\Big((\Gamma^h-\Gamma)(T_0^{(\ell-1)},z_{\ell-1},T_0^{(\ell)},z_\ell)\Big)\\
&\times \prod_{i=\ell}^{m-1}\Gamma^h(T_0^{(i)},z_i,T_{0}^{(i+1)},z_{i+1})\Big]\prod_{i=1}^{m-1} dz_i.
\end{align*}
The heat kernel estimates \eqref{aronson-gamma} and \eqref{aronson-gammah} as well as the error analysis control \eqref{BD_THM-besov} in short time are invariant under time shift. This readily gives the statement for the arbitrary time $t$ considered.
} 
}

We conclude this section with some heuristics about power counting which give some intuition about the final rate we obtain. Let us consider the special case $\alpha=2^- $, $\beta=-1/2^+ $, $p=q=r=\infty $. Going back the Duhamel representations \eqref{duhamel-Diff} and \eqref{duhamel-scheme}, we see that for the spatial integrals to make sense, it is necessary to consider, by duality (see eq. \eqref{dual-ineq}), the $B_{1,1}^{1/2^-} $ norm of $\nabla p_\alpha(t-r,\cdot) $ yielding already a time singularity of order $(t-r)^{-(1/2^++1/4^-)}=(t-r)^{-3/4^-} $. Recalling from the previous description of the overall strategy that it is also necessary to consider the $\rho=1/2^- $-Hölder modulus of the normalized error, we see the resulting time singularity would read as $(t-r)^{-1^-} $, which leads to an almost vanishing convergence rate, corresponding to the remaining space to keep integrability. This matches what was previously stated in Theorem \ref{thm-besov} and Remark \ref{REM_AFTER_THM}.

\subsection{Decomposition of the error}\label{ssc:decomposition}
{\color{black}We decompose the error into several terms that will be bounded separately. From the Duhamel expansions stated in Proposition \ref{prop-main-estimates-D} we have the following:}
\begin{align}
	& \Gamma^h(0,x,t,y)-\Gamma(0,x,t,y)\notag\\
	& \qquad =\int_0^h \E_{0,x} \left[b (s,X_s)\cdot \nabla_y p_\alpha (t-s,y-X_s)-\mathfrak{b}_h(s,{x})\cdot\nabla_y p_\alpha (t-s,y-X_s^h) \right] \d s\notag\\
	& \qquad \qquad + \int_h^{\tau_t^h-h} \E_{0,x} \left[ b(s,X_s)\cdot \nabla_y p_\alpha(t-s,y-X_s)-b(s,X_{\tau_s^h})\cdot \nabla_y  p_\alpha (t-s,y-X_{\tau_s^h})	\right] \d s\notag\\
	& \qquad \qquad + \int_h^{\tau_t^h-h} \E_{0,x} \left[ b(s,X_{\tau_s^h})\cdot \nabla_y p_\alpha(t-s,y-X_{\tau_s^h})-
	\mathfrak b_h(s,X_{\tau_s^h})\cdot \nabla_y p_\alpha(t-s,y-X_{\tau_s^h})\right] \d s \nonumber\\
	& \qquad \qquad + \int_h^{\tau_t^h-h} \E_{0,x} \bigg[ \mathfrak b_h
	(s,X_{\tau_s^h})\cdot \nabla_y p_\alpha (t-s,y-X_{\tau_s^h})
	-
	\mathfrak b_h(s,X_{\tau_s^h}^h)\cdot \nabla_y  p_\alpha (t-s,y-X_{\tau_s^h}^h) \bigg] \d s\notag\\
	&\qquad \qquad + \int_h^{\tau_t^h-h} \E_{0,x} \left[\mathfrak{b}_h(s,X_{\tau_s^h}^h)\cdot \left( \nabla_y p_\alpha(t-s,y-X_{\tau_s^h}^h)-\nabla_y  p_\alpha (t-s,y-X_s^h)\right)\right]\d s\notag\\
	&\qquad\qquad+\int_{\tau_t^h-h}^t \E_{0,x} \left[ b(s,X_s)\cdot \nabla_y p_\alpha (t-s,y-X_s)-\mathfrak{b}_h(s,X_{\tau_s^h}^h)\cdot\nabla_y p_\alpha (t-s,y-X_s^h)  \right] \d s\notag\\
	&\qquad =:\big(\Delta_1 + \Delta_2 + \Delta_3 + \Delta_4 + \Delta_5+\Delta_6\big)(0,x,t,y).\label{DECOUP_ERR}
\end{align}
{\color{black}A visual depiction of the error decomposition can be found in Figure \ref{fig1} below.}
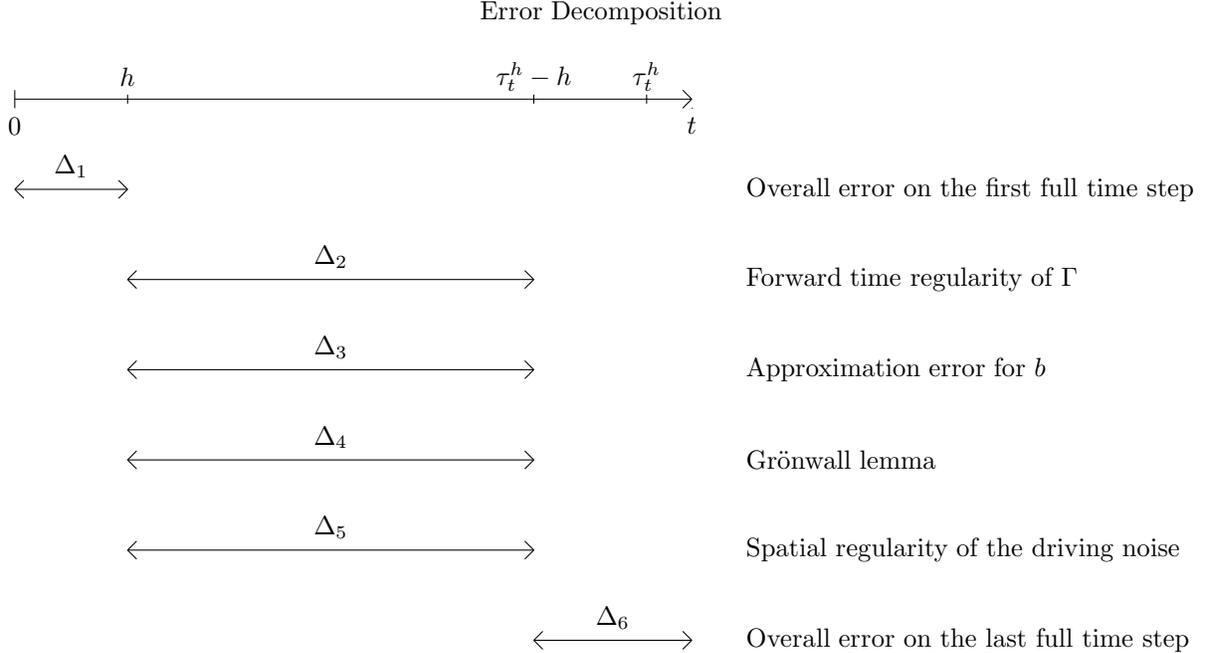
\begin{figure}[h] \label{fig1}
	\begin{centering}
{\color{black} Error Decomposition \\ \ \\}
		\begin{tikzpicture}[scale=6]
			\draw (1.5,-0.02)--(1.5,-0.02) node[anchor=north] {$t$};
			\draw (0,0.02)--(0,-0.02) node[anchor=north] {0};
			\draw (1.48,0.02)--(1.5,0) ;
			\draw (1.48,-0.02)--(1.5,0) ;		
			\draw (0,0)--(1.5,0);
			
			\node at (0.25,0.05) {$h$};
			\draw (0.25,0.01)--(0.25,-0.01) {};
			
			\node at (1.15,0.05) {$\tau_t^h-h$};
			\draw (1.15,0.01)--(1.15,-0.01) {};
			
			\node at (1.4,0.05) {$\tau_t^h$};
			\draw (1.4,0.01)--(1.4,-0.01) {};
			

			\draw (0,-0.2)--(0.25,-0.2);
			\draw (0,-0.2)--(0.02,-0.18);
			\draw (0,-0.2)--(0.02,-0.22);
			\draw (0.25,-0.2)--(0.23,-0.18);
			\draw (0.25,-0.2)--(0.23,-0.22);
			\node [draw=none] at (0.125,-0.15) {$\Delta_1$} ;
			\node [draw=none,anchor=west] at (1.6,-0.2) {Overall error on the first full time step} ;
			
			\draw (0.25,-0.4)--(1.15,-0.4);
			\draw (0.25,-0.4)--(0.27,-0.38);
			\draw (0.25,-0.4)--(0.27,-0.42);
			\draw (1.15,-0.4)--(1.13,-0.38);
			\draw (1.15,-0.4)--(1.13,-0.42);
			\node [draw=none] at (0.7,-0.35) {$\Delta_2$};
			\node [draw=none,anchor=west] at (1.6,-0.4) {Forward time regularity of $\Gamma$} ;
			
			\draw (0.25,-0.6)--(1.15,-0.6);
			\draw (0.25,-0.6)--(0.27,-0.58);
			\draw (0.25,-0.6)--(0.27,-0.62);
			\draw (1.15,-0.6)--(1.13,-0.58);
			\draw (1.15,-0.6)--(1.13,-0.62);
			\node [draw=none] at (0.7,-0.55) {$\Delta_3$};
			\node [draw=none,anchor=west] at (1.6,-0.6) {Approximation error for $b$} ;
			
			\draw (0.25,-0.8)--(1.15,-0.8);
			\draw (0.25,-0.8)--(0.27,-0.78);
			\draw (0.25,-0.8)--(0.27,-0.82);
			\draw (1.15,-0.8)--(1.13,-0.78);
			\draw (1.15,-0.8)--(1.13,-0.82);
			\node [draw=none] at (0.7,-0.75) {$\Delta_4$};
			\node [draw=none,anchor=west] at (1.6,-0.8) {Grönwall lemma} ;
			
			\draw (0.25,-1)--(1.15,-1);
			\draw (0.25,-1)--(0.27,-0.98);
			\draw (0.25,-1)--(0.27,-1.02);
			\draw (1.15,-1)--(1.13,-0.98);
			\draw (1.15,-1)--(1.13,-1.02);
			\node [draw=none] at (0.7,-0.95) {$\Delta_5$};
			\node [draw=none,anchor=west] at (1.6,-1) {Spatial regularity of the driving noise} ;
			
			\draw (1.15,-1.2)--(1.5,-1.2);
			\draw (1.15,-1.2)--(1.17,-1.18);
			\draw (1.15,-1.2)--(1.17,-1.22);
			\draw (1.5,-1.2)--(1.48,-1.18);
			\draw (1.5,-1.2)--(1.48,-1.22);
			\node [draw=none] at (1.325,-1.15) {$\Delta_6$};
			\node [draw=none,anchor=west] at (1.6,-1.2) {Overall error on the last full time step} ;
		\end{tikzpicture}\\[0.5cm]
	\end{centering}
	\caption{{\color{black} Decompositon of the error into 6 terms, including the first and last time steps ($\Delta_1$ and $\Delta_6$), the approximation error for the distribution $b$ ($\Delta_3$), the errors due to the (ir)regularity of the density and the noise ($\Delta_5$ and $\Delta_2$) and the term that is treated via a Gronwall-type argument ($\Delta_4$).} }
\end{figure}

\subsection{Control of the supremum norm}\label{subsec-proof-main-thm-supnorm}
\paragraph{Term $\Delta_1 $: first time step.}
    For $\Delta_1$, we rely on the fact that we work on the first time step.  Let us first expand the expectation:
\begin{align*}
	&\Delta_1(0,x,t,y)\\ 
	=&  \int_0^h \int \Big(\Gamma(0,x,s,z)b(s,z)\cdot \nabla_y  p_\alpha (t-s,y-z) - \Gamma^h(0,x,s,z)\mathfrak b_h(s,x) \cdot \nabla_y p_\alpha (t-s,y-z)\Big)\d z\d s\\
	=:& \Big(\Delta_{1,1}+\Delta_{1,2}\Big)(0,x,t,y).
\end{align*}
For $\Delta_{1,1}$, which involves the distributional $b$, we have to rely on the duality inequality in Besov spaces \eqref{dual-ineq}. Assuming w.l.o.g. that $t>2 h$ so that $(t-s)\asymp t$, then using \eqref{besov-estimate-gamma-and-stable} (taking therein $\zeta\in (-\beta,-\beta+\gamma) $) {\color{black}and H\"older inequality and subsequently \eqref{eq:Beta1} (taking $\aa=\frac\zeta\alpha$ and $\bb=\frac d {\alpha p}$ therein)} we get 
\begin{align}
	|\Delta_{1,1}(0,x,t,y)|&\lesssim \int_0^h \Vert b (s,\cdot)\Vert_{\B_{p,q}^\beta} \Vert\Gamma (0,x,s,\cdot)\nabla_y p_{\alpha} (t-s,y-\cdot)\Vert_{\B_{p',q'}^{-\beta}}\d s\nonumber\\
	&\lesssim \bar{p}_{\alpha}(t,y-x)\int_0^h \Vert b (s,\cdot)\Vert_{\B_{p,q}^\beta} \frac{t^{\frac{\beta}{\alpha}}}{(t-s)^{\frac{1}{\alpha}}}\left[\frac{1}{s^{\frac{d}{\alpha p}}} + \frac{1}{(t-s)^{\frac{d}{\alpha p}}}\right]\left[
	\frac{t^{\frac{\zeta}{\alpha}}}{s^{\frac{\zeta}{\alpha}}}+\frac{t^{\frac{\zeta}{\alpha}}}{(t-s)^{\frac{\zeta}{\alpha}}}\right]\d s\nonumber\\
	&{\color{black}\lesssim \bar{p}_{\alpha}(t,y-x)\int_0^h \Vert b (s,\cdot)\Vert_{\B_{p,q}^\beta} \frac{t^{\frac{\beta}{\alpha}}}{t^{\frac{1}{\alpha}}}\left[\frac{1}{s^{\frac{d}{\alpha p}}} + \frac{1}{t^{\frac{d}{\alpha p}}}\right]\left[\frac{t^{\frac{\zeta}{\alpha}}}{s^{\frac{\zeta}{\alpha}}}+\frac{t^{\frac{\zeta}{\alpha}}}{t^{\frac{\zeta}{\alpha}}}\right] \d s }\nonumber\\
	&{\color{black}\lesssim \bar{p}_{\alpha}(t,y-x) t^{\frac{\beta-1+\zeta}{\alpha}} \int_0^h \Vert b (s,\cdot)\Vert_{\B_{p,q}^\beta} \left[\frac{1}{s^{\frac{d}{\alpha p}}} + \frac{1}{t^{\frac{d}{\alpha p}}}\right] \frac{1}{s^{\frac{\zeta}{\alpha}}} \d s }\nonumber\\
	& {\color{black}\lesssim \bar{p}_{\alpha}(t,y-x) t^{\frac{\beta-1+\zeta}{\alpha}}  \Vert b \Vert_{L^r-\B_{p,q}^\beta} \left(\int_0^h \left[\frac{1}{s^{\frac{d}{\alpha p}}} + \frac{1}{t^{\frac{d}{\alpha p}}}\right]^{r'} \frac{1}{s^{\frac{\zeta r'}{\alpha}}} \d s\right)^{\frac{1}{r'}} }\nonumber\\
	&{\color{black}\lesssim \bar{p}_{\alpha}(t,y-x)  t^{\frac{\beta-1+\zeta}{\alpha}} \Vert b \Vert_{L^r-\B_{p,q}^\beta} \left[h^{1-\frac1r -\frac{\zeta}{\alpha} -\frac{d}{\alpha p}}+ t^{-\frac{d}{\alpha p}} h^{1-\frac1r -\frac{\zeta}{\alpha}}\right]}\nonumber\\
	&{\color{black}\lesssim \bar{p}_{\alpha}(t,y-x) t^{-\frac\beta\alpha} t^{\frac{2\beta-1+\zeta}{\alpha}}  \left[h^{1-\frac1r -\frac{\zeta}{\alpha} -\frac{d}{\alpha p}}+ t^{-\frac{d}{\alpha p}} h^{1-\frac1r -\frac{\zeta}{\alpha}}\right]}\nonumber\\
	&{\color{black}\lesssim \bar{p}_{\alpha}(t,y-x) t^{-\frac\beta\alpha} (2h)^{\frac{2\beta-1+\zeta}{\alpha}}  \left[h^{1-\frac1r -\frac{\zeta}{\alpha} -\frac{d}{\alpha p}}+ (2h)^{-\frac{d}{\alpha p}} h^{1-\frac1r -\frac{\zeta}{\alpha}}\right]}\nonumber\\
	&{\color{black}\lesssim \bar{p}_{\alpha}(t,y-x) t^{-\frac\beta\alpha} h^{\frac{2\beta -1 }{\alpha}} h^{1-\frac1r-\frac d{\alpha p}} }\nonumber\\
	&= \bar{p}_{\alpha}(t,y-x)h^{\frac{\gamma}{\alpha}}t^{-\frac{\beta}{\alpha}},\label{maj-D11-mainthm}
\end{align}
{\color{black}having used the fact that $t>2h$ and $-\frac d{\alpha p}<0$ and $ \frac {2\beta-1+\zeta}{\alpha}<0$ as well as the definition of $\gamma= \alpha+2\beta -\frac dp - \frac \alpha r-1$ given in  \eqref{eq:gamma}.} 
For $\Delta_{1,2}$, using the $L^\infty$ bound of $\mathfrak{b}_h$ {\color{black}given in} \eqref{CTR_PONCTUEL_BH}, \textcolor{black}{\eqref{aronson-gammah}}, {\color{black}the bound \eqref{derivatives-palpha} on the spatial derivative and the approximate convolution property \eqref{APPR_CONV_PROP} for $\bar p_\alpha$ and then the rescaled Beta function \eqref{eq:Beta}}, we get
\begin{align}
	|\Delta_{1,2}(0,x,t,y)|&\lesssim  \int_0^h s^{-\frac{d}{\alpha p} + \frac{\beta}{\alpha}} \|b(s,\cdot)\|_{\B_{p,q}^\beta}\frac{1}{(t-s)^{\frac{1}{\alpha}}}\int \bar p_\alpha (s,z-x) \bar p_\alpha (t-s,y-z)  \d z\d s\nonumber\\
	&{\color{black} \lesssim \bar p_{\alpha} (t,y-x)  \|b\|_{L^r-\B_{p,q}^\beta} \left( \int_0^h s^{(-\frac{d }{\alpha p} + \frac{\beta}{\alpha})r'} \frac{1}{(t-s)^{\frac{r'}{\alpha}}} \d s \right)^{\frac1{r'}}} \nonumber\\
	&{\color{black} \lesssim \bar p_{\alpha} (t,y-x)  \|b\|_{L^r-\B_{p,q}^\beta} \left( \int_0^h s^{(-\frac{d }{\alpha p} + \frac{\beta}{\alpha})r'} \frac{1}{(h-s)^{\frac{r'}{\alpha}}} \d s \right)^{\frac1{r'}}} \nonumber\\
	&{\color{black} \lesssim \bar p_{\alpha} (t,y-x)  \|b\|_{L^r-\B_{p,q}^\beta} \left( h^{1+(-\frac{d }{\alpha p} + \frac{\beta}{\alpha})r' - \frac{r'}{\alpha} }  \right)^{\frac1{r'}}} \nonumber\\
	&{\color{black} \lesssim \bar p_{\alpha} (t,y-x)   h^{1 - \frac1r -\frac{d }{\alpha p} + \frac{\beta}{\alpha} - \frac{1}{\alpha} }  t^{\frac\beta\alpha -\frac \beta\alpha} } \nonumber\\
&\lesssim \bar p_\alpha (t,y-x)h^{\frac \gamma\alpha}t^{-\frac{\beta}{\alpha}},\label{maj-D12-mainthm}
\end{align}
 recalling that $h\le t$ {\color{black} and $\beta<0$} for the last inequality. 
 We eventually get:
 \begin{equation}
 |\Delta_{1}(0,x,t,y)|\lesssim \bar p_\alpha (t,y-x)h^{\frac \gamma\alpha}t^{-\frac{\beta}{\alpha}}.\label{maj-D1-mainthm}
 \end{equation}
 \paragraph{Term $\Delta_2 $: time sensitivity of the density of the SDE.}
 Let us turn to $\Delta_2$. Expanding the expectation, using {\color{black}the duality inequality \eqref{dual-ineq}, then} the product rule \eqref{PR}, the heat kernel estimate \eqref{holder-time-gamma}, and the control \eqref{besov-estimate-stable}, we get for $\zeta,\rho\in (-\beta, -2\beta)$ (so that $\rho\in (-\beta, 1]$ is satisfied),
 \begin{align}\label{eq:Delta2}
 	&|\Delta_2(0,x,t,y)|\\ 
	=& \left| \int_h^{\tau_t^h-h}  \int [\Gamma(0,x,s,z)-\Gamma (0,x,\tau_s^h,z)] b(s,z)\cdot \nabla_y p_{\alpha}(t-s,y-z)	\d z \d s\right|\notag\\
 	 \lesssim&  \int_h^{\tau_t^h-h}  \Vert b(s,\cdot)\Vert_{\B_{p,q}^\beta}  \left\Vert \frac{\Gamma(0,x,s,\cdot)-\Gamma (0,x,\tau_s^h,\cdot)}{\bar p_\alpha (s,\cdot-x)} \right\Vert_{\B^{\rho}_{\infty,\infty}}\Vert \bar p_\alpha (s,\cdot-x)\nabla_y p_{\alpha}(t-s,y-\cdot)\Vert_{\B_{p',q'}^{-\beta}} \d s\notag\\
 	 \lesssim& \bar p_\alpha(t,y-x)\int_h^{\tau_t^h-h}  \Vert b(s,\cdot)\Vert_{\B_{p,q}^\beta} \frac{(s-\tau_s^h)^{\frac{\gamma-\eps}{\alpha}}}{s^{\frac{\gamma-\eps+\rho}{\alpha}}} \frac{t^{\frac{\beta}{\alpha}}}{(t-s)^{\frac{1}{\alpha}}}\left[\frac{1}{s^{\frac{d}{\alpha p}}} + \frac{1}{(t-s)^{\frac{d}{\alpha p}}}\right]\left[\frac{t^{\frac{\zeta}{\alpha}}}{s^{\frac{\zeta}{\alpha}}}+\frac{t^{\frac{\zeta}{\alpha}}}{(t-s)^{\frac{\zeta}{\alpha}}}\right] \d s\notag\\
	 \lesssim& \textcolor{black}{\bar p_\alpha(t,y-x) h^{\frac{\gamma-\varepsilon}{\alpha}}t^{\frac{\beta+\zeta}\alpha}\|b\|_{L^r-\B_{p,q}^\beta}\Big(\int_{h}^{\tau_t^h-h} \frac{1}{s^{r'\frac{\gamma-\eps+\rho}{\alpha}}} \frac{1}{(t-s)^{r'\frac{1}{\alpha}}}\left[\frac{1}{s^{\frac{d}{\alpha p}}} + \frac{1}{(t-s)^{\frac{d}{\alpha p}}}\right]^{r'}\left[\frac{1}{s^{\frac{\zeta}{\alpha}}}+\frac{1}{(t-s)^{\frac{\zeta}{\alpha}}}\right]^{r'} \d s\Big)^{\frac 1{r'}}}\notag\\
	  \lesssim& {\color{black}\bar p_\alpha(t,y-x) h^{\frac{\gamma-\varepsilon}{\alpha}}t^{\frac{\beta+ \zeta}\alpha} \Big(\int_{0}^{t} \frac{1}{s^{r'\frac{\gamma-\eps+\rho}{\alpha}}} \frac{1}{(t-s)^{r'\frac{1}{\alpha}}}\left[\frac{1}{s^{\frac{d}{\alpha p}}} + \frac{1}{(t-s)^{\frac{d}{\alpha p}}}\right]^{r'}\left[\frac{1}{s^{\frac{\zeta}{\alpha}}}+\frac{1}{(t-s)^{\frac{\zeta}{\alpha}}}\right]^{r'} \d s\Big)^{\frac 1{r'}} .} \notag
	  \end{align}
The singular integral can be bounded using Lemma \ref{lm:sing_int} with $\aa=\frac{\gamma-\varepsilon+\rho}\alpha, \bb= \frac1\alpha, \cc= \frac d{\alpha p}, \dd=\frac\zeta\alpha$, provided that 
\[
r'(\frac{\gamma-\varepsilon+\rho}\alpha+ \frac d{\alpha p} + \frac\zeta\alpha)<1 
\iff
\gamma - \varepsilon + \rho +\frac dp +\zeta < \alpha- \frac\alpha r
\iff 
2\beta -1-\varepsilon + \rho +\zeta <0
\] 
which is satisfied since $\rho,\zeta \in (-\beta, -2\beta)$ and $\beta>-\frac12$, 
and provided that
\[
r'(\frac{1}\alpha+ \frac d{\alpha p} + \frac\zeta\alpha)<1 
\]
which is equivalent to
\[
1 +\frac dp +\zeta < \alpha- \frac\alpha r
\iff 
1-\alpha +\frac dp +\frac\alpha r <-\zeta 
\ei{\iff -\gamma+2\beta +\zeta <0},
\]  
and which is satisfied  since $\zeta<-2\beta$ by assumption and $\gamma>0$.
Thus  by \eqref{eq:Beta2-5}
\[
\Big(\int_{0}^{t} \frac{1}{s^{r'\frac{\gamma-\eps+\rho}{\alpha}}} \frac{1}{(t-s)^{r'\frac{1}{\alpha}}}\left[\frac{1}{s^{\frac{d}{\alpha p}}} + \frac{1}{(t-s)^{\frac{d}{\alpha p}}}\right]^{r'}\left[\frac{1}{s^{\frac{\zeta}{\alpha}}}+\frac{1}{(t-s)^{\frac{\zeta}{\alpha}}}\right]^{r'} \d s\Big)^{\frac 1{r'}} \lesssim t^{1-\frac1r - \frac{\gamma-\varepsilon+  \rho +1 +\frac dp +\zeta}\alpha} = t^{-\frac{2\beta}\alpha + \frac{\varepsilon-\rho-\zeta}\alpha}, 
\]
and plugging it in \eqref{eq:Delta2} we finally get, \sm{recalling \eqref{CHOICE_FOR_RHO},} 
\begin{align}\label{maj-D2-mainthm}
		|\Delta_2(0,x,t,y)|
		 \lesssim 
		 \bar p_\alpha(t,y-x) h^{\frac{\gamma-\varepsilon}{\alpha}}t^{\frac{\beta+ \zeta}\alpha} t^{ -\frac{2\beta}\alpha + \frac{\varepsilon-\rho-\zeta}\alpha}
		 =\bar p_\alpha(t,y-x)h^{\frac{\gamma-\eps}{\alpha}} t^{\textcolor{black}{\frac{\eps-\beta-\rho}{\alpha}} }\sm{=\bar p_\alpha(t,y-x)h^{\frac{\gamma-\eps}{\alpha}} t^{{\frac{\eps}{2\alpha}} }}.
 \end{align}
   
\paragraph{Term $\Delta_3 $: approximation of the singular drift.} 
 Let us turn to $\Delta_3$. Expanding the inner expectation and using the proof of Proposition 2 in \cite{CdRJM22} \textcolor{black}{ to write $\|b(s,\cdot)-P_h^\alpha b(s,\cdot)\|_{\B_{p,q}^{\beta-\gamma}}\lesssim h^{\frac{\gamma}{\alpha}}\|b(s,\cdot)\|_{\B_{p,q}^{\beta}}  $}, \sm{we derive, using this time \eqref{dual-ineq} and \eqref{besov-estimate-gamma-and-stable}  replacing therein the argument $s$ of $\Gamma $ by $\tau_s^h $ thanks to Remark \ref{REM_WITH_DISCR}, that for $\varepsilon>0 $, $\zeta \in (-\beta+\gamma-\varepsilon,-\beta+\gamma)$:}
\begin{align}
	&\nonumber|\Delta_3(0,x,t,y)|\\
	=&\left| \int_h^{\tau_t^h-h} \int\Gamma (0,x,\tau_s^h,z) \left[b(s,z)- \mathfrak{b}_h(s,z)\right]\cdot \ei{\nabla_y}  p_\alpha (t-s,y-z) \d z \d s\right|\nonumber\\
	\nonumber\lesssim& \int_h^{\tau_t^h-h}  \|b(s,\cdot)-P_h^\alpha b(s,\cdot)\|_{\B_{p,q}^{\textcolor{black}{\beta-\gamma+\varepsilon}}} \left\Vert\Gamma (0,x,\tau_s^h,\cdot)\nabla_y p_{\alpha}(t-s,y-\cdot)\right\Vert_{\B_{p',q'}^{\textcolor{black}{-\beta+\gamma-\varepsilon}}}  \d s\\
	\nonumber\lesssim& h^{\frac{\gamma\textcolor{black}{-\varepsilon}}\alpha}\int_h^{\tau_t^h-h} \|b(s,\cdot)\|_{\B_{p,q}^\beta}\frac{\bar{p}_{\alpha} (t,x-y)}{(t-s)^{\frac{1}{\alpha}}} t^{\frac{\beta- \gamma\textcolor{black}{+\varepsilon}}{\alpha}}\left[ \frac{1}{(t-s)^{\frac{d }{\alpha  p}}} +\frac{1}{s^{\frac{d }{\alpha  p}}}\right] \left[ \frac{t^{\frac{\zeta}{\alpha}}}{(t-s)^{\frac{\zeta }{\alpha}}} + \frac{t^{\frac{\zeta}{\alpha}}}{s^{\frac{\zeta }{\alpha}}} \right] \d s\\
	\nonumber \lesssim&   { \color{black}\bar p_\alpha(t,y-x)h^{\frac{\gamma\textcolor{black}{-\varepsilon}}{\alpha}} t^{\frac{\beta - \gamma +\textcolor{black}{\varepsilon}+\zeta }{\alpha}} \|b\|_{L^r-\B_{p,q}^\beta} \Big(\int_{0}^{t} \frac{1}{(t-s)^{r'\frac{1}{\alpha}}}\left[\frac{1}{s^{\frac{d}{\alpha p}}} + \frac{1}{(t-s)^{\frac{d}{\alpha p}}}\right]^{r'}\left[\frac{1}{s^{\frac{\zeta}{\alpha}}}+\frac{1}{(t-s)^{\frac{\zeta}{\alpha}}}\right]^{r'} \d s\Big)^{\frac 1{r'}} }\\
 \nonumber  \lesssim&  { \color{black}\bar p_\alpha(t,y-x)h^{\frac{\gamma\textcolor{black}{-\varepsilon}}{\alpha}} t^{\frac{\beta - \gamma+\textcolor{black}{\varepsilon} +\zeta }{\alpha}}  t^{1-\frac1r - \frac1\alpha -\frac{d}{\alpha p}-\frac{\zeta}{\alpha}}}\\
 	 \lesssim& \bar p_\alpha(t,y-x)h^{\frac{\gamma\textcolor{black}{-\varepsilon}}{\alpha}} t^{-\frac{\beta}{\alpha}},\label{maj-D3-mainthm}
\end{align}
having used Lemma  \ref{lm:sing_int} as for the term $\Delta_2$ but with $\aa=0$, and the definition of $\gamma$.

\paragraph{Term $\Delta_4 $: Gronwall or circular type argument.} Due to the singular drift, this term emphasizes that  in order to derive the main theorem, not only will we need to control the supremum of the difference but as well a kind of H\"older modulus. Namely, for $\rho \in (-\beta,-\beta+\gamma)$, using \eqref{PR}, \ei{\eqref{CTR_BESOV_BH} and \sm{the modification of \eqref{besov-estimate-stable} discussed in Remark \ref{REM_WITH_DISCR}}, with $\zeta \in (-\beta, 1]$} we have,

\begin{align}
&|\Delta_4(0,x,t,y)|\notag\\
=&\left|\int_h^{\tau_t^h-h} \E_{0,x} \bigg[ \mathfrak b_h
	(s,X_{\tau_s^h})\cdot \nabla_y p_\alpha (t-s,y-X_{\tau_s^h})
	-
	\mathfrak b_h(s,X_{\tau_s^h}^h)\cdot \nabla_y  p_\alpha (t-s,y-X_{\tau_s^h}^h) \bigg] \d s\right|\notag\\
	=&\left|\int_h^{\tau_t^h-h}\int_{\R^d } (\Gamma-\Gamma^h)(0,x,\tau_s^h,z) \mathfrak b_h(s,z) \nabla_y  p_\alpha (t-s,y-z)  \d z \d s\right|\notag\\
	\le &\int_h^{\tau_t^h-h} \|\mathfrak b_h(s,\cdot)\|_{\B_{p,q}^\beta}\left \|\frac{(\Gamma-\Gamma^h)(0,x,\tau_s^h,\cdot)}{\bar p_\alpha(\tau_s^h,\cdot-x)}\right\|_{\B_{\infty,\infty}^\rho}\|\bar p_\alpha(\tau_s^h,\cdot-x) \nabla_y p_{\alpha}(t-s,y-\cdot)\|_{\B_{p',q'}^{-\beta}}\d s\notag\\
\le &\ei{	\int_h^{\tau_t^h-h} \| b(s,\cdot)\|_{\B_{p,q}^\beta}\left \|\frac{(\Gamma-\Gamma^h)(0,x,\tau_s^h,\cdot)}{\bar p_\alpha(\tau_s^h,\cdot-x)}\right\|_{\B_{\infty,\infty}^\rho} \bar p_\alpha(t,y-x)
\frac{t^{\frac\beta\alpha}}{(t-s)^\frac1\alpha} \left[ \frac{1}{s^{\frac{d }{\alpha  p}}}+\frac{1}{(t-s)^{\frac{d }{\alpha  p}}} \right] \left[\frac{t^{\frac{\zeta}{\alpha}}}{s^{\frac{\zeta }{\alpha}}}+\frac{t^{\frac{\zeta}{\alpha}}}{(t-s)^{\frac{\zeta }{\alpha}}}  \right] 
	\d s.\notag}
\end{align}	
{\color{black} Combining \eqref{THE_DECOUP_FOR_HOLDER_MOD_E} with \eqref{DEF_GH} we see that 
\[
\left \|\frac{(\Gamma-\Gamma^h)(0,x,\tau_s^h,\cdot)}{\bar p_\alpha(\tau_s^h,\cdot-x)}\right\|_{\B_{\infty,\infty}^\rho}
\lesssim \frac1{(\tau_s^h)^\frac\rho\alpha} g_{h,\rho}(\tau_s^h)\lesssim \sm{\frac1{s^\frac\rho\alpha} g_{h,\rho}(\tau_s^h),}
\]
\sm{on the considered time interval.}
}
With this bound at hand we write:
\begin{align}
|\Delta_4(0,x,t,y)|	\le &\sup_{s\in (h,T]}g_{h,\rho}(s)\bar p_\alpha(t,y-x)\notag\\
	&\hspace*{.5cm}\times\int_0^t \|b(s,\cdot)\|_{\B_{p,q}^\beta}\frac{1}{s^{\frac{\rho}{\alpha}}}\frac{t^{\frac{\beta}{\alpha}}}{(t-s)^{\frac{1}{\alpha}}}\left[\frac{1}{s^{\frac{d}{\alpha p}}} + \frac{1}{(t-s)^{\frac{d}{\alpha p}}}\right]\left[\frac{t^{\frac{\zeta}{\alpha}}}{s^{\frac{\zeta}{\alpha}}}+\frac{t^{\frac{\zeta}{\alpha}}}{(t-s)^{\frac{\zeta}{\alpha}}}\right] \d s\notag\\
	\lesssim & \ei{ \sup_{s\in (h,T]} g_{h,\rho}(s) \bar p_\alpha(t,y-x)  \|b\|_{L^r-\mathbb B^{\beta}_{p,q}} t^{\frac{\beta +\zeta}{\alpha}} t^{1- \frac1r -\frac\rho\alpha - \frac1\alpha - \frac d {\alpha p} -\frac{\zeta}{\alpha}} }\\
	\lesssim & \sup_{s\in (h,T]}g_{h,\rho}(s) t^{\frac{-\beta-\rho+\gamma}{\alpha}}\bar p_\alpha(t,y-x)=\sup_{s\in (h,T]}g_{h,\rho}(s) t^{\frac{\gamma\textcolor{black}{-\frac{\eps}{2}}}{\alpha}}\bar p_\alpha(t,y-x),\label{maj-D4-mainthm}
\end{align}
where 
 for the second inequality  \ei{we used} Lemma \ref{lm:sing_int} as in the term $\Delta_2$ but with $\aa=\frac \rho\alpha$. The above contribution actually emphasizes that, in order to control the error on the densities we actually need to control a related Hölder modulus of continuity.  Let us as well point out that for the time contribution to be small (in order to perform the circular argument on the quantity $g_{h,\rho}$) we will as well assume w.l.o.g. that $\gamma>\textcolor{black}{\eps} $.

\paragraph{Term $\Delta_5 $: spatial sensitivities of the driving noise.}
Let us now turn to \ei{$\Delta_5 $. Using the harmonicity of the stable heat kernel (or martingale property of the driving noise) we have}
\begin{align*}
\Delta_5(0,t,x,y)=&\int_h^{\tau_t^h-h} \E_{0,x} \left[\mathfrak{b}_h(s,X_{\tau_s^h}^h)\cdot \left( \nabla_y p_\alpha(t-s,y-X_{\tau_s^h}^h)-\nabla_y  p_\alpha (t-s,y-X_s^h)\right)\right]\d s\\
=&\int_h^{\tau_t^h-h} \E_{0,x} \left[\mathfrak{b}_h(s,X_{\tau_s^h}^h)\cdot \left( \nabla_y p_\alpha(t-s,y-X_{\tau_s^h}^h)\right.\right.\\
&
\left.\left.-\nabla_y  p_\alpha \Big(t-\tau_s^h,y-(X_{\tau_s^h}^h+\int_{\tau_s^h}^s\mathfrak{b}_h(u,X_{\tau_s^h}^h)\d u)\Big)\right)\right]\d s.
\end{align*}
Adding and subtracting the relevant term we write
\begin{align*}
&|\Delta_5(0,x,t,y)|\\
\le& \Bigg|\int_{h}^{\tau_t^h-h}\int_{\R^d} \Gamma^h(0,x,\tau_s^h,z)\mathfrak{b}_h(s,z)\cdot \left( \nabla_y p_\alpha(t-s,y-z) 
-\nabla_y  p_\alpha (t-\tau_s^h,y-z)\right) \d z  \d r\d s\Bigg|\\
&+\Bigg|\int_{h}^{\tau_t^h-h} \int_{\R^d} \Gamma^h(0,x,\tau_s^h,z)\mathfrak{b}_h(r,z)\cdot \left( \nabla_y p_\alpha(t-\tau_s^h,y-z) \right.\notag \\
&\hspace*{.5cm}\left. -\nabla_y  p_\alpha (t-\tau_s^h,y-(z+\int_{\tau_s^h}^s\mathfrak{b}_h(u,{\color{black} z})\d u))\right) \d z  \d r\d s\Bigg|=:|\Delta_{51}(0,x,t,y)|+|\Delta_{52}(0,x,t,y)|.
\end{align*}
For $\Delta_{51}$, using the duality inequality \eqref{dual-ineq}, \eqref{besov-estimate-gammah-stable-sensi-holder-time} {\color{black} and \eqref{CTR_BESOV_BH}}, we have
\begin{align}
|\Delta_{51}(0,x,t,y)|&\le  \int_h^{\tau_t^h-h}\left| \int \Gamma^h(0,x,\tau_s^h,z)\mathfrak{b}_h(s,z)\cdot \left( \nabla_y p_\alpha(t-\tau_s^h,y-z)-\nabla_y  p_\alpha (t-s,y-z)\right) \d z \right| \d s\nonumber\\
&\lesssim \int_h^{\tau_t^h-h} \|\mathfrak b_h(s,\cdot)\|_{\B_{p,q}^\beta}  \| \Gamma^h (0,x,\tau_s^h,\cdot)\left( \nabla_y p_\alpha(t-\tau_s^h,y-\cdot)-\nabla_y  p_\alpha (t-s,y-\cdot)\right)\|_{\B_{p',q'}^{-\beta}}\d s\nonumber\\
&\lesssim \bar{p}_\alpha(t,y-x) \int_h^{\tau_t^h-h} \| {\color{black}b(s,\cdot)}\|_{\B_{p,q}^\beta} \frac{(s-\tau_s^h)^{\frac{\gamma}{\alpha}}t^{\frac{\beta}{\alpha}}}{(t-s)^{\frac{1+\gamma}{\alpha}}}\left[\frac{1}{s^{\frac{d }{\alpha  p}}}+ \frac{1}{(t-s)^{\frac{d }{\alpha  p}}} \right] \left[
\frac{t^{\frac{\zeta}{\alpha}}}{s^{\frac{\zeta }{\alpha}}}+\frac{t^{\frac{\zeta}{\alpha}}}{(t-s)^{\frac{\zeta }{\alpha}}}  \right]\d s\nonumber\\
\nonumber & {\color{black}\lesssim \bar{p}_\alpha(t,y-x) h^{\frac{\gamma}{\alpha}}t^{\frac{\beta+\zeta}{\alpha}} \| b\|_{L^r-\B_{p,q}^\beta} \Big(
 \int_0^{t} \frac{1}{(t-s)^{r' \frac{1+\gamma}{\alpha}}}\left[\frac{1}{s^{\frac{d }{\alpha  p}}}+ \frac{1}{(t-s)^{\frac{d }{\alpha  p}}} \right]^{r'} \!\!\left[\frac{1}{s^{\frac{\zeta }{\alpha}}}+\frac{1}{(t-s)^{\frac{\zeta }{\alpha}}}  \right]^{r'}\!\!\!\!\d s
 \Big)^{\frac1{r'}}}\\
\nonumber  & {\color{black}\lesssim \bar{p}_\alpha(t,y-x) h^{\frac{\gamma}{\alpha}}t^{\frac{\beta+\zeta}{\alpha}} t^{1-\frac1r - \frac{1+\gamma}{\alpha}  - \frac{d}{\alpha p} -\frac\zeta\alpha}  } \\
&\lesssim \bar{p}_\alpha(t,y-x) h^{\frac{\gamma}{\alpha}}t^{-\frac{\beta}{\alpha}}\label{maj-D51},
\end{align}
{\color{black}having used again Lemma   \ref{lm:sing_int} as for the term $\Delta_2$ but with $\aa=0$.}
On the other hand, using this time \eqref{besov-estimate-gammah-stable-PERTURB_DRIFT} with $\zeta\in(-\beta, -\beta+\gamma)$ and the fact that $s \asymp \tau_s^h$ in the considered time interval, we have,
\begin{align}
&|\Delta_{52}(0,x,t,y)|\nonumber\\
& \le \int_{h}^{\tau_t^h-h}  \|\mathfrak b_h(s,\cdot)\|_{\B_{p,q}^\beta} \left\| \Gamma^h(0,x,\tau_s^h,\cdot)\Big(\nabla_y p_\alpha(t-\tau_s^h,y-\cdot) -\nabla_y  p_\alpha (t-\tau_s^h,y-(\cdot+\int_{\tau_s^h}^s\mathfrak{b}_h(u,\cdot)\d u)) \Big)\right \|_{\B_{p',q'}^{-\beta}}\!\!\!\!
\d s\nonumber\\
&\lesssim \bar p_\alpha(t,y-x)\int_{h}^{\tau_t^h-h}  \|b(s,\cdot)\|_{\B_{p,q}^\beta} \frac{h^{\frac{\gamma-\beta}{\alpha}}\textcolor{black}{t^{\frac{\beta}{\alpha}}}}{(t-\tau_s^h)^{\frac{1}{\alpha}}}\left[\frac{1}{(\tau_s^h)^{\frac{d}{\alpha p}}} + \frac{1}{(t-\tau_s^h)^{\frac{d}{\alpha p}}}\right]\left[\frac{t^{\frac \zeta\alpha}}{(\tau_s^h)^{\frac{\zeta}{\alpha }}} + \frac{t^{\frac\zeta\alpha}}{(t-\tau_s^h)^{\frac{\zeta}{\alpha }}}\right]\d s\nonumber\\
&\lesssim \textcolor{black}{\bar p_\alpha(t,y-x)\int_{h}^{\tau_t^h-h}  \|b(s,\cdot)\|_{\B_{p,q}^\beta} \frac{h^{\frac{\gamma-\beta}{\alpha}}\textcolor{black}{t^{\frac{\beta}{\alpha}}}}{(t-s)^{\frac{1}{\alpha}}}\left[\frac{1}{s^{\frac{d}{\alpha p}}} + \frac{1}{(t-s)^{\frac{d}{\alpha p}}}\right]\left[\frac{t^{\frac \zeta\alpha}}{s^{\frac{\zeta}{\alpha }}} + \frac{t^{\frac\zeta\alpha}}{(t-s)^{\frac{\zeta}{\alpha }}}\right]\d s}\nonumber\\
&\lesssim \textcolor{black}{\bar p_\alpha(t,y-x)\int_{0}^{t}  \|b(s,\cdot)\|_{\B_{p,q}^\beta} \frac{h^{\frac{\gamma-\beta}{\alpha}}\textcolor{black}{t^{\frac{\beta}{\alpha}}}}{(t-s)^{\frac{1}{\alpha}}}\left[\frac{1}{s^{\frac{d}{\alpha p}}} + \frac{1}{(t-s)^{\frac{d}{\alpha p}}}\right]\left[\frac{t^{\frac \zeta\alpha}}{s^{\frac{\zeta}{\alpha }}} + \frac{t^{\frac\zeta\alpha}}{(t-s)^{\frac{\zeta}{\alpha }}}\right]\d s}\nonumber\\
&\lesssim \bar p_\alpha(t,y-x)h^{\frac{\gamma-\beta}{\alpha}}t^{\frac{\gamma-\beta}{\alpha}},\label{maj-D52}
\end{align}
where for the last inequality we again applied Lemma   \ref{lm:sing_int} \ei{as in $\Delta_2$ but} with $\aa=0$. 
From \eqref{maj-D51} and \eqref{maj-D52} we thus derive:
\begin{align}\label{maj-D5-mainthm}
|\Delta_5(0,x,t,y)|\lesssim \bar p_\alpha(t,y-x)h^{\frac{\gamma}{\alpha}} t^{-\frac{\beta}{\alpha}}.
\end{align}
\paragraph{Term $\Delta_6 $: last time step.}
This term is handled very much like $\Delta_1$, in the sense that the smallness will come from each contribution and not from sensitivities, since the time-interval is itself small. Namely,
\begin{align*}
&|\Delta_6(0,x,t,y)|\\
=&\left|\int_{\tau_t^h-h}^t \E_{0,x} \left[ b(s,X_s)\cdot \nabla_y p_\alpha (t-s,y-X_s)-\mathfrak{b}_h(s,X_{\tau_s^h}^h)\cdot\nabla_y p_\alpha (t-s,y-X_s^h)  \right] \d s\right|\\
\le&\left|\int_{\tau_t^h-h}^t \E_{0,x} \left[ b(s,X_s)\cdot \nabla_y p_\alpha (t-s,y-X_s)\right]\d s \right|+\left|\int_{\tau_t^h-h}^t\E_{0,x}\left[\mathfrak{b}_h(s,X_{\tau_s^h}^h)\cdot\nabla_y p_\alpha (t-s,y-X_s^h)  \right] \d s\right|\\
=:&\Big(|\Delta_{6,1}|+|\Delta_{6,2}|\Big)(0,x,t,y).
\end{align*}
 We assume w.l.o.g\ that $t>3h$, so for $\tau_t^h -h< s<t$ then $s \asymp t$ (since $s\in(\frac t2, t)$) and $t-(\tau_t^h-h) \asymp h$ (since $t-(\tau_t^h-h)  \in (h,2h)$). We rely on the duality inequality \eqref{dual-ineq} and on \eqref{besov-estimate-gamma-and-stable} with $j=0$ and any $\zeta\in(-\beta,1]$ to get
\begin{align}
	|\Delta_{6,1}(0,x,t,y)|&\lesssim \int_{\tau_t^h-h}^t \Vert b (s,\cdot)\Vert_{\B_{p,q}^\beta}  \Vert\Gamma (0,x,s,\cdot)\nabla_y p_{\alpha} (t-s,y-\cdot)\Vert_{\B_{p',q'}^{-\beta}}\d s\nonumber\\
	&\lesssim \bar{p}_{\alpha}(t,y-x)\int_{\tau_t^h-h}^t \Vert b (s,\cdot)\Vert_{\B_{p,q}^\beta} \frac{t^{\frac{\beta}{\alpha}}}{(t-s)^{\frac{1}{\alpha}}}\left[\frac{1}{s^{\frac{d}{\alpha p}}} + \frac{1}{(t-s)^{\frac{d}{\alpha p}}}\right]\left[
	\frac{t^{\frac{\zeta}{\alpha}}}{s^{\frac{\zeta}{\alpha}}}+\frac{t^{\frac{\zeta}{\alpha}}}{(t-s)^{\frac{\zeta}{\alpha}}}\right]\d s\nonumber\\
	&{\color{black}\lesssim \bar{p}_{\alpha}(t,y-x)  t^{\frac{\beta+\zeta}{\alpha}}  \int_{\tau_t^h-h}^t \Vert b (s,\cdot)\Vert_{\B_{p,q}^\beta} \frac{1}{(t-s)^{\frac{1+\zeta}{\alpha}}}\left[\frac{1}{t^{\frac{d}{\alpha p}}} + \frac{1}{(t-s)^{\frac{d}{\alpha p}}}\right] \d s\nonumber}\\
	&{\color{black}\lesssim \bar{p}_{\alpha}(t,y-x) \Vert b \Vert_{L^r-\B_{p,q}^\beta} t^{\frac{\beta+\zeta}\alpha} \left(\int_{\tau_t^h-h}^t\frac1 {(t-s)^{r'\frac{1+\zeta}{\alpha}}}\left[\frac{1}{t^{\frac{d}{\alpha p}}} + \frac{1}{(t-s)^{\frac{d}{\alpha p}}}\right]^ {r'} \d s\right)^{\frac{1}{r'}}\nonumber}\\
	&{\color{black}\lesssim \bar{p}_{\alpha}(t,y-x) t^{\frac{\beta+\zeta}\alpha} \left[ h^{1-\frac1r -\frac{1+\zeta}\alpha - \frac d {\alpha p}} + t^{-\frac d {\alpha p} } h^{1-\frac1r -\frac{1+\zeta}\alpha } \right],  }
\end{align}
{\color{black} having used the second bullet point of Lemma \ref{lm:sing_int} with $\aa=\frac{1+\zeta}\alpha, \bb= \frac d {\alpha p}$ and $v= \tau_t^h-h$. Notice that the assumption $r'(a+b)<1$  is equivalent to 
\[
\frac{1+\zeta + \frac dp}\alpha <1-\frac1r \iff 1+\zeta+\frac dp < \alpha - \frac \alpha r \iff \zeta< \alpha-\frac\alpha r  -\frac dp - 1
\] 
which is satisfied as long as $\zeta<-2\beta$ since $-2\beta<  \alpha-\frac\alpha r  -\frac dp - 1$ by assumption \eqref{serrin}. Thus for all $\zeta\in (-\beta, -2\beta)$ using that $3h<t$ we have
\begin{align}
	|\Delta_{6,1}(0,x,t,y)|& \lesssim \bar{p}_{\alpha}(t,y-x)t^{\frac{\beta+\zeta}{\alpha}}\left[h^{1-\frac1r- \frac{1+\zeta}{\alpha }-\frac{d}{\alpha p}} +h^{1-\frac1r- \frac{1+\zeta}{\alpha }} t^{-\frac{d}{\alpha p}} \right]\nonumber\\
	& \lesssim \bar{p}_{\alpha}(t,y-x)t^{\frac{\beta+\zeta}{\alpha}} h^{1-\frac1r- \frac{1+\zeta}{\alpha }-\frac{d}{\alpha p}} h^{\frac{2\beta}\alpha-\frac{2\beta}\alpha }\nonumber\\
	&\lesssim \bar{p}_{\alpha}(t,y-x)h^{\frac{\gamma}{\alpha}}t^{-\frac{\beta}{\alpha}},\label{maj-D61-mainthm}
	\end{align}
having used the fact that $\zeta<-2\beta$ so that $h^{-\frac{\zeta}\alpha-\frac{2\beta}\alpha }\leq t^{ -\frac{\zeta}\alpha-\frac{2\beta}\alpha }$.}

\sm{Let us now turn to  $\Delta_{6,2}$ and write:
\begin{align*}
|\Delta_{6,2}(0,t,x,y)|&=\left|\int_{\tau_t^h-h}^t\E_{0,x}\left[\mathfrak{b}_h(s,X_{\tau_s^h}^h)\cdot\nabla_y p_\alpha (t-s,y-X_s^h)  \right] \d s\right|\\
&=\left|\int_{\tau_t^h-h}^t\E_{0,x}\left[\mathfrak{b}_h(s,X_{\tau_s^h}^h)\cdot\nabla_y p_\alpha (t-\tau_s^h,y-(X_{\tau_s^h}^h+\int_{\tau_s^h}^s\mathfrak{b}_h(r,X_{\tau_s^h}^h)\d r)  \right] \d s\right|\\
&\le \int_{\tau_t^h-h}^t\int_{\R^d} \Gamma^h(0,x,\tau_s^h,z)|\mathfrak{b}_h(s,z)|\Big|\nabla_y p_\alpha (t-\tau_s^h,y-(z+\int_{\tau_s^h}^s\mathfrak{b}_h(r,z)\d r))\Big|  \d z \d s,
\end{align*}
where we again used the harmonicity of the stable heat kernel for the second equality. Then, from  Lemma \ref{Lemma_TRANS_SCHEME} (taking therein $v=\tau_s^h $), since $t-\tau_s^h\ge s-\tau_s^h $, it holds that 
$$\Big|\nabla_y p_\alpha (t-\tau_s^h,y-(z+\int_{\tau_s^h}^s\mathfrak{b}_h(r,z)\d r))\Big|\le \frac{C}{(t-\tau_s^h)^{\frac 1\alpha}}\bar p_\alpha(t-\tau_s^h,y-z).$$
Hence, recalling from \eqref{CTR_PONCTUEL_BH} that $|\mathfrak b_h(s,z)| \le C(s-\tau_s^h)^{-\frac d{\alpha p}+\frac \beta\alpha}\|b(s,\cdot)\|_{\B_{p,q}^\beta}$, and using as well the upper bound \eqref{aronson-gammah} for $\Gamma^h $,  we derive:}
\begin{align*}
|\Delta_{6,2}(0,t,x,y)|&\lesssim \int_{\tau_t^h-h}^t \int_{\R^d} \frac{1}{(t-s)^{\frac{1}{\alpha}}}\|b(s,\cdot)\|_{\B_{p,q}^\beta}(s-\tau_s^h)^{-\frac{d}{\alpha p} + \frac{\beta}{\alpha}}\bar p_\alpha (\tau_s^h,z-x) \bar p_\alpha (t-\tau_s^h,y-z)  \d z\d s\\
	&\ei{\lesssim  \bar p_{\alpha} (t,y-x)\textcolor{black}{\Vert b \Vert_{L^r-\B_{p,q}^\beta}}\left(\int_{\tau_t^h-h}^t \frac{1}{(s-\tau_s^h)^{r'(\frac{d}{\alpha p}{\color{black}-} \frac{\beta}{\alpha})}(t-s)^{\frac{r'}{\alpha}}} \d s\right)^{\frac 1{r'}}. }\nonumber
\end{align*}
The integral is computed by splitting the domain into  $[{\tau_t^h-h}, t ]=[{\tau_t^h-h}, {\tau_h^t}] \cup [{\tau_t^h}, t]$ and making use of the rescaled Beta function \eqref{eq:Beta}, to obtain
\begin{align}
&\int_{\tau_t^h-h}^t \frac{1}{(s-\tau_s^h)^{r'(\frac{d}{\alpha p}{\color{black}-} \frac{\beta}{\alpha})}(t-s)^{\frac{r'}{\alpha}}} \d s \notag\\
&= \int_{\tau_t^h-h}^{\tau^h_t} \frac{1}{(s-(\tau_t^h-h))^{r'(\frac{d}{\alpha p}- \frac{\beta}{\alpha})}(t-s)^{\frac{r'}{\alpha}}} \d s + \int_{\tau_t^h}^{t} \frac{1}{(s-\tau_t^h)^{r'(\frac{d}{\alpha p}- \frac{\beta}{\alpha})}(t-s)^{\frac{r'}{\alpha}}} \d s\notag\\
&\lesssim   \int_{\tau_t^h-h}^{\tau^h_t} \frac{1}{(s-(\tau_t^h-h))^{r'(\frac{d}{\alpha p}- \frac{\beta}{\alpha})}}  \frac{1}{(\tau^h_t-s)^{\frac{r'}{\alpha}}} \d s + \int_{\tau_t^h}^{t} \frac{1}{(s-\tau_t^h)^{r'(\frac{d}{\alpha p}- \frac{\beta}{\alpha})}} \frac{1}{(t-s)^{\frac{r'}{\alpha}}} \d s\notag\\
&\lesssim (\tau^h_t - (\tau^h_t -h))^{1- r'(\frac{d}{\alpha p}- \frac{\beta}{\alpha} +\frac1\alpha)} + (t - \tau^h_t )^{1- r'(\frac{d}{\alpha p}- \frac{\beta}{\alpha} +\frac1\alpha)} \notag\\
&\lesssim h^{1- r'(\frac{d}{\alpha p}- \frac{\beta}{\alpha} +\frac1\alpha) }, \label{SPLIT_BETA_LAST_TIME_STEP}
\end{align}
having used that $t - \tau^h_t \leq h$ and $1- r'(\frac{d}{\alpha p}- \frac{\beta}{\alpha} +\frac1\alpha) >0$. Thus the term $\Delta_{6,2}$ is bounded by 
\begin{align}
	|\Delta_{6,2}(0,x,t,y)|	&\lesssim \bar p_{\alpha} (t,y-x)	h^{\frac1{r'}- (\frac{d}{\alpha p}- \frac{\beta}{\alpha} +\frac1\alpha) } \nonumber\\
	&\lesssim \bar p_{\alpha} (t,y-x)h^{1-\frac 1\alpha-\frac{1}{r}-\frac{d}{\alpha p} + \frac{\beta}{\alpha}} \nonumber \\	
	&\lesssim \bar p_{\alpha} (t,y-x)h^{\frac{\gamma}{\alpha}}h^{- \frac{\beta}{\alpha}} \nonumber \\
	&\lesssim \bar p_\alpha (t,y-x)h^{\frac \gamma\alpha}t^{-\frac{\beta}{\alpha}}.\label{maj-D62-mainthm}
\end{align}
From \eqref{maj-D61-mainthm}, \eqref{maj-D62-mainthm}, we eventually get:
\begin{align}\label{maj-D6-mainthm}
	|\Delta_{6}(0,x,t,y)|&\lesssim \bar p_\alpha (t,y-x)h^{\frac \gamma\alpha}t^{-\frac{\beta}{\alpha}}.
\end{align}

\subsection{Control of the Hölder modulus for the error}\label{subsec-proof-main-thm-holder-mod}
It was seen in the control of the previous term  $\Delta_4 $ that a contribution in $\sup_{s\in (0,T]}\left \|\frac{(\Gamma-\Gamma^h)(0,x,\tau_s^h,\cdot)}{\bar p_\alpha(\tau_s^h,\cdot-x)}\right\|_{\B_{\infty,\infty}^\rho}$, $\rho$ defined in \eqref{CHOICE_FOR_RHO} appeared in the r.h.s. of \eqref{maj-D4-mainthm}. This in particular means that, in order to make a circular/Gronwall type argument, we need to control the corresponding  normalized Besov norm for the error in its final variable through its decomposition in \eqref{DECOUP_ERR}. To this end, we will mainly reproduce the former computations observing that we still have some margin in the time singularities. We will here focus on the $\rho $-Hölder modulus of continuity in the diagonal regime (for the current time considered and the final variable), since otherwise the previous controls on the supremum norm already provide the estimates (see the definition in \eqref{DEF_GH} above). 

\sm{From the overall strategy introduced  at the beginning of Section \ref{SEC_PROOF_THM}, and coherently with the former analysis for the supremum norm which also made the $\B_{\infty,\infty}^\rho $ norm of the normalized error appear}, we will thus now focus on the first term in the r.h.s. of \eqref{THE_DECOUP_FOR_HOLDER_MOD_E} using the error expansion \eqref{DECOUP_ERR}. With a slight abuse of notation from now on, for $y,y'\in \R^d $, s.t. $|y-y'|\le t^{\frac 1\alpha}  $, we will denote for $i\in \{1,\cdots, 6\} $,
$$\Delta_i(0,x,t,y,y'):=\Delta_i(0,x,t,y')-\Delta_i(0,x,t,y'),$$
where the terms $\Delta_i(0,x,t,y),\Delta_i(0,x,t,y') $ are those introduced in \eqref{DECOUP_ERR}. 

\paragraph{Term $\Delta_1 $: first time step.}
    For $\Delta_1$, we rely on the fact that we work on the first time step.  Let us first expand the expectation:
\begin{align*}
	&\Delta_1(0,x,t,y,y')\\ 
	:=&  \int_0^h \int \Big(\Gamma(0,x,s,z)b(s,z)\cdot \big(\nabla_y  p_\alpha (t-s,y-z)-\nabla_y  p_\alpha (t-s,y'-z)\big) \\
	& \qquad\qquad- \Gamma^h(0,x,s,z)\mathfrak b_h(s,x) \cdot \big( \nabla_y p_\alpha (t-s,y-z)-\nabla_y p_{\alpha} (t-s,y'-\textcolor{black}{z})\big)\Big)\d z\d s\\
	=:& \Big(\Delta_{1,1}+\Delta_{1,2}\Big)(0,x,t,y,y').
\end{align*}
For $\Delta_{1,1}$, which involves the distributional $b$, we have to rely on duality inequalities in Besov spaces. Assuming w.l.o.g. that $t>2 h$ so that $(t-s)\asymp t$ and using \eqref{besov-estimate-gamma-and-stable_NO_NORM_Holder} (taking therein $\zeta\in (-\beta,\textcolor{black}{-2\beta}) $), we get
\begin{align}
	&|\Delta_{1,1}(0,x,t,y,y')|\notag\\
	&\quad\lesssim \int_0^h \Vert b (s,\cdot)\Vert_{\B_{p,q}^\beta}  \Vert \Gamma (0,x,s,\cdot)\Big(\nabla_y p_{\alpha} (t-s,y-\cdot)-\nabla_y p_{\alpha} (t-s,y'-\cdot)\Big)\Vert_{\B_{p',q'}^{-\beta}}\d s\nonumber\\
	&\quad\lesssim \bar{p}_{\alpha}(t,y-x)|y-y'|^\rho\int_0^h \Vert b (s,\cdot)\Vert_{\B_{p,q}^\beta} \frac{t^{\frac{\beta}{\alpha}}}{(t-s)^{\frac{1}{\alpha}+\frac \rho\alpha}}\left[\frac{1}{s^{\frac{d}{\alpha p}}} + \frac{1}{(t-s)^{\frac{d}{\alpha p}}}\right]\left[
	\frac{t^{\frac \zeta\alpha}}{(t-s)^{\frac{\zeta}{\alpha}}}+\frac{t^{\frac \zeta\alpha}}{s^{\frac{\zeta}{\alpha}}}\right]\d s\nonumber\\
	&\quad\lesssim \bar{p}_{\alpha}(t,y-x) |y'-y|^\rho\Vert b \Vert_{L^r-\B_{p,q}^\beta} \textcolor{black}{t^{\frac{\beta-1-\rho+\zeta}{\alpha}}\left(\int_0^h\left[\frac{1}{s^{\frac{dr'}{\alpha p}}} + \frac{1}{t^{\frac{dr'}{\alpha p}}}\right]\frac{1}{s^{\frac{\zeta r'}{\alpha}}} 
	\d s\right)^{\frac{1}{r'}}}\nonumber\\
	&\quad\lesssim \bar{p}_{\alpha}(t,y-x)\textcolor{black}{|y'-y|^\rho}\left[h^{1-\frac{1}{r}-\frac{d}{\alpha p}-\frac{\zeta}{\alpha}}t^{\frac{\beta-1-\rho+\zeta}{\alpha}}+h^{1-\frac{1}{r}-\frac{\zeta}{\alpha}}t^{\frac{\beta-1-\rho+\zeta}{\alpha}-\frac{d}{\alpha p}}\right]\nonumber\\
	&\quad\lesssim \bar{p}_{\alpha}(t,y-x) |y'-y|^\rho h^{\frac{\ei{\gamma-\eps}}{\alpha}}t^{-\frac{\ei{\beta+\rho-\eps}}{\alpha}}
	= \bar{p}_{\alpha}(t,y-x) |y'-y|^\rho h^{\frac{\ei{\gamma-\eps}}{\alpha}}t^{\ei{\frac{\eps}{2\alpha}}},\label{maj-D11-mainthm_H}
\end{align}
having used \eqref{eq:Beta1} (with $\aa=\frac\zeta\alpha$ and $\bb=\frac d {\alpha p}$ therein) for the last but one inequality, 
the definition of $\gamma$,
\sm{the inequality $-2\beta-\zeta>0 $}
 and the fact that $h<t$ for the last inequality. Eventually,  the choice of $\rho$ in \eqref{CHOICE_FOR_RHO} gives the last equality.
 
For $\Delta_{1,2}$, using the $L^\infty$ bound \eqref{CTR_PONCTUEL_BH} for $\mathfrak{b}_h$,  \sm{the upper bound \eqref{aronson-gammah} for the density of the scheme},  \sm{\eqref{holder-space-palpha}},  
\ei{together with the approximate convolution property \eqref{APPR_CONV_PROP} and the fact that $\bar p_\alpha (t,y'-x)$ can be replaced by $\bar p_\alpha (t,y-x)$ in the diagonal regime \sm{(see Remark \ref{REM_DIAG_REGIME_DENS})}}, 
we get
\begin{align}
	&|\Delta_{1,2}(0,x,t,y,y')|\notag\\
	\lesssim&  \int_0^h s^{-\frac{d}{\alpha p} + \frac{\beta}{\alpha}} \|b(s,\cdot)\|_{\B_{p,q}^\beta}  \int_{\R^d} \bar p_\alpha (s,z-x)    \frac{|y-y'|^\rho}{(t-s)^{\frac{1+\rho}{\alpha}}}  (\bar p_\alpha (t-s,y-z)+\bar p_\alpha (t-s,y'-z))  \d z\d s\nonumber\\
	\lesssim&\ei{ \Big(\int_0^h s^{(-\frac{d}{\alpha p} + \frac{\beta}{\alpha})r'}\d s \Big)^{\frac {1}{r'}} t^{-\frac {1+\rho}\alpha}  |y'-y|^\rho\|b\|_{L^r-\B_{p,q}^\beta} [ \bar p_{\alpha} (t,y-x) + \bar p_{\alpha} (t,y'-x) ] }\nonumber\\
	\lesssim& \bar p_{\alpha} (t,y-x)\textcolor{black}{|y'-y|^\rho}h^{1-\frac{1}{r}-\frac{d}{\alpha p} + \frac{\beta}{\alpha}} t^{-\frac {1+\rho}\alpha}\nonumber\\
	\lesssim& \bar p_{\alpha} (t,y-x)\textcolor{black}{|y'-y|^\rho}h^{\frac{\ei{\gamma-\eps}}{\alpha}}h^{\ei{\frac\eps\alpha}+\frac{1}{\alpha}- \frac{\beta}{\alpha}}t^{-\frac {1+\rho}\alpha}\lesssim \bar p_\alpha (t,y-x) |y'-y|^\rho h^{\frac \gamma\alpha}t^{\ei{\frac{\eps}{2\alpha}}},\label{maj-D12-mainthm_H}
\end{align} 
 recalling that $h\le t$ \ei{and the choice of $\rho$} for the last inequality. We eventually get from \eqref{maj-D12-mainthm_H} and \eqref{maj-D11-mainthm_H}:
 \begin{equation}
 |\Delta_{1}(0,x,t,y,y')|\lesssim \bar p_\alpha (t,y-x)\textcolor{black}{|y'-y|^\rho}h^{\frac \gamma\alpha}t^{\ei{\frac{\eps}{2\alpha}}}.\label{maj-D1-mainthm_H}
 \end{equation}
 
  \paragraph{Term $\Delta_2 $: time sensitivity of the density of the SDE.}
 Let us turn to $\Delta_2$. Expanding the expectation, using the product rule \eqref{PR}, the duality inequality \eqref{dual-ineq} in Besov spaces, the heat kernel estimate \eqref{holder-time-gamma} and the control \eqref{besov-estimate-stable-derivees_temps_esp_sensi_holder_esp}, we get for $\zeta\ei{\in(-\beta,1]} $,
 \begin{align}
 	&|\Delta_2(0,x,t,y,y')|\notag\\ 
	&\quad = \left| \int_h^{\tau_t^h-h}  \int [\Gamma(0,x,s,z)-\Gamma (0,x,\tau_s^h,z)] b(s,z)\cdot \left(\nabla_y p_{\alpha}(t-s,y-z)-\nabla_y p_{\alpha}(t-s,y'-z)\right)	\d z \d s\right|\notag\\
 	 &\quad\lesssim  \int_h^{\tau_t^h-h}  \Vert b(s,\cdot)\Vert_{\B_{p,q}^\beta}	 \left\Vert \frac{\Gamma(0,x,s,\cdot)-\Gamma (0,x,\tau_s^h,\cdot)}{\bar p_\alpha (s,\cdot-x)}\right\Vert_{\B^{\rho}_{\infty,\infty}}\nonumber \\ & \quad \qquad \qquad \qquad \times\Vert \bar p_\alpha (s,\cdot-x) \big(\nabla_y p_{\alpha}(t-s,y-\cdot)-\nabla_y p_{\alpha}(t-s,y'-\cdot)\big)\Vert_{\B_{p',q'}^{-\beta}} \d s\notag\\
 	 &\quad\lesssim |y-y'|^\rho\bar p_\alpha(t,y-x)\int_h^{\tau_t^h-h}  \Vert b(s,\cdot)\Vert_{\B_{p,q}^\beta} \frac{(s-\tau_s^h)^{\frac{\gamma-\eps}{\alpha}}}{s^{\frac{\gamma-\eps+\rho}{\alpha}}} \frac{t^{\frac{\beta}{\alpha}}}{(t-s)^{\frac{1+\rho}{\alpha} }} \left[\frac{1}{s^{\frac{d}{\alpha p}}} + \frac{1}{(t-s)^{\frac{d}{\alpha p}}}\right]\left[\frac{t^{\frac{\zeta}{\alpha}}}{s^{\frac{\zeta}{\alpha}}}+\frac{t^{\frac{\zeta}{\alpha}}}{(t-s)^{\frac{\zeta}{\alpha}}}\right] \d s \notag\\
 & \ei{ \quad\lesssim |y-y'|^\rho\bar p_\alpha(t,y-x) \Vert b\Vert_{L^r-\B_{p,q}^\beta} h^{\frac{\gamma-\eps}{\alpha}} t^{\frac{\beta+ \zeta}{\alpha}}  } \notag \\
 & \quad \qquad \qquad \qquad \times     \ei{ 
 \left( \int_h^{\tau_t^h-h} \frac{1}{s^{\frac{\gamma-\eps+\rho}{\alpha} r'} } \frac{1}{(t-s)^{\frac{1+\rho}{\alpha}r'}} \left[\frac{1}{s^{\frac{d}{\alpha p}}} + \frac{1}{(t-s)^{\frac{d}{\alpha p}}} \right]^{r'} \left[\frac{1}{s^{\frac{\zeta}{\alpha}}}+\frac{1}{(t-s)^{\frac{\zeta}{\alpha}}}\right]^{r'} \d s 
 \right)^{\frac1{r'} }  
 }\notag.
 \end{align}
The singular integral can be bounded using Lemma \ref{lm:sing_int} with $\aa=\frac{\gamma-\varepsilon+\rho}\alpha, \bb= \frac1\alpha+\textcolor{black}{\frac \rho \alpha}, \cc= \frac d{\alpha p}, \dd=\frac\zeta\alpha$, provided that 
\[
r'(\frac{\gamma-\varepsilon+\rho}\alpha+ \frac d{\alpha p} + \frac\zeta\alpha)<1 
\iff 
\gamma - \varepsilon + \rho +\frac dp +\zeta < \alpha- \frac\alpha r
\iff 
2\beta -1-\varepsilon + \rho +\zeta <0,
\] 
which is satisfied since $\rho,\zeta \in (-\beta, -2\beta)$ and $\beta>-\frac12$, 
and provided that
\[
r'(\frac{1}\alpha+\textcolor{black}{\frac \rho \alpha} +\frac d{\alpha p} + \frac\zeta\alpha)<1 
\iff 
1 +\textcolor{black}{\rho} +\frac dp +\zeta < \alpha- \frac\alpha r
\iff 
1-\alpha +\frac dp +\frac\alpha r <-(\zeta+\textcolor{black}{\rho}) .
\]  
\ei{Recall} \eqref{CHOICE_FOR_RHO} where we set $\rho=-\beta+\frac \varepsilon 2$, 
\ei{so the latter condition is equivalent to} 
$$
 0<\alpha-1-\frac dp-\frac \alpha r+\ei{ \beta -\frac   \varepsilon2 -\zeta \iff 0< \gamma  - \beta - \frac \varepsilon 2-\zeta},$$
having used the definition \eqref{eq:gamma} of $\gamma $.
  \ei{Choosing  $\zeta=-\beta+  \frac\varepsilon2 $, which is allowed, we see that the constraint becomes $0<\gamma-\varepsilon$, hence it is satisfied since    we  assumed $\gamma> \varepsilon  $.} 
Thus  by \eqref{eq:Beta2-5} \ei{from Lemma \ref{lm:sing_int} we get}
\begin{align*}
\Big(\int_{0}^{t} \frac{1}{s^{r'\frac{\gamma-\eps+\rho}{\alpha}}}
 \frac{1}{ (t-s)^{r'\frac{1+ \rho}{\alpha} }}
 \left[\frac{1}{s^{\frac{d}{\alpha p}}} 
 + \frac{1}{(t-s)^{\frac{d}{\alpha p}}}\right]^{r'}\left[\frac{1}{s^{\frac{\zeta}{\alpha}}}+\frac{1}{(t-s)^{\frac{\zeta}{\alpha}}}\right]^{r'} \d s\Big)^{\frac 1{r'}} &\lesssim t^{1-\frac1r - \frac{\gamma-\varepsilon+  \textcolor{black}{2}\rho +1 +\frac dp +\zeta}\alpha} \\
&= t^{-\frac{2\beta}\alpha + \frac{\varepsilon-\textcolor{black}{2}\rho-\zeta}\alpha}, 
\end{align*}
which eventually yields:
\begin{align}
 	|\Delta_2(0,x,t,y,y')|\lesssim \bar p_\alpha(t,y-x)|y-y'|^\rho h^{\frac{\gamma-\eps}{\alpha}} t^{\textcolor{black}{\frac{\eps-\beta-2\rho}{\alpha}}}.\label{maj-D2-mainthm_H}
\end{align}

\paragraph{Term $\Delta_3 $: approximation of the singular drift.}

 Let us turn to $\Delta_3$. Expanding the inner expectation and from the proof of Proposition 2 in \cite{CdRJM22} we derive, using \eqref{besov-estimate-gamma-and-stable_NO_NORM_Holder} with some $\zeta\in (-\beta+\gamma\textcolor{black}{-\varepsilon,-\beta+\gamma })$,
\begin{align*}
	&|\Delta_3(0,x,t,y,y')|\\
	&\quad =\left| \int_h^{\tau_t^h-h} \int\Gamma (0,x,\tau_s^h,z) \left[b(s,z)- \mathfrak{b}_h(s,z)\right]\cdot \Big(\nabla_y  p_\alpha (t-s,y-z)-\nabla_y  p_\alpha (t-s,y'-z)\Big) \d z \d s\right|\\
	&\quad \lesssim \int_h^{\tau_t^h-h}  \|b(s,\cdot)-P_h^\alpha b(s,\cdot)\|_{\B_{p,q}^{\beta-\gamma+\textcolor{black}{\varepsilon}}} \| \Gamma (0,x,\tau_s^h,\cdot)\big(\nabla p_{\alpha}(t-s,y-\cdot)-\nabla p_{\alpha}(t-s,y'-\cdot)\big)\|_{\B_{p',q'}^{-\beta+\gamma\textcolor{black}{-\varepsilon}}}  \d s\\
		&\quad \lesssim h^{\frac{\gamma\textcolor{black}{-\varepsilon}}\alpha}|y'-y|^\rho\int_h^{\tau_t^h-h} \|b(s,\cdot)\|_{\B_{p,q}^\beta}\frac{\bar{p}_{\alpha} (t,x-y)}{(t-s)^{\frac{1}{\alpha}+\frac \rho\alpha}} t^{\frac{\beta-\gamma\textcolor{black}{+\varepsilon}}{\alpha}}\left[ \frac{1}{(t-s)^{\frac{d }{\alpha  p}}} +\frac{1}{s^{\frac{d }{\alpha  p}}}\right] \left[ \frac{t^{\frac{\zeta}{\alpha}}}{(t-s)^{\frac{\zeta }{\alpha}}} + \frac{t^{\frac{\zeta}{\alpha}}}{s^{\frac{\zeta }{\alpha}}} \right] \d s\\
		&\quad \lesssim h^{\frac{\gamma\textcolor{black}{-\varepsilon}}\alpha}|y'-y|^\rho t^{\frac{\beta-\gamma\textcolor{black}{+\varepsilon}+\zeta}{\alpha}}\|b\|_{L^r-\B_{p,q}^\beta}\bar{p}_{\alpha} (t,x-y)\\
		&\quad\quad\times\Bigg(\int_h^{\tau_t^h-h} \frac{1}{(t-s)^{r'(\frac{1+\rho}{\alpha})}} \left[ \frac{1}{(t-s)^{\frac{d }{\alpha  p}}} +\frac{1}{s^{\frac{d }{\alpha  p}}}\right]^{r'} \left[ \frac{1}{(t-s)^{\frac{\zeta }{\alpha}}} + \frac{1}{s^{\frac{\zeta }{\alpha}}} \right]^{r'} \d s\Bigg)^{\frac 1{r'}}.
\end{align*}
Notice that, this time, since $\zeta>-\beta+\gamma\textcolor{black}{-\varepsilon}$, the previous integral is  indeed convergent (see Lemma \ref{lm:sing_int} which applies here with \ei{$\aa=0,\bb=\frac{1+\rho}{\alpha},\cc=\frac{d}{\alpha p},\dd=\frac\zeta\alpha$}). Indeed, the highest singularity appears around $s=t$.
For the corresponding  exponent to be integrable one needs:
$$r'(\frac{1+\rho}\alpha+\frac{d}{\alpha p}+\frac \zeta\alpha)<1\iff 0<\alpha-1-\frac dp-\frac \alpha r-\rho-\zeta\iff 0<\gamma-2\beta-\rho-\zeta=-\beta+\gamma-\zeta-\frac \eps 2,$$
recalling \eqref{CHOICE_FOR_RHO} for the last equality. 
Taking \ei{$\zeta = -\beta+\gamma- \frac{3\eps}{4}$} thus \ei{guarantees the integrability condition as well as $\zeta\in(-\beta+\gamma-\eps, -\beta+\gamma)$}. We get
\begin{align}
	&|\Delta_3(0,x,t,y,y')| \lesssim h^{\frac{\gamma-\varepsilon}\alpha}|y'-y|^\rho t^{\frac{\beta-\gamma+\varepsilon +\zeta}{\alpha}} \bar p_\alpha(t,y-x)  \ei{t^{1-\frac1r- \frac{1+\rho}{\alpha} - \frac{d}{\alpha p} -\frac\zeta\alpha}} \lesssim h^{\frac{\gamma-\eps}\alpha}|y'-y|^\rho t^{\frac{\eps}{\textcolor{black}{4}\alpha}}\textcolor{black}{\bar p_\alpha(t,y-x)}.\label{maj-D3-mainthm_H}
\end{align}

\paragraph{Term $\Delta_4 $: Gronwall-like or circular type argument.}
We now need to control the Hölder norm of the  term associated with the Gronwall or circular type argument. 
\ei{Proceeding similarly as for the control of the supremum norm of $\Delta_4$ but using \eqref{besov-estimate-stable-derivees_temps_esp_sensi_holder_esp} in place of \eqref{besov-estimate-stable} with $\zeta,\rho\in (-\beta,1]$ we have}
\begin{align}
&|\Delta_4(0,x,t,y,y')|\notag\\
:=&\left|\int_h^{\tau_t^h-h} \bigg\{\E_{0,x} \bigg[ \mathfrak b_h
	(s,X_{\tau_s^h})\cdot \nabla_y p_\alpha (t-s,y-X_{\tau_s^h})
	-
	\mathfrak b_h(s,X_{\tau_s^h}^h)\cdot \nabla_y  p_\alpha (t-s,y-X_{\tau_s^h}^h) \bigg] \right.\notag\\
	&\left. -\E_{0,x} \bigg[\mathfrak b_h
	(s,X_{\tau_s^h})\cdot \nabla_y p_\alpha (t-s,y'-X_{\tau_s^h})
	-
	\mathfrak b_h(s,X_{\tau_s^h}^h)\cdot \nabla_y  p_\alpha (t-s,y'-X_{\tau_s^h}^h)\bigg]\bigg\}\d s\right|\notag\\
	=&\left|\int_h^{\tau_t^h-h}\int_{\R^d } (\Gamma-\Gamma^h)(0,x,\tau_s^h,z) \mathfrak b_h(s,z) \big(\nabla_y  p_\alpha (t-s,y-z)-\nabla_y  p_\alpha (t-s,y'-z)\big)  \d z \d s\right|\notag\\
	\le &\int_h^{\tau_t^h-h} \|\mathfrak b_h(s,\cdot)\|_{\B_{p,q}^\beta}\left \|\frac{(\Gamma-\Gamma^h)(0,x,\tau_s^h,\cdot)}{\bar p_\alpha(\tau_s^h,\cdot-x)}\right\|_{\B_{\infty,\infty}^\rho}\notag\\
	&\qquad\qquad\times\|\bar p_\alpha(\tau_s^h,\cdot-x) \big(\nabla p_{\alpha}(t-s,y-\cdot)-\nabla p_{\alpha}(t-s,y'-\cdot)\big)\|_{\B_{p',q'}^{-\beta}}\d s\notag\\
	\le &|y-y'|^\rho\sup_{s\in (h,T]}g_{h,\rho}(s)
	\bar p_\alpha(t,y-x)\notag\\
	&\hspace*{.5cm}\times\int_0^t \|b(s,\cdot)\|_{\B_{p,q}^\beta}\frac{1}{s^{\frac \rho\alpha }}\frac{t^{\frac{\beta}{\alpha}}}{(t-s)^{\frac{1}{\alpha}+\frac\rho\alpha}}\left[\frac{1}{s^{\frac{d}{\alpha p}}} + \frac{1}{(t-s)^{\frac{d}{\alpha p}}}\right]\left[\frac{t^{\frac{\zeta}{\alpha}}}{s^{\frac{\zeta}{\alpha}}}+\frac{t^{\frac{\zeta}{\alpha}}}{(t-s)^{\frac{\zeta}{\alpha}}}\right] \d s\notag\\
	\lesssim & |y-y'|^\rho\sup_{s\in (h,T]}g_{h,\rho}(s)
	t^{\frac{\gamma-\beta-2\rho}{\alpha}}\bar p_\alpha(t,y-x)\notag\\
	\lesssim & |y-y'|^\rho\sup_{s\in (h,T]}g_{h,\rho}(s)
	t^{\frac{\textcolor{black}{\gamma+\beta-\eps}}{\alpha}}\bar p_\alpha(t,y-x),\label{maj-D4-mainthm_H}
\end{align}keeping in mind the definition \eqref{DEF_GH} 
and having used the \ei{first} inequality  of Lemma \ref{lm:sing_int} as for the former term $\Delta_2 $ of this section but here  with $\aa=\frac \rho\alpha $, recalling as well the choice for $\rho $ in \eqref{CHOICE_FOR_RHO}.

\paragraph{Term $\Delta_5 $: spatial sensitivities of the driving noise.}
Let us now turn to $\Delta_5 $. Using the harmonicity of the stable heat kernel (or martingale property of the driving noise) we have
\begin{align*}
&\Delta_5(0,t,x,y,y')\\
=&\int_h^{\tau_t^h-h} \left\{\E_{0,x} \left[\mathfrak{b}_h(s,X_{\tau_s^h}^h)\cdot \left( \nabla_y p_\alpha(t-s,y-X_{\tau_s^h}^h)-\nabla_y  p_\alpha (t-s,y-X_s^h)\right)\right]\right.\\
&\qquad\qquad-\left. \E_{0,x} \left[\mathfrak{b}_h(s,X_{\tau_s^h}^h)\cdot \left( \nabla_y p_\alpha(t-s,y'-X_{\tau_s^h}^h)-\nabla_y  p_\alpha (t-s,y'-X_s^h)\right)\right]\right\}\d s\\
=&\int_h^{\tau_t^h-h} \E_{0,x} \left[\mathfrak{b}_h(s,X_{\tau_s^h}^h)\cdot \left( \nabla_y p_\alpha(t-s,y-X_{\tau_s^h}^h)
-\nabla_y  p_\alpha \Big(t-\tau_s^h,y-(X_{\tau_s^h}^h+\int_{\tau_s^h}^s\mathfrak{b}_h(u,X_{\tau_s^h}^h)\d u)\Big)\right.\right.\\
&\left.\left.  -\Big(\nabla_y p_\alpha(t-s,y'-X_{\tau_s^h}^h)
-\nabla_y  p_\alpha \Big(t-\tau_s^h,y'-(X_{\tau_s^h}^h+\int_{\tau_s^h}^s\mathfrak{b}_h(u,X_{\tau_s^h}^h)\d u)\Big)\Big)\right)\right]\d s.
\end{align*}
Write now,
\begin{align*}
&|\Delta_5(0,x,t,y,y')|\\
= &\Bigg|\int_{h}^{\tau_t^h-h} \int_{\R^d} \Gamma^h(0,x,\tau_s^h,z)\mathfrak{b}_h(s,z)\cdot \left( \nabla_y p_\alpha(t-s,y-z) 
-\nabla_y  p_\alpha (t-\tau_s^h,y-(z+\int_{\tau_s^h}^s\mathfrak{b}_h(u,z)\d u))\right.\\
&\qquad \qquad \left. -\Big(\nabla_y p_\alpha(t-s,y'-z) 
-\nabla_y  p_\alpha (t-\tau_s^h,y'-(z+\int_{\tau_s^h}^s\mathfrak{b}_h(u,z)\d u))\Big)
\right) \d z \d s\Bigg|\\
\le& \Bigg|\int_{h}^{\tau_t^h-h} \int_{\R^d} \Gamma^h(0,x,\tau_s^h,z)\mathfrak{b}_h(s,z)\cdot \left( \nabla_y p_\alpha(t-s,y-z) 
-\nabla_y  p_\alpha (t-\tau_s^h,y-z)\right.\\
&\left.\qquad\qquad-\Big( \nabla_y p_\alpha(t-s,y'-z) 
-\nabla_y  p_\alpha (t-\tau_s^h,y'-z) \Big)\right) \d z  \d s\Bigg|\\
&+\Bigg|\int_{h}^{\tau_t^h-h} \int_{\R^d} \Gamma^h(0,x,\tau_s^h,z)\mathfrak{b}_h(r,z)\cdot \left( \nabla_y p_\alpha(t-\tau_s^h,y-z) \phantom{\int}\right.\notag \\
&\hspace*{.5cm}\left. -\nabla_y  p_\alpha (t-\tau_s^h,y-(z+\int_{\tau_s^h}^s\mathfrak{b}_h(u,z)\d u))\right.\\
&\hspace*{.5cm}\left. -\Big(\nabla_y p_\alpha(t-\tau_s^h,y'-z)-\nabla_y  p_\alpha (t-\tau_s^h,y'-(z+\int_{\tau_s^h}^s\mathfrak{b}_h(u,z)\d u))\Big)\right) 
\d z  \d s\Bigg|\\
=:&|\Delta_{51}(0,x,t,y,y')|+|\Delta_{52}(0,x,t,y,y')|.
\end{align*}
We have:
\begin{align*}
&|\Delta_{51}(0,x,t,y,y')|\\
\lesssim& \int_h^{\tau_t^h-h}\left| \int_{\R^d} \Gamma^h(0,x,\tau_s^h,z)\mathfrak{b}_h(s,z)\cdot \left( \nabla_y p_\alpha(t-\tau_s^h,y-z)-\nabla_y  p_\alpha (t-s,y-z)\right.\right.\\
&\left.\left. \qquad\qquad - \big(\nabla_y p_\alpha(t-\tau_s^h,y'-z)-\nabla_y  p_\alpha (t-s,y'-z)\big)\right) \d z \right| \d s\\
\lesssim& h \int_h^{\tau_t^h-h}\left| \int_0^1 \int_{\R^d} \Gamma^h(0,x,\tau_s^h,z)\mathfrak{b}_h(s,z)\cdot \left( \partial_t\nabla_y p_\alpha(t-(\tau_s^h+\lambda(s-\tau_s^h)),y-z)\right.\right.\\
&\qquad\phantom{\int_0^1d\lambda}\left.\left.-\partial_t\nabla_y p_\alpha(t-(\tau_s^h+\lambda(s-\tau_s^h)),y'-z)\right)\d z   \d \lambda \right|\d s \\
\lesssim& h\int_h^{\tau_t^h-h} \d s\|\mathfrak b_h(s,\cdot)\|_{\B_{p,q}^\beta} \\ & \times  \int_0^1\d \lambda\left\| \Gamma^h (0,x,\tau_s^h,\cdot)\left( \partial_t\nabla_y p_\alpha(t-(\tau_s^h+\lambda(s-\tau_s^h)),y-\cdot)-\partial_t\nabla_y  p_\alpha (t-(\tau_s^h+\lambda(s-\tau_s^h)),y'-\cdot)\right)\right\|_{\B_{p',q'}^{-\beta}},
\end{align*}
\sm{using \eqref{dual-ineq} for the last inequality}.

\ei{Using   \eqref{CTR_BESOV_BH} and then} \eqref{besov-estimate-gammah-stable-derivees_temps_esp_sensi_holder_esp} \ei{with $\theta=1, \zeta \in (-\beta, -\beta+\gamma) $ and $\sm{r=\tau_s^h+\lambda (s-\tau_s^h)}$ therein,  together with the fact that $h\leq h^{\frac{\gamma-\eps}\alpha} (t-s)^{\frac{\alpha-\gamma+\eps}\alpha}$} we get 
\begin{align}
		&|\Delta_{51}(0,x,t,y,y')|\nonumber\\		
		&\quad \lesssim|y-y'|^\rho\int_{h}^{\tau_t^h-h}\|\mathfrak b_h(s,\cdot)\|_{\B_{p,q}^\beta} h^{\frac{\ei{\gamma-\eps}}{\alpha}}\frac{\bar p_\alpha(t,y-x) }{(t-s)^{\frac{\ei{1+\gamma+\rho-\eps}}{\alpha}}} t^{\frac{\beta}{\alpha}}\left[\frac{1}{s^{\frac{d }{\alpha  p}}}+ \frac{1}{(t-s)^{\frac{d }{\alpha  p}}} \right] \left[\frac{t^{\frac{\zeta}{\alpha}}}{s^{\frac{\zeta }{\alpha}}}+\frac{t^{\frac{\zeta}{\alpha}}}{(t-s)^{\frac{\zeta }{\alpha}}}  \right]  \d s\nonumber\\
		&\quad \lesssim |y-y'|^\rho\bar p_\alpha(t,y-x)h^{\frac{\ei{\gamma-\eps}}{\alpha}}t^{\frac \beta\alpha}\left(\int_{h}^{\tau_t^h-h} \frac{1}{(t-s)^{\frac{\ei{1+\gamma+\rho-\eps}}{\alpha} r'}}\left[\frac{1}{s^{\frac{d }{\alpha  p}}}+ \frac{1}{(t-s)^{\frac{d }{\alpha  p}}} \right]^{r'} \left[ \frac{t^{\frac{\zeta}{\alpha}}}{s^{\frac{\zeta }{\alpha}}}+\frac{t^{\frac{\zeta}{\alpha}}}{(t-s)^{\frac{\zeta }{\alpha}}}  \right]^{r'}\d s\right)^{\frac 1{r'}}.\nonumber
\end{align}
\ei{The latter singular integral is treated by   Lemma \ref{lm:sing_int} with $\aa=0,\bb=\frac{1+\gamma+\rho-\eps}{\alpha},\cc=\frac{d}{\alpha p},\dd=\frac{\zeta}{\alpha} $ thanks to the fact that the integrability condition writes 
\[
r'( \frac{1+\gamma+\rho-\eps}{\alpha}+\frac{d}{\alpha p}+\frac{\zeta}{\alpha}) <1 \iff1 - \frac1r - \frac{1+\gamma+\rho-\eps}{\alpha}-\frac{d}{\alpha p}-\frac{\zeta}{\alpha}\sm{=\frac{-2\beta+\eps-\rho-\zeta}\alpha}>0 \iff \frac\eps4>0,
\]
having chosen $\zeta= -\beta+\frac\eps4$ and recalling the choice \eqref{CHOICE_FOR_RHO} for $\rho$,
 \begin{align}\label{maj-D51_H}
	\Delta_{51} 	&   \lesssim |y-y'|^\rho \bar p_\alpha(t,y-x) h^{\frac{\gamma-\eps}{\alpha} } t^{ \frac{\beta+\zeta}{\alpha}} t^{1-\frac1r -\frac{1+\gamma+\rho-\eps}{\alpha} -\frac{ d}{\alpha p} - \frac \zeta \alpha}
\lesssim |y-y'|^\rho\bar p_\alpha(t,y-x)h^{\frac{\gamma-\eps}{\alpha}} t^{\frac{\eps}{2\alpha}}.
\end{align}
}
On the other hand, using \ei{\eqref{dual-ineq}, then} \eqref{CTR_BESOV_BH} and \eqref{besov-estimate-gammah-stable-PERTURB_DRIFT_SENSI_HOLDER} \ei{with $\rho=-\beta +\frac\eps2$ as in \eqref{CHOICE_FOR_RHO} and $\zeta = -\beta+\frac\eps2 \in (-\beta, -\beta+\gamma)$}, we have
\begin{align}
&|\Delta_{52}(0,x,t,y,y')|\nonumber\\
&\quad \lesssim \int_{h}^{\tau_t^h}  \|\mathfrak b_h(s,\cdot)\|_{\B_{p,q}^\beta} \int_0^1 \left\| \Gamma^h(0,x,\tau_s^h,\cdot)\Big(\nabla_y^2 p_\alpha (t-\tau_s^h,y-(\cdot+\lambda\int_{\tau_s^h}^s\mathfrak{b}_h(u,\cdot)\d u))\right.\nonumber\\
&\left.\qquad\qquad\qquad\qquad\qquad\qquad - \nabla_y^2  p_\alpha (t-\tau_s^h,y'-(\cdot+\lambda\int_{\tau_s^h}^s\mathfrak{b}_h(u,\cdot)\d u))\Big)\int_{\tau_s^h}^s\mathfrak{b}_h(u,\cdot)\d u\right \|_{\B_{p',q'}^{-\beta}} \d \lambda \d s\nonumber\\
&\quad \lesssim |y-y'|^\rho\bar p_\alpha(t,y-x) t^{\frac{\beta}{\alpha}}\int_{h}^{\tau_t^h-h}  \frac{\|b(s,\cdot)\|_{\B_{p,q}^\beta}h^{\frac{\gamma-\beta}{\alpha}}}{(t-\tau_s^h)^{\frac{1+\rho}{\alpha}}} \left[\frac{1}{(\tau_s^h)^{\frac{d}{\alpha p}}} + \frac{1}{(t-\tau_s^h)^{\frac{d}{\alpha p}}}\right]\left[\frac{t^{\frac \zeta\alpha}}{(\tau_s^h)^{\frac{\zeta}{\alpha }}} + \frac{t^{\frac\zeta\alpha}}{(t-\tau_s^h)^{\frac{\zeta}{\alpha }}}\right]\d s\notag\\
&\quad \lesssim  |y-y'|^\rho h^{\frac{\gamma-\beta}{\alpha}} \bar p_\alpha(t,y-x) \ei{ t^{\frac{\beta+\zeta}{\alpha}} t^{1-\frac1r-\frac{1+\rho}{\alpha} -\frac{d}{\alpha p} - \frac\zeta\alpha } } \notag\\
& \quad \lesssim  |y-y'|^\rho \bar p_\alpha(t,y-x) \ei{h^{\frac{\gamma-\beta}{\alpha}} t^{\frac{\gamma- \frac{\eps}{2}}{\alpha}} } 
\lesssim  |y-y'|^\rho \bar p_\alpha(t,y-x) \ei{h^{\frac{\gamma-\eps}{\alpha}} t^{\frac{\eps}{2\alpha}} }, \label{maj-D52_H}
\end{align}
having used \ei{H\"older inequality and then} Lemma \ref{lm:sing_int} with $\aa=0,\bb=\frac{1+\rho}{\alpha},\cc=\frac{d}{\alpha p},\dd=\frac{\zeta}{\alpha} $, \ei{which is allowed since the integrability condition holds (it is less resctrictive than the one for the term $\Delta_{51}$), and recalling that we choose $\gamma >\eps$ for the last inequality.
}
From \eqref{maj-D51_H} and \eqref{maj-D52_H} we thus derive:
\begin{align}\label{maj-D5-mainthm_H}
|\Delta_5(0,x,t,y,y')|\lesssim |y-y'|^\rho \bar p_\alpha(t,y-x) h^{\frac{\gamma-\varepsilon}{\alpha}} t^{\frac{\eps}{2\alpha}}.
\end{align}

\paragraph{Term $\Delta_6 $: last time steps}
This term is once again handled very much like $\Delta_1$, in the sense that the smallness will come from each contribution and not from sensitivities. Namely,
\begin{align*}
&|\Delta_6(0,x,t,y,y')|\\
=&\left|\int_{\tau_t^h-h}^t \E_{0,x} \left[ b(s,X_s)\cdot \Big(\nabla_y p_\alpha (t-s,y-X_s)-\nabla_y p_\alpha (t-s,y'-X_s)\Big)\right]\right.\\
&\left.-\E_{0,x} \left[\mathfrak{b}_h(s,X_{\tau_s^h}^h)\cdot\Big(\nabla_y p_\alpha (t-s,y-X_s^h)-\nabla_y p_\alpha (t-s,y'-X_s^h)\Big)  \right] \d s\right|\\
\le&\left|\int_{\tau_t^h-h}^t \E_{0,x} \left[ b(s,X_s)\cdot \Big(\nabla_y p_\alpha (t-s,y-X_s)-\nabla_y p_\alpha (t-s,y'-X_s)\Big)\right]\d s \right|\\
&+\left|\int_{\tau_t^h-h}^t\E_{0,x}\left[\mathfrak{b}_h(s,X_{\tau_s^h}^h)\cdot\Big(\nabla_y p_\alpha (t-s,y-X_s^h)-\nabla_y p_\alpha (t-s,y'-X_s^h)\Big)  \right] \d s\right|\\
=:&\Big(|\Delta_{6,1}|+|\Delta_{6,2}|\Big)(0,x,t,y,y').
\end{align*}
\ei{As done for the bounds of term $\Delta_6$ in the sup-norm, we assume w.l.o.g\ that $t>3h$, so for $\tau_t^h -h< s<t$ so that $t-(\tau_t^h-h) \asymp h$ (since $t-(\tau_t^h-h)  \in (h,2h)$). We rely also on the duality inequality \eqref{dual-ineq}.}
We use \eqref{besov-estimate-gamma-and-stable_NO_NORM_Holder} (taking therein $\zeta\in (-\beta,\ei{-\beta+\gamma), j=0, k=1}$ \ei{and $\rho = -\beta+\frac\eps2$ as defined in \eqref{CHOICE_FOR_RHO}), and proceed similarly as for the control of the sup norm of $\Delta_{6,1}$} to get
\begin{align}
	&|\Delta_{6,1}(0,x,t,y,y')|\notag\\
	&\quad \lesssim\int_{\tau_t^h-h}^t \Vert b (s,\cdot)\Vert_{\B_{p,q}^\beta} \Vert \Gamma (0,x,s,\cdot)\Big(\nabla_y p_{\alpha} (t-s,y-\cdot)-\nabla_y p_{\alpha} (t-s,y'-\cdot)\Big)\Vert_{\B_{p',q'}^{-\beta}}\d s\nonumber\\
	&\quad \lesssim|y-y'|^\rho\bar{p}_{\alpha}(t,y-x)\int_{\tau_t^h-h}^t \Vert b (s,\cdot)\Vert_{\B_{p,q}^\beta} \frac{t^{\frac{\beta}{\alpha}}}{(t-s)^{\frac{1+\rho}{\alpha}}}\left[\frac{1}{s^{\frac{d}{\alpha p}}} + \frac{1}{(t-s)^{\frac{d}{\alpha p}}}\right]\left[
	\frac{t^{\frac{\zeta}{\alpha}}}{s^{\frac{\zeta}{\alpha}}}+\frac{t^{\frac{\zeta}{\alpha}}}{(t-s)^{\frac{\zeta}{\alpha}}}\right]\d s\nonumber\\
	&\quad \lesssim|y-y'|^\rho \bar{p}_{\alpha}(t,y-x) \Vert b \Vert_{L^r-\B_{p,q}^\beta} t^{\frac{\beta+\zeta}\alpha}\left(\int_{\tau_t^h-h}^t \ei{\frac{1}{(t-s)^{\frac{1+\rho+\zeta}{\alpha} r'} }}\left[\frac{1}{t^{\frac{d}{\alpha p}}} + \frac{1}{(t-s)^{\frac{d}{\alpha p}}}\right]^{\textcolor{black}{r'}}
	\d s\right)^{\frac{1}{r'}}\nonumber\\
	&\quad \lesssim|y-y'|^\rho\bar{p}_{\alpha}(t,y-x) \ei{ t^{\frac{\beta+\zeta}\alpha}} h^{1-\frac{1}{r}} \ei{\left( h^{-\frac{1+\rho+\zeta}{\alpha} - \frac{ d} {\alpha p}}+  t^{\frac{d}{\alpha p}} h^{-\frac{1+\rho+\zeta}{\alpha}} \right)}\nonumber\\
		&\quad \lesssim \textcolor{black}{|y-y'|^\rho\bar{p}_{\alpha}(t,y-x)  t^{\frac{\beta+\zeta}{\alpha}} h^{\frac{\gamma-2\beta-(\rho+\zeta)}{\alpha}} }\nonumber\\
	&\quad \lesssim |y-y'|^\rho\bar{p}_{\alpha}(t,y-x) h^{\frac{\gamma-\varepsilon}{\alpha}}t^{\frac{\eps}{2\alpha}},
	\label{maj-D61-mainthm_H}
\end{align}
having used the second bullet point of Lemma \ref{lm:sing_int} (equation \eqref{eq:Beta6}) with $\aa=\frac{1+\textcolor{black}{\rho}+\zeta}\alpha, \bb= \frac d {\alpha p}$ and $v= \tau_t^h-h$. Notice that the assumption $r'(a+b)<1$ \ei{(using the definition  \eqref{CHOICE_FOR_RHO} of $\rho$)}  is equivalent to:
\begin{align*}
\frac{1+\rho+\zeta + \frac dp}\alpha <1-\frac1r &\iff 1+\textcolor{black}{\rho}+\zeta+\frac dp < \alpha - \frac \alpha r \iff \zeta+\rho< \alpha-\frac\alpha r  -\frac dp - 1\\
& \iff \zeta<  \alpha-\frac\alpha r  -\frac dp - 1 +\beta -\frac\eps2
\end{align*} 
which is satisfied  e.g.\ for the choice $\zeta=\rho=-\beta+\frac \eps2 $ as in  since the above condition then rewrites $\gamma>\eps $ (which as been assumed for the whole analysis).

\ei{For $\Delta_{6,2}$ we proceed  similarly as for $\Delta_{6,1}$ and write using the $L^\infty$ bound of $\mathfrak{b}_h$, \eqref{CTR_PONCTUEL_BH}, the fact that $\Gamma^h\lesssim p_\alpha$ (see \eqref{aronson-gammah}) and  \eqref{holder-space-palpha} (with $\delta=0,\zeta=1, \theta=\rho $):
}
\begin{align*}
|\Delta_{6,2}(0,x,t,y,y')|=&\left|\int_{\tau_t^h-h}^t\E_{0,x}\left[\mathfrak{b}_h(s,X_{\tau_s^h}^h)\cdot\Big(\nabla_y p_\alpha (t-s,y'-X_s^h)-\nabla_y p_\alpha (t-s,y-X_s^h)\Big)  \right] \d s\right|\\
=&\left|\int_{\tau_t^h-h}^t\E_{0,x}\left[\mathfrak{b}_h(s,X_{\tau_s^h}^h)\cdot\Big(\nabla_{y'} p_\alpha (t-\tau_s^h,y'-(X_{\tau_s^h}^h+\int_{\tau_s^h}^s\mathfrak{b}_h(r,X_{\tau_s^h}^h)\d r)\right.\right.\\
&-\left.\left.\nabla_y p_\alpha (t-\tau_s^h,y-(X_{\tau_s^h}^h+\int_{\tau_s^h}^s\mathfrak{b}_h(r,X_{\tau_s^h}^h)\d r)\Big)  \right] \d s\right|\\
\le& \int_{\tau_t^h-h}^t\int_{\R^d} \Gamma^h(0,x,\tau_s^h,z)|\mathfrak{b}_h(s,z)|\Big|\nabla_{y'} p_\alpha (t-\tau_s^h,y'-(z+\int_{\tau_s^h}^s\mathfrak{b}_h(r,z)\d r))\\
&-\nabla_y p_\alpha (t-\tau_s^h,y-(z+\int_{\tau_s^h}^s\mathfrak{b}_h(r,z)\d r))\Big|  \d z \d s\\
\le &\int_{\tau_t^h-h}^t\int_{\R^d} \bar p_\alpha(\tau_s^h,z-x)\|b(s,\cdot)\|_{\B_{p,q}^\beta}(s-\tau_s^h)^{-\frac{d}{\alpha p} + \frac{\beta}{\alpha}}\frac{|y-y'|^\rho}{(t-s)^{\frac{1+\rho}{\alpha}}}\\
&\times \Big(\bar p_\alpha(t-\tau_s^h,y'-(z+\int_{\tau_s^h}^s\mathfrak{b}_h(r,z)\d r)) +\bar p_\alpha(t-\tau_s^h,y-(z+\int_{\tau_s^h}^s\mathfrak{b}_h(r,z)\d r))\Big).
\end{align*}
From Lemma \ref{Lemma_TRANS_SCHEME} we eventually derive:

{\color{black}
\begin{align}&|\Delta_{6,2}(0,x,t,y,y')|\notag\\
	\lesssim&  |y-y'|^\rho\int_{\tau_t^h-h}^t \int \frac{1}{(t-s)^{\frac{1}{\alpha}+\frac \rho\alpha}}\|b(s,\cdot)\|_{\B_{p,q}^\beta}(s-\tau_s^h)^{-\frac{d}{\alpha p} + \frac{\beta}{\alpha}}\bar p_\alpha (s,z-x) (\bar p_\alpha (t-s,y-z)+\bar p_\alpha (t-s,y'-z))  \d z\d s\nonumber\\
	\lesssim& |y-y'|^\rho\bar p_{\alpha} (t,y-x)\left(\int_{\tau_t^h-h}^t \frac{1}{(s-\tau_s^h)^{r'(\frac{d}{\alpha p} - \frac{\beta}{\alpha})}(t-s)^{r'\big(\frac{1}{\alpha}+\frac \rho\alpha\big)}} 
	\d s\right)^{\frac 1{r'}}\notag\\
	\lesssim& |y-y'|^\rho\bar p_{\alpha} (t,y-x) h^{\frac{\gamma-(\rho+\beta)}{\alpha}}\lesssim |y-y'|^\rho\bar p_{\alpha} (t,y-x) h^{\frac{\gamma-\varepsilon}{\alpha}}t^{\frac{\eps}{2\alpha}}=|y-y'|^\rho\bar p_{\alpha} (t,y-x) h^{\frac{\gamma-\varepsilon}{\alpha}}t^{ \frac{\beta+\rho}{\alpha}},
	\label{maj-D62-mainthm_H}
\end{align}
where we used the approximate convolution property \eqref{APPR_CONV_PROP} as well as the fact that the diagonal regime holds, i.e. $|y'-y|\le t^{1/\alpha} $, for the second inequality (see Remark \ref{REM_DIAG_REGIME_DENS}). The last inequality follows by direct integration of the rescaled $\beta$-function following the time splitting already detailed in \eqref{SPLIT_BETA_LAST_TIME_STEP} as well as the specific choice of $\rho=-\beta+\frac \eps2 $ in \eqref{CHOICE_FOR_RHO}.
}

From \eqref{maj-D61-mainthm_H}, \eqref{maj-D62-mainthm_H}, we eventually get:
\begin{align}\label{maj-D6-mainthm_H}
	|\Delta_{6}(0,x,t,y,y')|&\lesssim |y-y'|^\rho\bar p_\alpha (t,y-x)h^{\frac {\gamma-\varepsilon}\alpha }t^{\ei{\frac{\eps}{2\alpha} } }.
\end{align}

\subsection{Closing the circular argument}\label{sec:circular}

{\color{black}
Putting together the $L^\infty$-controls \eqref{maj-D1-mainthm}, \eqref{maj-D2-mainthm}, \eqref{maj-D3-mainthm}, \eqref{maj-D4-mainthm}, \eqref{maj-D5-mainthm}, \eqref{maj-D6-mainthm} we derive:

\begin{equation}\label{PREAL_BD_THM}
\frac{|(\Gamma-\Gamma^h)(0,t,x,y)|}{\bar p_\alpha(t,y-x)}\lesssim h^{\frac{\gamma-\eps}{\alpha}}t^{\frac{\eps}{2\alpha}} +\sup_{s\in (h,T]}g_{h,\rho}(s)t^{\frac{\gamma-\frac{\eps}{2}}{\alpha}}.
\end{equation}
On the other hand, putting together the H\"older controls \eqref{maj-D1-mainthm_H}, \eqref{maj-D2-mainthm_H}, \eqref{maj-D3-mainthm_H}, \eqref{maj-D4-mainthm_H}, \eqref{maj-D5-mainthm_H}, \eqref{maj-D6-mainthm_H} we derive
\begin{align*}
&\frac{\left|(\Gamma-\Gamma^h)(0,x,t,y)-(\Gamma-\Gamma^h)(0,x,t,y')\right|}{\bar p_\alpha(t,y-x)|y-y'|^\rho} \lesssim h^{\frac{\gamma-\varepsilon}\alpha}  t^{\frac{\eps}{4 \alpha}}+  h^{\frac{\gamma-\varepsilon}\alpha}  t^{\frac{\beta}{\alpha}}+\sup_{s\in (h,T]} g_{h,\rho}(s) t^{\frac{\gamma+\beta-\eps}{\alpha}},
\end{align*}
}
which in turn yields from the choice of $\rho=-\beta+\frac \eps2 $ in \eqref{CHOICE_FOR_RHO},
\begin{align}
&t^{\frac{\rho}{\alpha}}\frac{\left|(\Gamma-\Gamma^h)(0,x,t,y)-(\Gamma-\Gamma^h)(0,x,t,y')\right|}{\bar p_\alpha(t,y-x)|y-y'|^\rho} \lesssim h^{\frac{\gamma-\varepsilon}\alpha}\textcolor{black}{ t^{\frac{\eps}{\textcolor{black}{4}\alpha}}}+\sup_{s\in (h,T]} g_{h,\rho}(s) t^{\frac{\gamma-\frac{\eps}{2}}{\alpha}}.\label{PREAL_CIRCU}
\end{align}
We can now use  \eqref{PREAL_BD_THM} and \eqref{PREAL_CIRCU} to derive, 
\begin{align*}
&\sup_{y\in \R^d}\frac{|\Gamma(0,x,t,y)-\Gamma^h(0,x,t,y)|}{\bar p_\alpha(t,y-x)}+t^{\frac \rho\alpha}\sup_{(y,y')\in (\R^d)^2}\frac{\left|(\Gamma-\Gamma^h)(0,x,t,y)-(\Gamma-\Gamma^h)(0,x,t,y')\right|}{\bar p_\alpha(t,y-x)|y-y'|^\rho}\notag\\
&\quad \lesssim h^{\frac{\gamma-\varepsilon}\alpha} +\sup_{s\in (h,T]} g_{h,\rho}(s)t^{\frac{\gamma-\frac{\eps}{2}}{\alpha}},
\end{align*}
which rewrites from \eqref{DEF_GH} and \eqref{THE_DECOUP_FOR_HOLDER_MOD_E}:
\begin{align*}
g_{h,\rho}(t)
\lesssim h^{\frac{\gamma-\varepsilon}\alpha}+\sup_{s\in (h,T]} g_{h,\rho}(s)t^{\frac{\gamma-\frac{\eps}{2}}{\alpha}}.
\end{align*}
Since the exponent of $t$ in the above was chosen to be non-negative, we can take the supremum in $t\in (h,T] $ in the above equation to derive, for a constant $C\geq 1$,
\begin{align*}
\sup_{t\in (h,T]}g_{h,\rho}(t)
\leq C\left( h^{\frac{\gamma-\varepsilon}\alpha} +\sup_{s\in (h,T]} g_{h,\rho}(s) T^{\frac{\gamma-\frac{\eps}{2}}{\alpha}}\right),
\end{align*}
which for $T$ small enough (i.e. s.t. $CT^{\frac{\gamma-\frac{\eps}{2}}{\alpha}}\leq 1/2$) eventually yields:
\begin{align*}
\sup_{t\in (h,T]}g_{h,\rho}(t)
\leq 2C h^{\frac{\gamma-\varepsilon}\alpha} \lesssim h^{\frac{\gamma-\eps}{\alpha}}.
\end{align*}
The theorem is proved.

	\section{Proof of technical Lemmas}\label{sec-proofs-lemmas}
	\subsection{Proof of Lemma \ref{lemma-besov-estimates-palpha}}\label{subsec-lemma-besov}
	 \subsubsection{Proof of \eqref{big-lemma-2}}\label{PROOF_BIG_LEMME_2}
%
 We provide here a full proof of the key relation \eqref{big-lemma-2} that we recall for the sake of completeness. Setting for this paragraph,
 \begin{align*}
\mathfrak{q}_{x,w,\textcolor{black}{y}}^{s,t}(\cdot):=\bar{p}_{\alpha} (s,x-\cdot)\left[\frac{\nabla p_\alpha (t-s,w-\cdot)}{\bar{p}_{\alpha} (t,w-x)} -\frac{\nabla p_\alpha (t-s,y-\cdot)}{\bar{p}_{\alpha} (t,y-x)}\right],
 \end{align*}
 it rewrites:
\begin{align*}
		&\left\Vert \bar{p}_{\alpha} (s,x-\cdot) \mathfrak{q}_{x,w,\textcolor{black}{y}}^{s,t}\right\Vert_{ \B_{p',q'}^{-\beta}} 
		\lesssim \frac{|w-y|^\zeta }{(t-s)^{\frac{\zeta+1}{\alpha}}} t^{\frac{\beta}{\alpha}} \left[\frac{1}{s^{\frac{d }{\alpha  p}}}+ \frac{1}{(t-s)^{\frac{d }{\alpha  p}}} \right] \left[1+  \frac{t^{\frac{\zeta}{\alpha}}}{s^{\frac{\zeta }{\alpha}}}+\frac{t^{\frac{\zeta}{\alpha}}}{(t-s)^{\frac{\zeta }{\alpha}}}  \right] .
\end{align*} 
It is plain from the $L^1-L^{p'}$ convolution inequality and the definition in Section \ref{subsec-besov} that
\begin{align}
\left\|\hat \phi * \mathfrak{q}_{x,w,\textcolor{black}{y}}^{s,t} \right\|_{L^{p'}}
\le 
\left\|\bar{p}_{\alpha} (s,x-\cdot)\mathfrak{q}_{x,w,\textcolor{black}{y}}^{s,t} \right\|_{L^{p'}} (\I_{|w-y|> (t-s)^{\frac 1\alpha}}+\I_{|w-y|\le (t-s)^{\frac 1\alpha}})
=: T_1+T_2.\label{NORM_NON_THERM}
\end{align}
For the \textit{off-diagonal} contribution we readily derive from \eqref{p-q-convo}:
\begin{align*}
T_1 \le & \Big(\frac{1}{(t-s)^{\frac 1\alpha}}\left[\frac{\|p_\alpha (t-s,w-\cdot)\bar{p}_{\alpha} (s,x-\cdot)\|_{L^{p'}}}{\bar p(t,w-x)}+ \frac{\|p_\alpha (t-s,y-\cdot)\bar{p}_{\alpha} (s,x-\cdot)\|_{L^{p'}}}{\bar p(t,y-x)}\right]\Big)\I_{|w-y|> (t-s)^{\frac 1\alpha}}\\
\lesssim &\frac{|w-y|^\zeta}{(t-s)^{\frac{1+\zeta}\alpha}} \Big(\frac{1}{s^{\frac{d}{\alpha p}}}+\frac{1}{(t-s)^{\frac d{\alpha p}}}\Big).
\end{align*}
In the \textit{diagonal regime} we write from \eqref{p-q-convo} and \eqref{holder-space-palpha}:
\begin{align*}
T_2\le &\Bigg(\left\| \bar{p}_{\alpha} (s,x-\cdot)\frac{\int_0^1 \d \lambda D^2 p_\alpha (t-s,y+\lambda(w-y)-\cdot)(w-y)}{\bar{p}_{\alpha} (t,w-x)}\right\|_{L^{p'}}\\
&+\left\|\bar{p}_{\alpha} (s,x-\cdot)\nabla p_\alpha (t-s,y-\cdot) \right\|_{L^{p'}}\left |\frac{1}{\bar p_\alpha(t,w-x)}-\frac{1}{\bar p_\alpha(t,y-x)}\right|\Bigg)\I_{|w-y|\le (t-s)^{\frac1\alpha}}\\
\lesssim & \Bigg(\frac{|w-y|}{(t-s)^{\frac 2\alpha} \bar p_\alpha(t,w-x)} \bar p_\alpha(t,y-x)\left(\frac{1}{s^{\frac{d}{\alpha p}}}+\frac{1}{(t-s)^{\frac d{\alpha p}}} \right)\\
&+\frac{1}{(t-s)^\frac1\alpha}\left(\frac{1}{s^{\frac{d}{\alpha p}}}+\frac{1}{(t-s)^{\frac d{\alpha p}}} \right)\bar p_\alpha(t-s,y-x)\frac{\int_0^1 \d\lambda \nabla \bar p_\alpha(t,y+\lambda (w-y)-x)\cdot(w-y) }{\bar p_\alpha(t,w-x)\bar p_\alpha(t,y-x)}\Bigg)\I_{|w-y|\le (t-s)^{\frac1\alpha}}\\
\lesssim &\frac{|w-y|^\zeta}{(t-s)^{\frac {1+\zeta}\alpha}}
\left(\frac{1}{s^{\frac{d}{\alpha p}}}+\frac{1}{(t-s)^{\frac d{\alpha p}}}\right),
\end{align*}
using that \eqref{holder-space-palpha} still holds for $\bar p_\alpha $ and recalling that $\bar p_\alpha(t-s,y-x)\asymp \bar p_\alpha(t-s,y-x) $ and $\bar p_\alpha(t,y-x)\asymp \bar p_\alpha(t,y-x) $ on the considered regime $|w-y|\le (t-s)^{\frac 1\alpha} $. Plugging the above bounds for $T_1$ and $T_2$ into \eqref{NORM_NON_THERM} gives that for the non-thermic part of the norm:
\begin{align}
\label{CTR_NT_HOLD}
\left\|\hat \phi *\mathfrak{q}_{x,w,\textcolor{black}{y}}^{s,t} \right\|_{L^{p'}}\lesssim\frac{|w-y|^\zeta}{(t-s)^{\frac {1+\zeta}\alpha}}
\left(\frac{1}{s^{\frac{d}{\alpha p}}}+\frac{1}{(t-s)^{\frac d{\alpha p}}}\right).
\end{align}
That is, the contribution is indeed compatible with the bound stated in \eqref{big-lemma-2}.

Let us now turn to the thermic part. 
Assuming w.l.o.g. that $q'<+\infty$,:
\begin{align}
\mathcal T_{p',q'}^{-\beta}\Bigg(\mathfrak{q}_{x,w,\textcolor{black}{y}}^{s,t}\Bigg)
\lesssim&\Big(\int_0^t \frac{dv}{v}v^{(1+\frac{\beta}{\alpha})q'}\|\partial_v p_\alpha(v,\cdot)*\mathfrak{q}_{x,w,\textcolor{black}{y}}^{s,t}\|_{L^{p'}}^{q'}\Big)^{\frac {1}{q'}}+\Big(\int_t^1 \frac{dv}{v}v^{(1+\frac{\beta}{\alpha})q'}\|\partial_v p_\alpha(v,\cdot)*\mathfrak{q}_{x,w,\textcolor{black}{y}}^{s,t}\|_{L^{p'}}^{q'}\Big)^{\frac 1{q'}}\notag\\
&=:\mathcal T_{p',q'}^{-\beta,(0,t)}(\mathfrak{q}_{x,w,\textcolor{black}{y}}^{s,t})+\mathcal T_{p',q'}^{-\beta,(1,t)}(\mathfrak{q}_{x,w,\textcolor{black}{y}}^{s,t}),\label{NORM_THERM_SPLIT}
\end{align}
where we have split according to the cutting level $t\le 1$. From the previous bounds on $T_1,T_2$ and \eqref{NORM_NON_THERM}, we readily get:
\begin{align}
\mathcal T_{p',q'}^{-\beta,(1,t)}(\mathfrak{q}_{x,w,\textcolor{black}{y}}^{s,t})\lesssim& \Big(\int_t^1 \frac{dv}{v}v^{(1+\frac{\beta}{\alpha})q'}\|\partial_v p_\alpha(v,\cdot)*\mathfrak{q}_{x,w,\textcolor{black}{y}}^{s,t}\|_{L^{p'}}^{q'}\Big)^{\frac 1{q'}}\notag\\
\lesssim&\Big(\int_t^1 \frac{dv}{v}v^{(1+\frac{\beta}{\alpha})q'}\|\partial_v p_\alpha(v,\cdot)\|_{L^1}^{q'}\|\mathfrak{q}_{x,w,\textcolor{black}{y}}^{s,t}\|_{L^{p'}}^{q'}\Big)^{\frac 1{q'}}\notag\\
\lesssim& \frac{|w-y|^\zeta}{(t-s)^{\frac {1+\zeta}\alpha}}
\left(\frac{1}{s^{\frac{d}{\alpha p}}}+\frac{1}{(t-s)^{\frac d{\alpha p}}}\right)\Big(\int_t^1 \frac{dv}{v}v^{\frac{\beta}{\alpha}q'}\Big)^{\frac 1{q'}}\lesssim \frac{|w-y|^\zeta}{(t-s)^{\frac {1+\zeta}\alpha}}
\left(\frac{1}{s^{\frac{d}{\alpha p}}}+\frac{1}{(t-s)^{\frac d{\alpha p}}}\right)t^{\frac \beta\alpha}.\label{CTR_UPPER_CUT_HOLD}
\end{align}
It remains to handle $ \mathcal T_{p',q'}^{-\beta,(0,t)}(\mathfrak{q}_{x,w,\textcolor{black}{y}}^{s,t})$ for which a cancellation argument is needed to absorb the time singularity in the thermic variable $v$. Write first:
\begin{align*}
\|\partial_v p_\alpha(v,\cdot)*\mathfrak{q}_{x,w,\textcolor{black}{y}}^{s,t}\|_{L^{p'}}^{p'}=\int_{\R^d} \left|\int_{\R^d}\partial_v p_\alpha(v,z-\bar z)\left[\mathfrak{q}_{x,w,\textcolor{black}{y}}^{s,t}(\bar z)-\mathfrak{q}_{x,w,\textcolor{black}{y}}^{s,t}(z)\right] \d \bar z\right|^{p'} \d z
\end{align*}
It is here needed ton control the $\zeta>-\beta $ modulus of continuity of $z\mapsto \mathfrak{q}_{x,w,\textcolor{black}{y}}^{s,t}(z) $ which must also make appear the a prefactor in $|y-w|^\zeta $ as far as the spatial variables $y,w$ are concerned. As above we discuss according to the regime of $|y-w| $ w.r.t. $(t-s)^{\frac 1\alpha} $. In the off-diagonal regime $|y-w|>(t-s)^{\frac 1\alpha} $, we write:
 \begin{align}
&| \mathfrak{q}_{x,w,\textcolor{black}{y}}^{s,t}(z)-\mathfrak{q}_{x,w,\textcolor{black}{y}}^{s,t}(\bar z)|\notag\\
\le& \Bigg| \bar{p}_{\alpha} (s,x-z)\left[\frac{\nabla p_\alpha (t-s,w-z)}{\bar{p}_{\alpha} (t,w-x)} -\frac{\nabla p_\alpha (t-s,y-z)}{\bar{p}_{\alpha} (t,y-x)}\right]\notag\\
&-\bar{p}_{\alpha} (s,x-\bar z)\left[\frac{\nabla p_\alpha (t-s,w-\bar z)}{\bar{p}_{\alpha} (t,w-x)} -\frac{\nabla p_\alpha (t-s,y-\bar z)}{\bar{p}_{\alpha} (t,y-x)}\right]\Bigg|\I_{|w-y|>(t-s)^{\frac 1\alpha}}\notag\\
\le &\Bigg(\frac{1}{\bar{p}_{\alpha} (t,w-x)}\Bigg| \bar{p}_{\alpha} (s,x-z)\nabla p_\alpha (t-s,w-z)-\bar{p}_{\alpha} (s,x-\bar z)\nabla p_\alpha (t-s,w-\bar z)\Bigg|\notag\\
&+\frac{1}{\bar{p}_{\alpha} (t,y-x)}\Bigg| \bar{p}_{\alpha} (s,x-z)\nabla p_\alpha (t-s,y-z)-\bar{p}_{\alpha} (s,x-\bar z)\nabla p_\alpha (t-s,y-\bar z)\Bigg|\Bigg)\I_{|w-y|>(t-s)^{\frac 1\alpha}}\notag\\
\le & \Bigg(\sum_{\fy\in\{w,y\}}
\frac{1}{\bar{p}_{\alpha} (t,\fy-x)}\Bigg| \bar{p}_{\alpha} (s,x-z)\nabla p_\alpha (t-s,\fy-z)-\bar{p}_{\alpha} (s,x-\bar z)\nabla p_\alpha (t-s,\fy-\bar z)\Bigg|\Bigg)\I_{|w-y|>(t-s)^{\frac 1\alpha}}\notag\\
=:&\Big(\sum_{\fy\in\{w,y\}} D_{s,t,x,\fy}(z,\bar z)\Big)\I_{|w-y|>(t-s)^{\frac 1\alpha}}.
\label{PREAL_TAGLIO_DOPPIA_DISC_FD}
\end{align}
At this point an additional discussion must be performed concerning the diagonal or off-diagonal regime of the thermic variables. When $|z-\bar z|> s^{\frac 1\alpha} $, we write:
\begin{align*}
|D_{s,t,x,\fy}(z,\bar z)|\lesssim \frac{|z-\bar z|^\zeta}{s^{\frac \zeta\alpha}(t-s)^{\frac1\alpha}\bar{p}_{\alpha} (t,\fy-x)}\Big(\bar{p}_{\alpha} (s,x-z)\bar p_\alpha (t-s,\fy-z)+\bar{p}_{\alpha} (s,x-\bar z)\bar p_\alpha (t-s,\fy-\bar z)\Big).
\end{align*}
If now $|z-\bar z|\le s^{\frac 1\alpha} $, we get from \eqref{holder-space-palpha},
\begin{align*}
|D_{s,t,x,\fy}(z,\bar z)|\lesssim& \frac{|z-\bar z|^{\zeta}}{(t-s)^{\frac1\alpha}}\frac{1}{\bar p_\alpha(t,\fy-x)}\Bigg[\frac{1}{s^{\frac \zeta\alpha}}(\bar p_\alpha(s,x-z)+\bar p_\alpha(s,x-\bar z)) \bar p_\alpha(t-s,\fy-z))\notag\\
&+\frac{1}{ (t-s)^{\frac {\zeta}\alpha}}\bar p_\alpha(s,x-\bar z)(\bar p_\alpha(t-s,\fy-z)+\bar p_\alpha(t-s,\fy-\bar z))\Bigg]\\
\lesssim&\frac{|z-\bar z|^{\zeta}}{(t-s)^{\frac1\alpha}}\Big(\frac{1}{s^{\frac \zeta\alpha}}+\frac{1}{(t-s)^{\frac \zeta\alpha}}\Big)\frac{1}{\bar p_\alpha(t,\fy-x)}\Bigg[\bar p_\alpha(s,x-z)\bar p_\alpha(t-s,\fy-z)\notag\\
&+\bar p_\alpha(s,x-\bar z)\bar p_\alpha(t-s,\fy-\bar z)\Bigg],
\end{align*}
where we precisely exploited that $\bar p_\alpha(s,x-\bar z)\lesssim \bar p_\alpha(s,x-z) $ in the considered regime.

This then gives in \eqref{PREAL_TAGLIO_DOPPIA_DISC_FD}
\begin{align}
&| \mathfrak{q}_{x,w,\textcolor{black}{y}}^{s,t}(z)-\mathfrak{q}_{x,w,\textcolor{black}{y}}^{s,t}(\bar z)|\notag\\
\le & |z-\bar z|^{\zeta}\Bigg(\frac{1}{ s^{\frac {\zeta}\alpha}}+\frac{1}{ (t-s)^{\frac {\zeta}\alpha}}\Bigg)\notag\\
&\Bigg[\sum_{\fy \in \{ w,y\}}\frac{ (\bar p_\alpha(s,z-x)\bar p_\alpha(t-s,\fy-z)+\bar p_\alpha(s,\bar z-x)\bar p_\alpha(t-s,\fy-\bar z))}{\bar p_\alpha(t,\fy-x)}\Bigg]\frac{|w-y|^\zeta}{(t-s)^{\frac{1+\zeta}\alpha}}.\label{DIFF_IN_PSI}
 \end{align}
Hence, 
\begin{align}
&\int_{\R^d} \left|\int_{\R^d}\partial_v p_\alpha(v,z-\bar z)\left[\mathfrak{q}_{x,w,\textcolor{black}{y}}^{s,t}(\bar z)-\mathfrak{q}_{x,w,\textcolor{black}{y}}^{s,t}(z)\right| \d \bar z\right|^{p'} \d z\I_{|w-y|>(t-s)^{\frac 1\alpha}}\notag\\
\lesssim&  \Bigg(v^{-1}\Bigg(\frac{1}{ s^{\frac {\zeta}\alpha}}+\frac{1}{ (t-s)^{\frac {\zeta}\alpha}}\Bigg)\frac{|w-y|^\zeta}{(t-s)^{\frac{1+\zeta}\alpha}}\Bigg)^{p'}\notag\\
&\times \Bigg(\sum_{\fy\in \{w,y\}}\frac{1}{\bar p_\alpha(t,\fy-x)^{p'}}\int_{\R^d}\Big(\int_{\R^d}\bar p_\alpha(v,z-\bar z)|z-\bar z|^\zeta(\bar p_\alpha(s,z-x)\bar p_\alpha(t-s,\fy-z)\notag\\
&+\bar p_\alpha(s,\bar z-x)\bar p_\alpha(t-s,\fy-\bar z)) \d \bar z\Big)^{p'}  \d z\Bigg)\notag\\
\lesssim&  \Bigg(v^{-1}\Bigg(\frac{1}{ s^{\frac {\zeta}\alpha}}+\frac{1}{ (t-s)^{\frac {\zeta}\alpha}}\Bigg)\frac{|w-y|^\zeta}{(t-s)^{\frac{1+\zeta}\alpha}}\Bigg)^{p'}\notag\\
&\times \Bigg(\sum_{\fy\in \{w,y\}}\frac{1}{\bar p_\alpha(t,\fy-x)^{p'}}\Big(\int_{\R^d} (\bar p_\alpha(s,z-x)\bar p_\alpha(t-s,\fy-z))^{p'}\d z v^{\frac{\zeta}{\alpha}p'} \notag\\
&+ \int_{\R^d}\int_{\R^d}\bar p(v,z-\bar z)|z-\bar z|^\zeta(\bar p_\alpha(s,\bar z-x)\bar p_\alpha(t-s,\fy-\bar z))^{p'} \d \bar z  \d z\Big)\Bigg)\notag\\
\lesssim&\Bigg(v^{-1+\frac \zeta\alpha}\Bigg(\frac{1}{ s^{\frac {\zeta}\alpha}}+\frac{1}{ (t-s)^{\frac {\zeta}\alpha}}\Bigg)\frac{|w-y|^\zeta}{(t-s)^{\frac{1+\zeta}\alpha}}\Bigg[\frac{1}{s^{\frac d{p\alpha}}}+\frac{1}{(t-s)^{\frac d{p\alpha}}} \Bigg]\Bigg)^{p'},\label{THE_CTR_LPP_DUALITY_DEL}
\end{align}
 using the $L^1-L^{p'} $ Young inequality, the Fubini theorem and \eqref{p-q-convo} for the last inequality.
 It remains to establish the same control remains valid when $|w-y|\le (t-s)^{\frac 1\alpha} $ (diagonal regime). Write:
  \begin{align}
&|\mathfrak{q}_{x,w,\textcolor{black}{y}}^{s,t}(z)-\mathfrak{q}_{x,w,\textcolor{black}{y}}^{s,t}(\bar z)|\notag\\
\le& \Bigg| \bar{p}_{\alpha} (s,x-z)\left[\frac{\nabla p_\alpha (t-s,w-z)}{\bar{p}_{\alpha} (t,w-x)} -\frac{\nabla p_\alpha (t-s,y-z)}{\bar{p}_{\alpha} (t,y-x)}\right]\notag\\
&-\bar{p}_{\alpha} (s,x-\bar z)\left[\frac{\nabla p_\alpha (t-s,w-\bar z)}{\bar{p}_{\alpha} (t,w-x)} -\frac{\nabla p_\alpha (t-s,y-\bar z)}{\bar{p}_{\alpha} (t,y-x)}\right]\Bigg|\I_{|w-y|\le (t-s)^{\frac 1\alpha}}\notag\\
\le& \Bigg| \bar{p}_{\alpha} (s,x-z)\left[\frac{\nabla p_\alpha (t-s,w-z)-\nabla p_\alpha (t-s,y-z)}{\bar{p}_{\alpha} (t,w-x)} +\nabla p_\alpha (t-s,y-z)\left[\frac{1}{\bar{p}_{\alpha} (t,w-x)}-\frac{1}{\bar{p}_{\alpha} (t,y-x)}\right]\right]\notag\\
-&\left(\bar{p}_{\alpha} (s,x-\bar z)\left[\frac{\nabla p_\alpha (t-s,w-\bar z)-\nabla p_\alpha (t-s,y-\bar z)}{\bar{p}_{\alpha} (t,w-x)} +\nabla p_\alpha (t-s,y-\bar z)\left[\frac{1}{\bar{p}_{\alpha} (t,w-x)}-\frac{1}{\bar{p}_{\alpha} (t,y-x)}\right]\right]\right)\Bigg|\notag\\
&\times \I_{|w-y|\le (t-s)^{\frac 1\alpha}}\notag\\
\lesssim &\frac{1}{\bar p_\alpha(t,w-x)} \Bigg[\Bigg|\bar{p}_{\alpha} (s,x-z) \int_{0}^1 \d \lambda D^2p_\alpha(t-s,y+\lambda(w-y)-z)(w-y)\notag\\
&-  \bar{p}_{\alpha} (s,x-\bar z) \int_{0}^1 \d \lambda D^2p_\alpha(t-s,y+\lambda(w-y)-\bar z)(w-y)\Bigg|\notag\\
&+  |\bar{p}_{\alpha} (s,x-z)\nabla p_\alpha (t-s,y-z)-\bar{p}_{\alpha} (s,x-\bar z)\nabla p_\alpha (t-s,y-\bar z)|\left[\frac{1}{\bar{p}_{\alpha} (t,w-x)}-\frac{1}{\bar{p}_{\alpha} (t,y-x)}\right]\Bigg]\notag\\
&\times\I_{|w-y|\le (t-s)^{\frac 1\alpha}}.\label{LOW_CUT_OFF_DIAG_REGIME_PROV}
\end{align}
It now remains to discuss as above and use a splitting according to the position of $|z-\bar z| $ with respect to $s^{\frac 1\alpha} $. Starting with the associated off-diagonal regime $|z-\bar z|>s^{\frac 1\alpha} $, we derive from \eqref{LOW_CUT_OFF_DIAG_REGIME_PROV} that in that case:
\begin{align}
&| \mathfrak{q}_{x,w,\textcolor{black}{y}}^{s,t}(z)-\mathfrak{q}_{x,w,\textcolor{black}{y}}^{s,t}(\bar z)|\notag\\
\le& \frac{|z-\bar z|^\zeta}{s^{\frac \zeta \alpha}}\Bigg[ \frac{1}{\bar p_\alpha(t,y-x)}\Big(\sum_{\fz\in \{z,\bar z\}}\bar p_\alpha(s,x-\fz)\bar p_\alpha(t-s,y-\fz)\Big)\frac{|w-y|^\zeta}{(t-s)^{\frac{1+\zeta}\alpha}}\Bigg]\I_{|w-y|\le (t-s)^{\frac 1\alpha}},\notag
\end{align}
using again thoroughly the controls of Lemma \ref{lemma-stable-sensitivities} (which also extend to $\bar p_\alpha $). A very similar control is then obtained  in the diagonal regime $|z-\bar z|\le s^{\frac 1\alpha} $, up to additional inequalities making the $\zeta $-H\"older modulus appear in the $z,\bar z $ variables and exploiting that then $\bar p_\alpha(s,x-z)\asymp \bar p_\alpha(s,x-\bar z) $. Precisely,  one derives:
\begin{align}
&| \mathfrak{q}_{x,w,\textcolor{black}{y}}^{s,t}(z)-\mathfrak{q}_{x,w,\textcolor{black}{y}}^{s,t}(\bar z)|\notag\\
\le& |z-\bar z|^\zeta \Big( \frac{1}{s^{\frac \zeta \alpha}}+\frac{1}{(t-s)^\frac{\alpha}\alpha}\Big)\Bigg[ \frac{1}{\bar p_\alpha(t,y-x)}\Big(\sum_{\fz\in \{z,\bar z\}}\bar p_\alpha(s,x-\fz)\bar p_\alpha(t-s,y-\fz)\Big)\frac{|w-y|^\zeta}{(t-s)^{\frac{1+\zeta}\alpha}}\Bigg]\I_{|w-y|\le (t-s)^{\frac 1\alpha}},\notag
\end{align}
whose r.h.s. is eventually bounded by the one of \eqref{DIFF_IN_PSI}. Thus, we get that in any case:
\begin{align*}
\|\partial_v p_\alpha(v,\cdot)*\mathfrak{q}_{x,w,\textcolor{black}{y}}^{s,t}\|_{L^{p'}}\lesssim&v^{-1+\frac \zeta\alpha}\Bigg(\frac{1}{ s^{\frac {\zeta}\alpha}}+\frac{1}{ (t-s)^{\frac {\zeta}\alpha}}\Bigg)\frac{|w-y|^\zeta}{(t-s)^{\frac{1+\zeta}\alpha}}\Bigg[\frac{1}{s^{\frac d{p\alpha}}}+\frac{1}{(t-s)^{\frac d{p\alpha}}} \Bigg],
\end{align*}
which from the definition of $\mathcal T_{p',q'}^{-\beta,(0,t)}(\mathfrak{q}_{x,w,\textcolor{black}{y}}^{s,t}) $ in \eqref{NORM_THERM_SPLIT} eventually gives for $\zeta\in (-\beta,1] $:
\begin{align*}
\mathcal T_{p',q'}^{-\beta,(0,t)}(\mathfrak{q}_{x,w,\textcolor{black}{y}}^{s,t})&\lesssim \Big(\int_0^t v^{-1}v^{(1+\frac \beta\alpha)q'}v^{(-1+\frac \zeta \alpha)q'} \d v\Big)^{\frac 1{q'}}\Bigg(\frac{1}{ s^{\frac {\zeta}\alpha}}+\frac{1}{ (t-s)^{\frac {\zeta}\alpha}}\Bigg)\frac{|w-y|^\zeta}{(t-s)^{\frac{1+\zeta}\alpha}}\Bigg[\frac{1}{s^{\frac d{p\alpha}}}+\frac{1}{(t-s)^{\frac d{p\alpha}}} \Bigg]\notag\\
&\lesssim t^{\frac \beta\alpha+\frac \zeta\alpha}\Bigg(\frac{1}{ s^{\frac {\zeta}\alpha}}+\frac{1}{ (t-s)^{\frac {\zeta}\alpha}}\Bigg)\frac{|w-y|^\zeta}{(t-s)^{\frac{1+\zeta}\alpha}}\Bigg[\frac{1}{s^{\frac d{p\alpha}}}+\frac{1}{(t-s)^{\frac d{p\alpha}}} \Bigg],
\end{align*}
which together with \eqref{CTR_UPPER_CUT_HOLD}, \eqref{NORM_THERM_SPLIT} and \eqref{CTR_NT_HOLD} gives the statement.
	\subsubsection{Proofs of \eqref{besov-estimate-gamma-and-stable_NO_NORM_Holder}, \eqref{besov-estimate-gammah-stable-sensi-holder-time}, \eqref{besov-estimate-stable-derivees_temps_esp_sensi_holder_esp} and \eqref{besov-estimate-gammah-stable-derivees_temps_esp_sensi_holder_esp}}\label{subsubsec-lemma-besov-1}
		Let us first prove \eqref{besov-estimate-gamma-and-stable_NO_NORM_Holder}. Denote for this section $\mathfrak{q}_{x,y,\textcolor{black}{y'}}^{s,t}(\cdot):=\Gamma(0,x,s,\cdot) \left(\nabla_y p_\alpha (t-s,y-\cdot)-\nabla_{y'} p_\alpha (t-s,y'-\cdot)\right)$, of which we will control the $\B_{p',q'}^{-\beta}$ norm using the \textcolor{black}{thermic characterization introduced in \eqref{HEAT_CAR}}:
		
		$$\Vert \mathfrak{q}_{x,y,\textcolor{black}{y'}}^{s,t}\Vert_{\B_{p',q'}^{-\beta}} = \Vert \textcolor{black}{\mathcal F^{-1}(\phi  \mathcal F \mathfrak{q}_{x,y,\textcolor{black}{y'}}^{s,t})}\Vert_{L^{p'}} + \mathcal{T}_{p',q'}^{-\beta} [\mathfrak{q}_{x,y,\textcolor{black}{y'}}^{s,t}] .$$
		
		\begin{paragraph}{Thermic part \\[0.5cm]}
			
			\textcolor{black}{We assume w.l.o.g. that $t<1$ and $q<+\infty$.} Let us recall the definition of the thermic part and split it in two parts:
			\begin{align}
				\mathcal{T}_{p',q'}^{-\beta} [\mathfrak{q}_{x,y,\textcolor{black}{y'}}^{s,t}] ^{q'} &= \int_0^t \frac{\d v}{v} v^{\left(1+\frac{\beta}{\alpha}\right)q'} \Vert \partial_v p_\alpha (v,\cdot) \star \mathfrak{q}_{x,y,\textcolor{black}{y'}}^{s,t} (\cdot)\Vert_{L^{p'}}^{q'} +\int_t^1 \frac{\d v}{v} v^{\left(1+\frac{\beta}{\alpha}\right)q'} \Vert \partial_v p_\alpha (v,\cdot) \star \mathfrak{q}_{x,y,\textcolor{black}{y'}}^{s,t} (\cdot)\Vert_{L^{p'}}^{q'}\notag\\
				&=:\mathcal{T}_{p',q'}^{-\beta,(0,t)} [\mathfrak{q}_{x,y,\textcolor{black}{y'}}^{s,t}] ^{q'} + \mathcal{T}_{p',q'}^{-\beta,(t,1)} [\mathfrak{q}_{x,y,\textcolor{black}{y'}}^{s,t}] ^{q'}.\label{DEF_CUT_1}
			\end{align}
			For the upper part on $(t,1)$, using a $L^1-L^{p'}$ convolution inequality, we get
			\begin{align*}
				\mathcal{T}_{p',q'}^{-\beta,(t,1)} [\mathfrak{q}_{x,y,\textcolor{black}{y'}}^{s,t}] ^{q'} & \lesssim \int_t^1 \frac{\d v}{v} v^{\left(1+\frac{\beta}{\alpha}\right)q'} \Vert \partial_v p_\alpha (v,\cdot)\Vert_{L^{1}}^{q'} \Vert \Gamma(0,x,s,\cdot) \left(\nabla_y p_\alpha (t-s,y-\cdot)-\nabla_{y'} p_\alpha (t-s,y'-\cdot)\right)\Vert_{L^{p'}}^{q'}
			\end{align*}
			Using the pointwise estimate \eqref{holder-space-palpha} and \eqref{aronson-gamma}, then \eqref{p-q-convo}, we have, for any $\rho \in (-\beta,1]$,
			\begin{align*}
				\Vert \Gamma(0,x,s,\cdot) \left(\nabla_y p_\alpha (t-s,y-\cdot)-\nabla_{y'} p_\alpha (t-s,y'-\cdot)\right)\Vert_{L^{p'}} \lesssim \bar{p}_{\alpha}(t,y-x) \frac{|y-y'|^\rho}{(t-s)^{\frac{\rho+1}{\alpha}}} \left[\frac{1}{s^{\frac{d}{\alpha p}}}+\frac{1}{(t-s)^{\frac{d}{\alpha p}}}\right],
			\end{align*}
			\textcolor{black}{recalling as well for the last inequality that we assumed $|y-y'|\leq t^{\frac{1}{\alpha}}$}.
			
			This yields
			\begin{align}
				\mathcal{T}_{p',q'}^{-\beta(t,1)} [\mathfrak{q}_{x,y,\textcolor{black}{y'}}^{s,t}]  \nonumber& \lesssim \bar{p}_{\alpha}(t,y-x) \frac{|y-y'|^{\rho }}{(t-s)^{\frac{\rho+1}{\alpha}}} \left[\frac{1}{s^{\frac{d }{\alpha p}}}+\frac{1}{(t-s)^{\frac{d}{\alpha p}}}\right]\left(\int_t^1  v^{\frac{\beta q'}{\alpha}-1}  \d v\right)^{\frac{1}{q'}}\\
				&\lesssim \bar{p}_{\alpha}(t,y-x)\frac{|y-y'|^{\rho }}{(t-s)^{\frac{\rho+1}{\alpha}}} \left[\frac{1}{s^{\frac{d }{\alpha p}}}+\frac{1}{(t-s)^{\frac{d}{\alpha p}}}\right]\textcolor{black}{t^{\frac{\beta}{\alpha}}}.\label{UPPER_CUT_FIRST}
			\end{align}
			For the lower part, let us write
			\begin{align}
				\Vert \partial_v p_\alpha (v,\cdot)\star \mathfrak{q}_{x,y,\textcolor{black}{y'}}^{s,t}\Vert_{L^{p'}}^{p'}&\nonumber =\int \left|\int \partial_v p_\alpha (v,z-w)\mathfrak{q}_{x,y,\textcolor{black}{y'}}^{s,t}(w)\d w\right|^{p'} \d z\\
				&  =\int \bigg|\int \partial_v p_\alpha (v,z-w)\big[\mathfrak{q}_{x,y,\textcolor{black}{y'}}^{s,t}(w)-\mathfrak{q}_{x,y,\textcolor{black}{y'}}^{s,t}(z)\big]\d w\bigg|^{p'} \d z,\label{lower-part-cancel}
			\end{align}
			using a cancellation argument for the last equality. Next, let us distinguish whether this difference is in diagonal or off-diagonal regime \textcolor{black}{w.r.t. the intermediate time $s$}.
			\begin{trivlist}
				\item[$\bullet $] \textbf{Diagonal case: $|z-w|\leq s^\frac{1}{\alpha}$}. Let us write
				\begin{align*}
					&\mathfrak{q}_{x,y,\textcolor{black}{y'}}^{s,t}(w)-\mathfrak{q}_{x,y,\textcolor{black}{y'}}^{s,t}(z)\\
					&= \Gamma(0,x,s,w) \big[\left(\nabla_y p_\alpha (t-s,y-w)-\nabla_{y'} p_\alpha (t-s,y'-w)\right)-\left(\nabla_y p_\alpha (t-s,y-z)-\nabla_{y'} p_\alpha (t-s,y'-z)\right)\\
					&\quad - \big[\Gamma(0,x,s,z)-\Gamma(0,x,s,w)\big] \left(\nabla_y p_\alpha (t-s,y-z)-\nabla_{y'} p_\alpha (t-s,y'-z)\right)\\
					&=\Gamma(0,x,s,w) \Big(\int_0^1 \left[\nabla_y^2 p_\alpha (t-s,y-w-\lambda (z-w))-\nabla_{y'}^2 p_\alpha (t-s,y'-w-\lambda(z-w))\right](w-z)\d \lambda \big]\textcolor{black}{{\mathbb I}_{|z-w|\le (t-s)^{\frac 1\alpha}}}\\
					&\textcolor{black}{+\big[\left(\nabla_y p_\alpha (t-s,y-w)-\nabla_{y'} p_\alpha (t-s,y'-w)\right)-\left(\nabla_y p_\alpha (t-s,y-z)-\nabla_{y'} p_\alpha (t-s,y'-z)\right)\big]{\mathbb I}_{|z-w|>(t-s)^{\frac 1\alpha}}\Big)}\\
					&\quad - \big[\Gamma(0,x,s,z)-\Gamma(0,x,s,w)\big] \left(\nabla_y p_\alpha (t-s,y-z)-\nabla_{y'} p_\alpha (t-s,y'-z)\right).
				\end{align*}
				Using the regularity of $p_\alpha$, \eqref{holder-space-palpha} and the forward regularity of $\Gamma$, \eqref{holder-forward-gamma}, we get, for any $\zeta\in (0,\textcolor{black}{-\beta+\gamma})$\footnote{\textcolor{black}{Pay attention that here again the choice for the exponent $\zeta $ is limited by the forward regularity of the density.}},
				\begin{align}
					\nonumber&|\mathfrak{q}_{x,y,\textcolor{black}{y'}}^{s,t}(w)-\mathfrak{q}_{x,y,\textcolor{black}{y'}}^{s,t}(z)|\\
					\nonumber&\lesssim\bar{p}_{\alpha} (s,w-x) \textcolor{black}{\frac{|y-y'|^\rho }{(t-s)^{\frac{\rho+1}{\alpha}}}\frac{|z-w|^\zeta}{(t-s)^{\frac{\zeta}{\alpha}}}\Big(\int_0^1\left[\bar{p}_{\alpha} (t-s,y-w-\lambda (z-w))+ \bar{p}_{\alpha} (t-s,y'-w-\lambda(z-w))\right]\d \lambda {\mathbb I}_{|z-w|\le (t-s)^{\frac 1\alpha}}}\\
					\nonumber & \textcolor{black}{+(\bar p_\alpha(t-s,y-w)+\bar p_\alpha(t-s,y'-w)+\bar p_\alpha(t-s,y-z)+\bar p_\alpha(t-s,y'-z)){\mathbb I}_{|z-w|> (t-s)^{\frac 1\alpha}}\Big)}\\
					\nonumber&\quad +\frac{|y-y'|^\rho }{(t-s)^{\frac{\rho+1}{\alpha}}}\frac{|z-w|^\zeta}{s^{\frac{\zeta}{\alpha}}}\left(\bar{p}_{\alpha} (s,z-x)+\bar{p}_{\alpha} (s,w-x)\right) \left( \bar{p}_{\alpha} (t-s,y-z)+ \bar{p}_{\alpha} (t-s,y'-z)\right)\\
					\nonumber&\lesssim\bar{p}_{\alpha} (s,w-x)\left(\bar{p}_{\alpha} (t-s,y-w)+\bar{p}_{\alpha} (t-s,y'-w)\right) \frac{|y-y'|^\rho }{(t-s)^{\frac{\rho+1}{\alpha}}}\frac{|z-w|^\zeta}{(t-s)^{\frac{\zeta}{\alpha}}}\\
					&\quad +\bar{p}_{\alpha} (s,z-x) \left( p_\alpha (t-s,y-z)+ p_\alpha (t-s,y'-z)\right)\frac{|y-y'|^\rho }{(t-s)^{\frac{\rho+1}{\alpha}}}\frac{|z-w|^\zeta}{s^{\frac{\zeta}{\alpha}}},\label{on-diag-holder-q}
				\end{align}
				using the current diagonal regime to write that $\bar{p}_{\alpha} (t-s,y-w-\lambda (z-w))\lesssim \bar{p}_{\alpha} (t-s,y-w)$ (and the same estimate for $y'$) and $\bar{p}_{\alpha} (s,w-x)\lesssim \bar{p}_{\alpha} (s,z-x)$.
				\item[$\bullet $] \textbf{Off-diagonal case: $|z-w|\geq s^\frac{1}{\alpha}$}. Using a triangular inequality, \eqref{holder-space-palpha}, \eqref{aronson-gamma} and the fact that $\frac{|z-w|^\zeta}{s^{\frac{\zeta}{\alpha}}}\geq 1$, we trivially have the following:
				\begin{align}
					|\mathfrak{q}_{x,y,\textcolor{black}{y'}}^{s,t}(w)-\mathfrak{q}_{x,y,\textcolor{black}{y'}}^{s,t}(z)|&\nonumber\lesssim\bar{p}_{\alpha} (s,w-x)\left(\bar{p}_{\alpha} (t-s,y-w)+\bar{p}_{\alpha} (t-s,y'-w)\right) \frac{|y-y'|^\rho }{(t-s)^{\frac{\rho+1}{\alpha}}}\frac{|z-w|^\zeta}{s^{\frac{\zeta}{\alpha}}}\\
					&\qquad +\bar{p}_{\alpha} (s,z-x) \left( \bar{p}_{\alpha} (t-s,y-z)+ \bar{p}_{\alpha} (t-s,y'-z)\right)\frac{|y-y'|^\rho }{(t-s)^{\frac{\rho+1}{\alpha}}}\frac{|z-w|^\zeta}{s^{\frac{\zeta}{\alpha}}}.\label{off-diag-holder-q}
				\end{align}
			\end{trivlist}
			Gathering \eqref{on-diag-holder-q} and \eqref{off-diag-holder-q}, we have
			\begin{align}
				|\mathfrak{q}_{x,y,\textcolor{black}{y'}}^{s,t}(w)-\mathfrak{q}_{x,y,\textcolor{black}{y'}}^{s,t}(z)|&\nonumber\lesssim\bar{p}_{\alpha} (s,w-x)\left(\bar{p}_{\alpha} (t-s,y-w)+\bar{p}_{\alpha} (t-s,y'-w)\right) \frac{|y-y'|^\rho }{(t-s)^{\frac{\rho+1}{\alpha}}}\left[\frac{|z-w|^\zeta}{s^{\frac{\zeta}{\alpha}}}+\frac{|z-w|^\zeta}{(t-s)^{\frac{\zeta}{\alpha}}}\right]\\
				&\qquad +\bar{p}_{\alpha} (s,z-x) \left( \bar{p}_{\alpha} (t-s,y-z)+ \bar{p}_{\alpha} (t-s,y'-z)\right)\frac{|y-y'|^\rho }{(t-s)^{\frac{\rho+1}{\alpha}}}\frac{|z-w|^\zeta}{s^{\frac{\zeta}{\alpha}}}.
			\end{align}
			Plugging this in \eqref{lower-part-cancel}, we get
			\begin{align}
				&\Vert \partial_v p_\alpha (v,\cdot)\star \mathfrak{q}_{x,y,\textcolor{black}{y'}}^{s,t}\Vert_{L^{p'}}^{p'}\nonumber =\int \left|\int \partial_v p_\alpha (v,z-w)\mathfrak{q}_{x,y}^{s,t}(w)\d w\right|^{p'} \d z\\
				&  \lesssim\int \bigg(\int v^{-1}\bar p_\alpha (v,z-w)\bar{p}_{\alpha} (s,w-x)\left(\bar{p}_{\alpha} (t-s,y-w)+\bar{p}_{\alpha} (t-s,y'-w)\right)\nonumber\\ & \qquad\qquad\qquad\qquad\qquad \qquad\qquad\qquad\qquad\qquad \qquad\times \frac{|y-y'|^\rho }{(t-s)^{\frac{\rho+1}{\alpha}}}\left[\frac{|z-w|^\zeta}{s^{\frac{\zeta}{\alpha}}}+\frac{|z-w|^\zeta}{(t-s)^{\frac{\zeta}{\alpha}}}\right]\d w\bigg)^{p'} \d z\nonumber\\
				&\qquad +\int \bigg(\int v^{-1}\bar p_\alpha (v,z-w)\bar{p}_{\alpha} (s,z-x) \left( \bar{p}_{\alpha} (t-s,y-z)+ \bar{p}_{\alpha} (t-s,y'-z)\right)\nonumber\\ & \qquad\qquad\qquad\qquad\qquad \qquad\qquad\qquad\qquad\qquad \qquad\times\frac{|y-y'|^\rho }{(t-s)^{\frac{\rho+1}{\alpha}}}\frac{|z-w|^\zeta}{s^{\frac{\zeta}{\alpha}}}\d w\bigg)^{p'} \d z.
			\end{align}
			From this point, we derive a smoothing effect in $v$ by using the moments estimate \eqref{spatial-moments}. It is immediate for the second term, whereas for the first one, due to the order of integration, we need to use an $L^1-L^{p'}$ convolution inequality. This yields 
			\begin{align}
				&\Vert \partial_v p_\alpha (v,\cdot)\star \mathfrak{q}_{x,y,\textcolor{black}{y'}}^{s,t}\Vert_{L^{p'}}^{p'}\nonumber \\
				&  \lesssim\left(\frac{v^{-1+\frac{\zeta}{\alpha}}|y-y'|^\rho }{(t-s)^{\frac{\rho+1}{\alpha}}}\left[s^{-\frac{\zeta}{\alpha}}+(t-s)^{-\frac{\zeta}{\alpha}}\right]\right)^{p'}\int \left( \bar{p}_{\alpha} (s,w-x) \left( \bar{p}_{\alpha} (t-s,y-w)+ \bar{p}_{\alpha} (t-s,y'-w)\right)\right)^{p'} \d w\nonumber\\
				&\qquad +\left(\frac{v^{-1+\frac{\zeta}{\alpha}}|y-y'|^\rho }{(t-s)^{\frac{\rho+1}{\alpha}}}s^{-\frac{\zeta}{\alpha}}\right)^{p'}\int \left( \bar{p}_{\alpha} (s,z-x) \left( \bar{p}_{\alpha} (t-s,y-z)+ \bar{p}_{\alpha} (t-s,y'-z)\right)\right)^{p'} \d z\nonumber\\
				&  \lesssim\left(\frac{v^{-1+\frac{\zeta}{\alpha}}|y-y'|^\rho }{(t-s)^{\frac{\rho+1}{\alpha}}}\left[s^{-\frac{\zeta}{\alpha}}+(t-s)^{-\frac{\zeta}{\alpha}}\right]\right)^{p'}\left[\frac{1}{s^{\frac{d}{\alpha p}}} + \frac{1}{(t-s)^{\frac{d}{\alpha p}}}\right]^{p'}\bar{p}_{\alpha}(t,y-x)^{p'},
			\end{align}
			\textcolor{black}{using \eqref{p-q-convo} for the last inequality as well as $|y-y'|\le t^{\frac 1\alpha} $}.
			Going back to the definition of $\mathcal{T}_{p',q'}^{-\beta,(0,t)} $ \textcolor{black}{in \eqref{DEF_CUT_1}}, we thus obtain, taking $\zeta>-\beta$, 
			\begin{align*}
				\mathcal{T}_{p',q'}^{-\beta,(0,t)}[\mathfrak{q}_{x,y,\textcolor{black}{y'}}^{s,t}] &\lesssim  \frac{|y-y'|^{\rho } }{(t-s)^{\frac{\rho+1}{\alpha}}}\left[\frac{1}{s^{\frac{\zeta }{\alpha}}} + \frac{1}{(t-s)^{\frac{\zeta }{\alpha}}}\right]\left[\frac{1}{s^{\frac{d}{\alpha p}}} + \frac{1}{(t-s)^{\frac{d}{\alpha p}}}\right]
				\bar{p}_{\alpha}(t,y-x)
				\left( \int_0^t v^{q'\frac{\beta+\zeta}{\alpha}-1} \d v\right)^{\frac{1}{q'}}\\
				&\lesssim  \frac{|y-y'|^{\rho } }{(t-s)^{\frac{\rho+1}{\alpha}}}\left[\frac{1}{s^{\frac{\zeta }{\alpha}}} + \frac{1}{(t-s)^{\frac{\zeta }{\alpha}}}\right]\left[\frac{1}{s^{\frac{d}{\alpha p}}} + \frac{1}{(t-s)^{\frac{d}{\alpha p}}}\right]\bar{p}_{\alpha}(t,y-x)t^{\frac{\beta+\zeta}{\alpha}} .
			\end{align*}
			This finally yields together with \eqref{UPPER_CUT_FIRST} and \eqref{DEF_CUT_1},
			\begin{align}
				\mathcal{T}_{p',q'}^{-\beta} [\mathfrak{q}_{x,y,\textcolor{black}{y'}}^{s,t}] \lesssim\frac{|y-y'|^{\rho } \textcolor{black}{t^{\frac \beta\alpha}}}{(t-s)^{\frac{\rho+1}{\alpha}}}\left[
				\frac{t^{\frac{\zeta}{\alpha}}}{s^{\frac{\zeta }{\alpha}}} + \frac{t^{\frac{\zeta}{\alpha}}}{(t-s)^{\frac{\zeta }{\alpha}}}\right]\left[\frac{1}{s^{\frac{d}{\alpha p}}} + \frac{1}{(t-s)^{\frac{d}{\alpha p}}}\right]\bar{p}_{\alpha}(t,y-x).\label{thermic-bound-1}
			\end{align}
		\end{paragraph}
		\begin{paragraph}{Non-thermic part \\[0.5cm]}
			Noticing that 
			\begin{align*}
				\Vert \mathcal{F}(\phi)\star \mathfrak{q}_{x,y,\textcolor{black}{y'}}^{s,t}\Vert_{L^{p'}} & \lesssim 	\Vert \mathcal{F}(\phi)\Vert_{L^{1}} \Vert \mathfrak{q}_{x,y,\textcolor{black}{y'}}^{s,t}\Vert_{L^{p'}},
			\end{align*}
		we see that \eqref{thermic-bound-1} is also a valid bound for the non-thermic part of $\Vert \mathfrak{q}_{x,y,\textcolor{black}{y'}}^{s,t} \Vert_{ \B_{p',q'}^{-\beta}}$. \\
		\end{paragraph}

		This concludes the proof of \eqref{besov-estimate-gamma-and-stable_NO_NORM_Holder} for $j=0$. \textcolor{black}{The case $j=1-\textcolor{black}{\varepsilon/\gamma} $ is similar up to choosing $\zeta>-\beta+\gamma-\textcolor{black}{\varepsilon} $}.\\
		
		Equation \eqref{besov-estimate-gammah-stable-sensi-holder-time} follows from the same proof, using the H\"older regularity in time of the stable kernel instead of its regularity in space (i.e. \eqref{holder-time-palpha} instead of \eqref{holder-space-palpha}), as well as the forward spatial regularity of $\Gamma^h$ instead of that of $\Gamma$ (i.e. \eqref{holder-forward-gammah} instead of \eqref{holder-forward-gamma}).\\
		
		Equations \eqref{besov-estimate-stable-derivees_temps_esp_sensi_holder_esp} and \eqref{besov-estimate-gammah-stable-derivees_temps_esp_sensi_holder_esp} also follow from the same proof, using the standard pointwise estimate $\forall (y,y',z)\in (\R^d)^3,r\in [0,t)$,
		\begin{equation}
		 |\partial_t \nabla_y p_\alpha (t-r,y-z)-\partial_t \nabla_y p_\alpha (t-r,y'-z)| \lesssim \frac{|y-y'|^\rho}{(t-r)^{\frac{\rho}{\alpha}+1}}\left( \bar p_\alpha (t-r,y-z)+ \bar p_\alpha (t-r,y'-z)\right)
		\end{equation}
		and the fact that for the considered time variables, $s\asymp r \asymp \tau_s^h$ to deal with the stable kernel. For the estimate \eqref{besov-estimate-gammah-stable-derivees_temps_esp_sensi_holder_esp}, we also rely on the heat kernel estimate \eqref{holder-forward-gammah} for the forward spatial regularity of $\Gamma^h$\\

		\subsubsection{Proof of \eqref{besov-estimate-gammah-stable-PERTURB_DRIFT} and \eqref{besov-estimate-gammah-stable-PERTURB_DRIFT_SENSI_HOLDER}: \textcolor{black}{controls involving the approximate drift}}\label{subsubsec-lemma-besov-3}
		\paragraph{\textcolor{black}{Proof of \eqref{besov-estimate-gammah-stable-PERTURB_DRIFT}: Hölder regularity involving the one step transition.}}\phantom{Booh}\\
		\textcolor{black}{The proof of \eqref{besov-estimate-gammah-stable-PERTURB_DRIFT} is somehow close to the  one in the previous subsection}. Denote this time
		$$\mathfrak{q}_{x,y}^{s,t,h}\textcolor{black}{(\cdot)}:= \Gamma^h(0,x,\tau_s^h,\cdot)\left[\nabla_y p_\alpha(t-\tau_s^h,y-\cdot) -\nabla_y  p_\alpha \left(t-\tau_s^h,y-\cdot-\int_{\tau_s^h}^s \mathfrak b_h(u,\cdot)\d u\right) \right].$$
		 For $s\in \textcolor{black}{[h,\tau_t^h-h]}$, we want to estimate $\left\|\mathfrak{q}_{x,y}^{s,t,h}\right \|_{\B_{p',q'}^{-\beta}} = \Vert \textcolor{black}{\mathcal F(\phi) \star}\mathfrak{q}_{x,y}^{s,t,h}\Vert_{L^{p'}}+  \mathcal{T}_{p',q'}^{-\beta}\left[\mathfrak{q}_{x,y}^{s,t,h}\right]$.
		As above let us start with the thermic part of the norm.
		\begin{paragraph}{Thermic part \\[0.5cm]}
		
	Let us split the thermic part into two parts:
	\begin{align*}
		\mathcal{T}_{p',q'}^{-\beta} [\mathfrak{q}_{x,y}^{s,t,h}] ^{q'} &= \int_0^t \frac{\d v}{v} v^{\left(1+\frac{\beta}{\alpha}\right)q'} \Vert \partial_v p_\alpha (v,\cdot) \star \mathfrak{q}_{x,y}^{s,t,h} \Vert_{L^{p'}}^{q'} +\int_t^1 \frac{\d v}{v} v^{\left(1+\frac{\beta}{\alpha}\right)q'} \Vert \partial_v p_\alpha (v,\cdot) \star \mathfrak{q}_{x,y}^{s,t,h} \Vert_{L^{p'}}^{q'}\\
		&=:\mathcal{T}_{p',q'}^{-\beta,(0,t)} [\mathfrak{q}_{x,y}^{s,t,h}] ^{q'} + \mathcal{T}_{p',q'}^{-\beta,(t,1)} [\mathfrak{q}_{x,y}^{s,t,h}] ^{q'}.
	\end{align*}
	For the upper part on $(t,1)$, we use the $L^1-L^{p'} $ convolution inequality  to write
	\begin{align*}
		\mathcal{T}_{p',q'}^{-\beta,(t,1)}[\mathfrak{q}_{x,y}^{s,t,h}]^{q'} & \lesssim \int_t^1 v^{q'\left(1+\frac{\beta}{\alpha}\right)-1} \Vert \partial_v p_{\alpha}(v,\cdot)\Vert_{L^1}^{q'}  \Vert \mathfrak{q}_{x,y}^{s,t,h}\Vert_{L^{p'}}^{q'}\d v.
	\end{align*}		
		Observe now carefully from the \textcolor{black}{pointwise control \eqref{CTR_PONCTUEL_BH_INT}} on $\mathfrak{b}_h$, the spatial regularity \eqref{holder-space-palpha} of $p_\alpha$, the heat kernel bound \eqref{aronson-gammah} and the Lebesgue estimate \eqref{p-q-convo} that 
		\begin{align*}
		\Vert \mathfrak{q}_{x,y}^{s,t,h}\Vert_{L^{p'}}&=\left\Vert\Gamma^h(0,x,\tau_s^h,\cdot) \left[\nabla_y p_{\alpha} (t-\tau_s^h,y-\cdot)-\nabla_y  p_\alpha \left(t-\tau_s^h,y-(\cdot+\int_{\tau_s^h}^s \mathfrak b_h(u,\cdot)\d u)\right) \right]\right\Vert_{L^{p'}}\\
		&\lesssim \bar p_\alpha(t,y-x) \frac{\big((s-\tau_s^h)^{1-\frac 1r-\frac d{\alpha p}+\frac\beta\alpha}\|b\|_{L^r-\B_{p,q}^\beta}\big)^\delta}{(t-\tau_s^h)^{\frac 1 \alpha+\frac{\delta}{\alpha}}} \Big(\frac{1}{\textcolor{black}{(\tau_s^h)}^{\frac {d}{\alpha p}}}+\frac{1}{(t-\tau_s^h)^{\frac d{\alpha p}}}\Big),
		\end{align*}
using as well Lemma \ref{Lemma_TRANS_SCHEME} for the last inequality, i.e. the drift component is negligible w.r.t. the increment of the noise on the corresponding considered time intervals. Hence,
		\begin{align}		
		\mathcal{T}_{p',q'}^{-\beta,(t,1)}[\mathfrak{q}_{x,y}^{s,t,h}] & \lesssim \bar p_\alpha(t,y-x)  \frac{\big(h^{\frac{\gamma+1-\beta}{\alpha}}\|b\|_{L^r-\B_{p,q}^\beta}\big)^\delta}{(t-\tau_s^h)^{\frac 1 \alpha+\frac{\delta}{\alpha}}}\left[\frac{1}{(\tau_s^h)^{\frac{d}{\alpha p}}} + \frac{1}{(t-\tau_s^h)^{\frac{d}{\alpha p}}}\right]\left( \int_t^1 v^{\frac{\beta q'}{\alpha}-1}\d v\right)^{\frac 1q'}\notag\\
		& \lesssim \bar p_\alpha(t,y-x)  \frac{h^{\frac{\gamma-\beta}{\alpha}}\|b\|_{L^r-\B_{p,q}^\beta}}{(t-\tau_s^h)^{\frac 1 \alpha}}\left[\frac{1}{(\tau_s^h)^{\frac{d}{\alpha p}}} + \frac{1}{(t-\tau_s^h)^{\frac{d}{\alpha p}}}\right]t^{\frac{\beta}{\alpha}},\label{CUT_HAUT}
	\end{align}
	taking $\delta=1 $ for the last inequality.\\

	Let us now turn to the lower part, for which we again \textcolor{black}{use the same previous cancellation argument}:
	\begin{align}
		\Vert \partial_v  p_\alpha (v,\cdot)\star \mathfrak{q}_{x,y}^{s,t,h}\Vert_{L^{p'}}^{p'}&\nonumber =\int \left|\int \partial_v p_\alpha (v,z-w)\mathfrak{q}_{x,y}^{s,t,h}(w)\d w\right|^{p'} \d z\\
		&  =\int \bigg|\int \partial_v p_\alpha (v,z-w)\big[\mathfrak{q}_{x,y}^{s,t,h}(w)-\mathfrak{q}_{x,y}^{s,t,h}(z)\big]\d w\bigg|^{p'} \d z. \label{cancel_BH}
	\end{align}
We now introduce a diagonal/off-diagonal splitting based on the position of $|w-z| $ w.r.t. $(\tau_s^h)^{\frac 1\alpha}$.
\begin{itemize}
	\item \textbf{Diagonal case: $|w-z|\le (\tau_s^h)^{\frac 1\alpha}$.}
In this case, the contribution into brackets in \eqref{cancel_BH} can be bounded as follows: for all $\delta\in [0,1] $, using \eqref{holder-forward-gammah} and \eqref{aronson-gammah}, we have \textcolor{black}{for $\zeta \in (-\beta,-2\beta\textcolor{black}{\wedge(-\beta+\gamma)}) $,}
\begin{align*}
	|\mathfrak{q}_{x,y}^{s,t,h}(w)-\mathfrak{q}_{x,y}^{s,t,h}(z)|&=\left|\left[\Gamma^h(0,x,\tau_s^h,w)\Big(\nabla_y p_{\alpha}(t-\tau_s^h,y-w)- \nabla_y p_{\alpha}(t-\tau_s^h,y-(w+\int_{\tau_s^h}^s\mathfrak b_h(u,w)\d u))\Big)\right.\right.\notag\\
	& \qquad -\left. \left. \Gamma^h(0,x,\tau_s^h,z)\Big(\nabla_y p_{\alpha}(t-\tau_s^h,y-z)- \nabla_y p_{\alpha}(t-\tau_s^h,y-(z+\int_{\tau_s^h}^s\mathfrak b_h(u,z)\d u))\Big) \right]\right|\\
	&\lesssim |\Gamma^h(0,x,\tau_s^h,w)-\Gamma^h(0,x,\tau_s^h,z)|\frac{\Big(\int_{\tau_s^h}^s\mathfrak b_h(u,w)\d u\Big)^\delta}{(t-\tau_s^h)^{\frac{1+\delta}\alpha}} \bar p_{\alpha}(t-\tau_s^h,y-w)\\
	&\quad+\Gamma^h(0,x,\tau_s^h,z) \Big(\int_0^1 \d \lambda \Big|\nabla_y^2 p_{\alpha}(t-\tau_s^h,y-(z+\lambda\int_{\tau_s^h}^s\mathfrak b_h(u,z)\d u))\int_{\tau_s^h}^s\mathfrak b_h(u,z)\d u	\\
	&\qquad -\nabla_y^2 p_{\alpha}(t-\tau_s^h,y-(w+\lambda \int_{\tau_s^h}^s\mathfrak b_h(u,w)\d u))\int_{\tau_s^h}^s \mathfrak b_h(u,w)\d u) \Big|\Big)\\
	&\lesssim  \frac{|w-z|^\zeta(h^{\frac{\gamma+1-\beta}{\alpha}}\|b\|_{L^r-\B_{p,q}^\beta})^\delta}{(t-\tau_s^h)^{\frac 1\alpha+\frac \delta\alpha}}\Big[\Big( \bar p_\alpha (\tau_s^h,w-x)+\bar p_\alpha (\tau_s^h,z-x)\Big) \frac{1}{(\tau_s^h)^{\frac\zeta \alpha}}\bar p_\alpha(t-\tau_s^h,y-w)\Big]\\
	&\quad+\bar p_\alpha(\tau_s^h,z-x)\frac{|w-z|^\zeta h^{\frac{\gamma+1-\beta}{\alpha}}\|b\|_{L^r-\B_{p,q}^\beta} }{(t-\tau_s^h)^{\frac 2\alpha+\frac  \zeta \alpha}}\Big(\bar p_\alpha(t-\tau_s^h,y-z)+\bar p_\alpha(t-\tau_s^h,y-w )\Big) \\
	&\quad+\bar p_\alpha(\tau_s^h,z-x)p_\alpha(t-\tau_s^h,y-w)\frac{|\int_{\tau_s^h}^s\Big(\mathfrak b_h(u,z)-\mathfrak b_h(u,w)\Big)\d u|^\zeta}{(t-\tau_s^h)^{\frac 2\alpha+\frac \zeta\alpha}}h^{\frac{\gamma+1-\beta}{\alpha}}\|b\|_{L^r-\B_{p,q}^\beta},
\end{align*}		 
using thoroughly \eqref{DRIFT_TO_NEGLECT} from Lemma \ref{Lemma_TRANS_SCHEME}. From \eqref{CTR_PONCTUEL_BH_INT_HOLDER}, we eventually derive (recalling that $t-\tau_s^h\ge h $ and \textcolor{black}{$\gamma+1-\beta-\zeta>0 $}):		 
\begin{align*}		 
		|\mathfrak{q}_{x,y}^{s,t,h}(w)-\mathfrak{q}_{x,y}^{s,t,h}(z)|		 &\lesssim  \frac{|w-z|^\zeta}{(t-\tau_s^h)^{\frac 1\alpha}}\Big[\Big( \bar p_\alpha (\tau_s^h,w-x)+\bar p_\alpha (\tau_s^h,z-x)\Big) \frac{h^{\frac{(\gamma-\beta)\delta} \alpha}}{(\tau_s^h)^{\frac\zeta \alpha}}\bar p_\alpha(t-\tau_s^h,y-w)\\
		 &\quad+\bar p_\alpha(\tau_s^h,z-x)\frac{h^{\frac{\gamma-\beta}\alpha}}{(t-\tau_s^h)^{\frac  \zeta \alpha}}\Big(\bar p_\alpha(t-\tau_s^h,y-z)+\bar p_\alpha(t-\tau_s^h,y-w )\Big) \Big]\\
		& \lesssim \frac{|w-z|^\zeta h^{\frac{(\gamma-\beta)} \alpha}}{(t-\tau_s^h)^{\frac 1\alpha}}\Big[\Big( \bar p_\alpha (\tau_s^h,w-x)+\bar p_\alpha (\tau_s^h,z-x)\Big) \frac{1}{(\tau_s^h)^{\frac\zeta \alpha}}\bar p_\alpha(t-\tau_s^h,y-w)\\
		 &\quad+\bar p_\alpha(\tau_s^h,z-x)\frac{1}{(t-\tau_s^h)^{\frac  \zeta \alpha}}\Big(\bar p_\alpha(t-\tau_s^h,y-z)+\bar p_\alpha(t-\tau_s^h,y-w )\Big) \Big],
\end{align*}
taking $\delta =1$ for the last inequality.
The terms that are \textit{a priori} delicate to integrate in \eqref{cancel_BH} are those emphasizing a \textit{cross dependence on the integration variables}, namely $p_\alpha (\tau_s^h,z-x)\bar p_\alpha(t-\tau_s^h,y-w) $. Anyhow, in the current diagonal regime $|w-z|\le (\tau_s^h)^{\frac 1\alpha} $, it holds that 
$$p_\alpha (\tau_s^h,z-x)\bar p_\alpha(t-\tau_s^h,y-w)
\lesssim p_\alpha (\tau_s^h,w-x)\bar p_\alpha(t-\tau_s^h,y-w),$$
which eventually gives, in the considered diagonal regime:
\begin{align}		 
		|\mathfrak{q}_{x,y}^{s,t,h}(w)-\mathfrak{q}_{x,y}^{s,t,h}(z)|&				 \label{BD_DIAG_DH} \lesssim  \frac{|w-z|^\zeta h^{\frac{(\gamma-\beta)}\alpha} }{(t-\tau_s^h)^{\frac 1\alpha}}\left[\frac{1}{(\tau_s^h)^{\frac\zeta \alpha}}+\frac{1}{(t-\tau_s^h)^{\frac  \zeta \alpha}}\right]\\ & \qquad +\left[ \bar p_\alpha (\tau_s^h,w-x) \bar p_\alpha(t-\tau_s^h,y-w)+\bar p_\alpha (\tau_s^h,z-x) \bar p_\alpha(t-\tau_s^h,y-z)\right]\notag.
\end{align}

\item \textbf{Off-diagonal case: $|w-z|> (\tau_s^h)^{\frac 1\alpha}$.} 
In that case, the contribution into brackets in \eqref{cancel_BH} can be bounded as follows: for all $\zeta\in (-\beta,\textcolor{black}{-\beta+\gamma)} $,
\begin{align*}
	&|\mathfrak{q}_{x,y}^{s,t,h}(w)-\mathfrak{q}_{x,y}^{s,t,h}(z)|\\
	=&\left|\left[\Gamma^h(0,x,\tau_s^h,w)\Big(\nabla_y p_{\alpha}(t-\tau_s^h,y-w)- \nabla_y p_{\alpha}(t-\tau_s^h,y-(w+\int_{\tau_s^h}^s\mathfrak b_h(u,w)\d u))\Big)\right.\right.\notag\\
		 & \qquad -\left. \left. \Gamma^h(0,x,\tau_s^h,z)\Big(\nabla_y p_{\alpha}(t-\tau_s^h,y-z)- \nabla_y p_{\alpha}(t-\tau_s^h,y-(z+\int_{\tau_s^h}^s\mathfrak b_h(u,z)\d u))\Big) \right]\right|\\
		 &\lesssim \frac{|w-z|^\zeta}{(\tau_s^h)^\frac\zeta\alpha}\frac{1}{(t-\tau_s^h)^{\frac 2\alpha}}\Big[  \bar p_\alpha (\tau_s^h,w-x)\textcolor{black}{\int_0^1 \d \lambda} \bar  p_{\alpha}(t-\tau_s^h,y-(w+\lambda \int_{\tau_s^h}^s\mathfrak b_h(u,w)\d u))\\
		 &\qquad \times\Big|\int_{\tau_s^h}^s\mathfrak b_h(u,w)\d u\Big|\\
		 &\qquad+\  \bar p_\alpha (\tau_s^h,z-x)\textcolor{black}{\int_0^1 \d \lambda}\bar  p_{\alpha}(t-\tau_s^h,y-(z+\lambda \int_{\tau_s^h}^s\mathfrak b_h(u,z)\d u))\Big|\int_{\tau_s^h}^s\mathfrak b_h(u,z)\d u\Big|\Big]\\
		 &\le \frac{|w-z|^\zeta}{(\tau_s^h)^\frac\zeta\alpha}\frac{h^{\frac{\gamma-\beta}\alpha}}{(t-\tau_s^h)^\frac 1\alpha}\Big( \bar p_\alpha (\tau_s^h,w-x)\bar  p_{\alpha}(t-\tau_s^h,y-w)+\bar p_\alpha (\tau_s^h,z-x)\bar  p_{\alpha}(t-\tau_s^h,y-z)\Big)
\end{align*}
using \eqref{CTR_PONCTUEL_BH_INT} and \eqref{DRIFT_TO_NEGLECT} for the last inequality, recalling as well that $t-\tau_s^h\ge h $.
\end{itemize} Plugging this control and \eqref{BD_DIAG_DH} into \eqref{cancel_BH}  yields, using the $L^1-L^{p'}$ convolution inequality:
\begin{align*}
		 &\Vert \partial_v p_{\alpha}(v,\cdot) \star \mathfrak{q}_{x,y}^{s,t,h}\Vert_{L^{p'}}\lesssim \frac{v^{-1+\frac\zeta\alpha}\textcolor{black}{h^{\frac{\gamma-\beta}\alpha}}}{(t-\tau_s^h)^\frac1\alpha}\Big[\frac{1}{(\tau_s^h)^\frac\zeta\alpha}+\frac{1}{(t-\tau_s^h)^\frac\zeta\alpha} \Big] \left[\frac{1}{(\tau_s^h)^{\frac{d}{\alpha p}}} + \frac{1}{(t-\tau_s^h)^{\frac{d}{\alpha p}}}\right]\bar p_\alpha(t,y-x).
\end{align*}
This eventually gives:
\begin{align}
\mathcal{T}_{p',q'}^{-\beta,(0,t)}[\mathfrak{q}_{x,y}^{s,t,h}]
		& \lesssim \bar p_\alpha(t,y-x)  \frac{h^{\frac{\gamma-\beta}{\alpha}}\|b\|_{L^r-\B_{p,q}^\beta}}{(t-\tau_s^h)^{\frac 1 \alpha}}\left[\frac{1}{(\tau_s^h)^{\frac{d}{\alpha p}}} + \frac{1}{(t-\tau_s^h)^{\frac{d}{\alpha p}}}\right] \left[\frac{1}{(\tau_s^h)^\frac\zeta\alpha}+\frac{1}{(t-\tau_s^h)^\frac\zeta\alpha} \right]t^{\frac{\beta+\zeta}{\alpha}},\label{CUT_BAS}
\end{align}
which together with \eqref{CUT_HAUT} gives that the bound for the thermic part indeed corresponds to the one of the statement.
\end{paragraph}

\paragraph{Non Thermic part.} Write:
\begin{align*}
&\Vert \mathcal F(\phi)\star \mathfrak{q}_{x,y}^{s,t,h}\Vert_{L^{p'}}
\le \|\mathcal F(\phi)\|_{L^1}\|\mathfrak{q}_{x,y}^{s,t,h}\Vert_{L^{p'}},
\end{align*}
so that from the previous computations (using again \eqref{CTR_PONCTUEL_BH_INT} and \eqref{DRIFT_TO_NEGLECT})
\begin{align*}
&\Vert \mathcal F(\phi)\star \mathfrak{q}_{x,y}^{s,t,h}\Vert_{L^{p'}}\lesssim\bar p_\alpha(t,y-x)  \frac{h^{\frac{\gamma-\beta}{\alpha}}\|b\|_{L^r-\B_{p,q}^\beta}}{(t-\tau_s^h)^{\frac 1 \alpha}}\left[\frac{1}{(\tau_s^h)^{\frac{d}{\alpha p}}} + \frac{1}{(t-\tau_s^h)^{\frac{d}{\alpha p}}}\right].
\end{align*}
From this last inequality, \eqref{CUT_HAUT} and \eqref{CUT_BAS}, the statement \eqref{besov-estimate-gammah-stable-PERTURB_DRIFT} is proved.\\

\paragraph{\textcolor{black}{Proof of \eqref{besov-estimate-gammah-stable-PERTURB_DRIFT_SENSI_HOLDER}: yet another control involving the one step transition}}\phantom{Booh}\\
Let us now turn to the proof of \eqref{besov-estimate-gammah-stable-PERTURB_DRIFT_SENSI_HOLDER}. Using the product rule for Besov spaces \eqref{PR}, we can write
\begin{align}
	&\left\| \bar p_\alpha(\tau_s^h,x,\cdot)\Big(\nabla_y^2  p_\alpha (t-\tau_s^h,y-(\cdot+ \lambda \int_{\tau_s^h}^{s}\mathfrak b_h(u,\cdot)\d u))- \nabla_y^2  p_\alpha (t-\tau_s^h,y'-(\cdot+ \lambda \int_{\tau_s^h}^{s}\mathfrak b_h(u,\cdot)\d u))\Big)\right.\notag\\
	&\qquad\left.\times\int_{\tau_s^h}^{s}\mathfrak b_h(u,\cdot)\d u)\right \|_{\B_{p',q'}^{-\beta}}\notag\\
	&\qquad \lesssim \left\| \bar p_\alpha(\tau_s^h,x,\cdot)\Big(\nabla_y^2  p_\alpha (t-\tau_s^h,y-(\cdot+ \lambda \int_{\tau_s^h}^{s}\mathfrak b_h(u,\cdot)\d u))- \nabla_y^2  p_\alpha (t-\tau_s^h,y'-(\cdot+ \lambda \int_{\tau_s^h}^{s}\mathfrak b_h(u,\cdot)\d u))\Big)\right \|_{\B_{p',q'}^{-\beta}}\nonumber\\ & \qquad \qquad \times \left\Vert \int_{\tau_s^h}^{s}\mathfrak b_h(u,\cdot)\d u\right \|_{\B_{\infty,\infty}^{-\beta+\eps}}
\end{align}
for any $\eps>0$. Then, write, using Lemma \ref{lemma-regularity-mollified-b} twice,
\begin{align*}
	\left\Vert \int_{\tau_s^h}^{s}\mathfrak b_h(u,\cdot)\d u\right \|_{\B_{\infty,\infty}^{-\beta+\eps}} &\asymp \left\Vert \int_{\tau_s^h}^{s}\mathfrak b_h(u,\cdot)\d u\right \|_{L^\infty} + \sup_{z\neq z'\in (\R^d)^2} \frac{|\int_{\tau_s^h}^{s}\mathfrak b_h(u,z)\d u-\int_{\tau_s^h}^{s}\mathfrak b_h(u,z')\d u|}{|z-z'|^{-\beta+\eps}}\\
	&\lesssim h^{\frac{\gamma+1-\beta}{\alpha}}+h^{\frac{\gamma+1-\eps}{\alpha}}\lesssim h^{\frac{\gamma+1-\eps}{\alpha}}.
\end{align*}
Equation \eqref{besov-estimate-gammah-stable-PERTURB_DRIFT_SENSI_HOLDER} then follows from controlling the previous $\B_{p',q'}^{-\beta}$ norm in the same way as for \eqref{besov-estimate-gammah-stable-PERTURB_DRIFT}. Namely, denoting this time 
\begin{align*}
	&\mathfrak{q}_{x,y}^{s,t,h}\textcolor{black}{(\cdot)}\\
	:=&\Gamma^h(0,x,\tau_s^h,\cdot)\left[\nabla_y^2  p_\alpha \left(t-\tau_s^h,y-(\cdot+ \lambda \int_{\tau_s^h}^{s}\mathfrak b_h(u,\cdot)\d u)\right)- \nabla_y^2  p_\alpha \left(t-\tau_s^h,y'-(\cdot+ \lambda \int_{\tau_s^h}^{s}\mathfrak b_h(u,\cdot)\d u)\right)\right]
\end{align*}
we have, in the diagonal regime $|z-w|\leq (\tau_s^h)^{\frac{1}{\alpha}}$, using \eqref{holder-space-palpha} and Lemma \ref{lemma-regularity-mollified-b} profusely,
\begin{align*}
	&|\mathfrak{q}_{x,y}^{s,t,h}(w)-\mathfrak{q}_{x,y}^{s,t,h}(z)|\\
	&=\Gamma^h(0,x,\tau_s^h,w)\left[\nabla_y^2  p_\alpha \left(t-\tau_s^h,y-(w+ \lambda \int_{\tau_s^h}^{s}\mathfrak b_h(u,w)\d u)\right)- \nabla_y^2  p_\alpha \left(t-\tau_s^h,y'-(w+ \lambda \int_{\tau_s^h}^{s}\mathfrak b_h(u,w)\d u)\right)\right]\\
	&\qquad - \Gamma^h(0,x,\tau_s^h,\textcolor{black}{z})\left[\nabla_y^2  p_\alpha \left(t-\tau_s^h,y-(z+ \lambda \int_{\tau_s^h}^{s}\mathfrak b_h(u,z)\d u)\right)- \nabla_y^2  p_\alpha \left(t-\tau_s^h,y'-(z+ \lambda \int_{\tau_s^h}^{s}\mathfrak b_h(u,z)\d u)\right)\right]\\
	&\lesssim \left[\Gamma^h(0,x,\tau_s^h,w)-\Gamma^h(0,x,\tau_s^h,z)\right]\frac{|y-y'|^\delta}{(t-\tau_s^h)^{\textcolor{black}{\frac 2\alpha}+\frac{\delta}{\alpha}}}\left[p_\alpha(t-\tau_s^h,y-w)+ p_\alpha(t-\tau_s^h,y'-w)\right]\\
	& \quad +\bar p_\alpha(\tau_s^h,z-x)\textcolor{black}{\Bigg(|y-y'|\int_0^1\bigg| \nabla_y^3  p_\alpha \left(t-\tau_s^h,y'+\mu(y-y')-(z+ \lambda \int_{\tau_s^h}^{s}\mathfrak b_h(u,z)\d u)\right)}\\
	&\qquad\qquad\qquad\qquad\qquad\qquad \quad \textcolor{black}{- \nabla_y^3  p_\alpha \left(t-\tau_s^h,y'+\mu(y-y')-(w+ \lambda \int_{\tau_s^h}^{s}\mathfrak b_h(u,w)\d u)\right)\bigg|\d \mu\mathbb{I}_{|y-y'|\le (t-\tau_s^h)^{\frac 1\alpha}}}\\
        &\qquad\textcolor{black}{+\frac{|y-y'|^\delta}{(t-\tau_s^h)^{\frac 2\alpha+\frac \delta\alpha+\frac \zeta\alpha}}(\bar p_\alpha(t-\tau_s^h,y-z)+\bar p_\alpha(t-\tau_s^h,y-w)+\bar p_\alpha(t-\tau_s^h,y'-z)+\bar p_\alpha(t-\tau_s^h,y'-w))}\\
        &\qquad \textcolor{black}{\times \left|z-w+ \lambda \int_{\tau_s^h}^{s}(\mathfrak b_h(u,z)-\mathfrak b_h(u,w))\d u)\right|^\zeta \mathbb{I}_{|y-y'|> (t-\tau_s^h)^{\frac 1\alpha}}\Bigg)}\\
	&\lesssim \bar p_\alpha(\tau_s^h,w-x)\frac{|w-z|^\zeta}{(\tau_s^h)^{\frac{\zeta}{\alpha}}}\frac{|y-y'|^\delta}{(t-\tau_s^h)^{\textcolor{black}{\frac 2\alpha}+\frac{\delta}{\alpha}}}\left[p_\alpha(t-\tau_s^h,y-w)+ p_\alpha(t-\tau_s^h,y'-w)\right]\\
	&\quad +\bar p_\alpha(\tau_s^h,z-x) \frac{|y-y'|^\delta}{(t-\tau_s^h)^{\textcolor{black}{\frac 2\alpha}+\frac{\delta+\zeta}{\alpha}}}\left|z-w+ \lambda \int_{\tau_s^h}^{s}(\mathfrak b_h(u,z)-\mathfrak b_h(u,w))\d u)\right|^\zeta \\
	&\qquad\quad \textcolor{black}{\times[\bar p_\alpha(t-\tau_s^h,y-z)+\bar p_\alpha(t-\tau_s^h,y-w)+\bar p_\alpha(t-\tau_s^h,y'-z)+\bar p_\alpha(t-\tau_s^h,y'-w)]}\\
	&\lesssim\bar p_\alpha(\tau_s^h,w-x)\frac{|w-z|^\zeta}{(\tau_s^h)^{\frac{\zeta}{\alpha}}}\frac{|y-y'|^\delta}{(t-\tau_s^h)^{\textcolor{black}{\frac 2\alpha}+\frac{\delta}{\alpha}}}\left[p_\alpha(t-\tau_s^h,y-w)+ p_\alpha(t-\tau_s^h,y'-w)\right]\\
	&\quad +\bar p_\alpha(\tau_s^h,z-x) \frac{|y-y'|^\delta}{(t-\tau_s^h)^{\textcolor{black}{\frac 2\alpha}+\frac{\delta+\zeta}{\alpha}}}\left|z-w
	\right|^\zeta [\bar p_\alpha (t-\tau_s^h,y-z)+\bar p_\alpha (t-\tau_s^h,y'-z)+\bar p_\alpha (t-\tau_s^h,y-w)\\
	&\qquad\qquad+\textcolor{black}{\bar p_\alpha (t-\tau_s^h,y'-w)}].
\end{align*}
\textcolor{black}{We can again get rid of the cross terms in the integration variable in the above inequality recalling that in the considered diagonal regime $|w-z|\le (\tau_s^h)^{\frac 1\alpha} $ it holds that $\bar p_\alpha(\tau_s^h,z-x)\lesssim \bar p_\alpha(\tau_s^h,w-x)$.
The rest of the proof is similar to the one of \eqref{besov-estimate-gammah-stable-PERTURB_DRIFT}}.

 \begin{paragraph}{\textbf{Declaration of competing interests}\\[0.3cm]}
 	
 	The authors declare that they have no known competing financial interests or personal relationships that could have appeared to influence the work reported in this paper.
 \end{paragraph}

\appendix
\section{Proof of the heat-kernel estimates of Proposition \ref{THE_PROP}}
\label{APP_HK}

\subsection{Heat-kernel bounds for the Euler scheme: proof of \eqref{ineq-density-scheme}.}
\label{PROOF_FOR_DENS_SCHEME}
	
%
	\textbf{Duhamel representation for the density of the scheme.}
	 Let us first prove the Duhamel representation \eqref{duhamel-scheme} for the density of the scheme. Let $t\in (t_k,T]$, $\phi$ be a ${\cal C}^2$ function with compact support and $v(s,y)={\mathbb 1}_{s<t}p_\alpha (t-s,\cdot)\star \phi(y)+{\mathbb 1}_{s=t}\phi(y)$. It is well known that $v$ is ${\cal C}^{1,2}$ on $[0,t]\times \R^d$ and solves the Feynman-Kac partial differential equation
         $$\forall (s,y)\in[0,t)\times\R^d,\;\partial_s v(s,y)+\mathcal{L}^\alpha v(s,y)=0 .$$ From Itô's formula between $t_k$ and $t$ applied along the Euler scheme  $(X_s^h)_{s\in[t_k,T]}$  started from $X^h_{t_k}=x$ and  with integral dynamics \eqref{euler-scheme-besov}, we obtain :
		\begin{equation*}
			\phi(X_t^h)=v(t_k,x)+M_{t_k,t}^h+\int_{t_k}^t \nabla v(s,X_s^h)\cdot \mathfrak b_h \left(s,X_{\tau_s^h}^h\right) \d s, 
		\end{equation*}
		where $M_{t_k,t}^h=\int_{t_k}^t \int_{\R^d\backslash\{0\}}\Big(v({s,} X_{s^-}^h+x)-v({s,} X_{s^-}^h)\Big) \tilde N({\rm d}s,{\rm d} x)$, in which $\tilde N $ is the compensated Poisson measure associated with $Z$ is a martingale.
		Take now the expectation  and use the Fubini theorem to derive
		\begin{equation*}
			\int \phi(y)\Gamma^h (t_k,x,t,y)\d y=v(t_k,x)+\int_{t_k}^t \E_{t_k,x} \left[ \nabla v(s,X_s^h)\cdot \mathfrak b_h \left(s,X_{\tau_s^h}^h\right)\right]\d s.
		\end{equation*}
		From the definition of $v$, it follows
		\begin{align*}
			&\int \phi(y)\Gamma^h (t_k,x,t,y)\d y\\ & \qquad=\int p_\alpha (t-t_k,x-y)\phi(y) \d y+\int \phi(y)\int_{t_k}^t   \E_{t_k,x} \left[ \nabla_y p_\alpha (t-s,X_s^h-y) \cdot \mathfrak b_h \left(s,X_{\tau_s^h}^h\right)\right] \d s\d y.
		\end{align*}
		The function $\phi$ being arbitrary and $p_\alpha(t-s,\cdot)$ even, we obtain that for almost all $ y\in \R^d$, 
		\begin{equation}\label{duhamel-inproof}
			\Gamma^h (t_k,x,t,y)= p_\alpha (t-t_k,x-y)-\int_{t_k}^t   \E_{t_k,x} \left[ \nabla_y p_\alpha (t-s,y-X_s^h) \cdot \mathfrak b_h \left(s,X_{\tau_s^h}^h\right)\right] \d s.
		\end{equation}
It is also plain from Lemmas \ref{lemma-stable-sensitivities} and \ref{lemma-regularity-mollified-b} to prove that $\Gamma^h (t_k,x,t,\cdot) $ is continuous and therefore
 \eqref{duhamel-inproof} actually holds for all $y\in \R^d$. This concludes the proof of \eqref{duhamel-scheme}.

	 We will now actually prove the H\"older regularity of $\Gamma^h$ in the forward space variable.
	Similarly to the proof of the main theorem, we will control $\left\Vert \frac{\Gamma^h(0,x,t,\cdot)}{\bar p_\alpha(t,\cdot-x)} \right\Vert_{\B_{\infty,\infty}^\rho}$ for $\rho \in (-\beta,\gamma-\beta)$ using a circular argument.\\
	
	\textbf{Control of the supremum norm}\\

	We write the Duhamel formula as follows:
	\begin{align*}
		\Gamma^h (0,x,t,y) &=p_\alpha(t,y-x) -\int_0^h \E_{0,x}\left[\mathfrak b_h(s,X^h_{\tau_s^h})\cdot\nabla_y  p_\alpha(t-s,y-X^h_s)\right]\d s\\
		&\qquad - \int_{h}^{t}\E_{0,x}\left[\mathfrak b_h(s,X^h_{\tau_s^h})\cdot\nabla_y  p_\alpha(t-s,y-X^h_{s})\right]\d s \\
		&=: \Delta_{1}(y)+\Delta_2 (y)+\Delta_3 (y).
	\end{align*}
	For $\Delta_2$, using \eqref{CTR_PONCTUEL_BH},  \eqref{DRIFT_TO_NEGLECT} and \eqref{derivatives-palpha},
	\begin{align}
		|\Delta_2| &=\left| \int_0^h \int \mathfrak b_h(s,x)\cdot p_\alpha \left(s,z-x-\int_0^s \mathfrak{b}_h(r,x)\d r \right)\nabla_y  p_\alpha(t-s,y-z)\d z\d s\right|\nonumber\\
		& \lesssim \int_0^h \int s^{-\frac{d}{\alpha p}+\frac{\beta}{\alpha}}\Vert b(s,\cdot)\Vert_{\B_{p,q}^\beta} \bar{p}_{\alpha}\left(s,z-x\right) (t-s)^{-\frac{1}{\alpha}}  \bar{p}_{\alpha}(t-s,y-z)\d z\d s\nonumber\\
		& \lesssim \bar{p}_{\alpha} (t,y-x)\Vert b\Vert_{L^r-\B_{p,q}^\beta}\left(\int_0^h s^{-\frac{dr'}{\alpha p}+\frac{\beta r'}{\alpha}}  (t-s)^{-\frac{r'}{\alpha}}  \d s\right)^{\frac{1}{r'}}\nonumber\\
		& \lesssim \bar{p}_{\alpha} (t,y-x) h^{1-\frac{1}{r}-\frac{d}{\alpha p}+\frac{\beta}{\alpha}}t^{-\frac{1}{\alpha}} \lesssim\bar{p}_{\alpha} (t,y-x)h^{\frac{\gamma}{\alpha}}h^{\frac{1}{\alpha}-\frac{\beta}{\alpha}}t^{-\frac{1}{\alpha}}\lesssim \bar{p}_{\alpha} (t,y-x)h^{\frac{\gamma}{\alpha}}t^{-\frac{\beta}{\alpha}}.\label{delta2-supnorm}
	\end{align}
	For $\Delta_3$, using the harmonicity of the stable kernel, for $\rho>-\beta$,
	\begin{align*}
		|\Delta_3| & =\left|\int_{h}^{t}\int \Gamma^h (0,x,\tau_s^h,z)\mathfrak b_h(s,z)\cdot\nabla_y  p_\alpha \left(t-\tau_s^h,y-z-\int_{\tau_s^h}^{s}\mathfrak{b}_h(r,z)\d r\right)\d z\d s\right|\\
		&\lesssim \int_{h}^{t} \left\Vert \frac{\Gamma^h (0,x,\tau_s^h,\cdot)}{\bar{p}_{\alpha}(\tau_s^h,\cdot-x)}\right\Vert_{\B_{\infty,\infty}^\rho} \Vert \mathfrak b_h(s,\cdot)\Vert_{\B_{p,q}^\beta} \\ & \qquad \qquad \qquad \times \left\Vert \bar{p}_{\alpha}(\tau_s^h,\cdot-x) \nabla_y  {p}_{\alpha} \left(t-\tau_s^h,y-\cdot-\int_{\tau_s^h}^{s}\mathfrak{b}_h(r,\cdot)\d r\right)\right\Vert_{\B_{p',q'}^{-\beta}}\d s.
	\end{align*}
	Using a triangular inequality, \eqref{besov-estimate-gammah-stable-PERTURB_DRIFT} (in which we can trivially replace $\Gamma^h$ with $p_\alpha$)  and \eqref{besov-estimate-stable}, we have, for any $\zeta \in (\textcolor{black}{-\beta},1]$,
	\begin{align}
		\nonumber&\left\Vert \bar{p}_{\alpha}(\tau_s^h,\cdot-x) \nabla_y  {p}_{\alpha} \left(t-\tau_s^h,y-\cdot-\int_{\tau_s^h}^{s}\mathfrak{b}_h(r,\cdot)\d r\right)\right\Vert_{\B_{p',q'}^{-\beta}}\\
		\nonumber& \lesssim \left\Vert \bar{p}_{\alpha}(\tau_s^h,\cdot-x) \left[\nabla_y  p_\alpha \left(t-\tau_s^h,y-\cdot-\int_{\tau_s^h}^{s}\mathfrak{b}_h(r,\cdot)\d r\right)-\nabla_y  {p}_{\alpha} \left(t-\tau_s^h,y-\cdot\right)\right]\right\Vert_{\B_{p',q'}^{-\beta}}\\ \nonumber& \qquad +\left\Vert \bar{p}_{\alpha}(\tau_s^h,\cdot-x) \nabla_y  {p}_{\alpha} \left(t-\tau_s^h,y-\cdot\right)\right\Vert_{\B_{p',q'}^{-\beta}}\\
		&\lesssim \bar{p}_{\alpha}(t,y-x)\left(1+h^{\frac{\gamma-\beta}{\alpha}}\right)\frac{t^{\frac{\beta}{\alpha}}}{ 
		(t-\tau_s^h)^{\frac{1}{\alpha}}} \left[\frac{1}{(\tau_s^h)^{\frac{d}{\alpha p}}} + \frac{1}{(t-\tau_s^h)^{\frac{d}{\alpha p}}}\right]\left[1+\frac{t^{\frac{\zeta}{\alpha}}}{(\tau_s^h)^{\frac{\zeta}{\alpha}}}+\frac{t^{\frac{\zeta}{\alpha}}}{(t-\tau_s^h)^{\frac{\zeta}{\alpha}}}\right].\label{triangular-stable-estimates}
	\end{align}
	Set now, for $s\in (0,T]$,
	\begin{equation}
		\tilde{g}_{h,\rho}(s):=\left\Vert \frac{\Gamma^h(0,x,s,\cdot)}{\bar{p}_{\alpha}(s,\cdot-x)}\right\Vert_{L^\infty} + s^{\frac{\rho}{\alpha}} \sup_{z\neq z'\in (\R^d)^2} \left| \frac{\frac{\Gamma^h(0,x,s,z)}{\bar{p}_{\alpha}(s,z-x)}-\frac{\Gamma^h(0,x,s,z')}{\bar{p}_{\alpha}(s,z'-x)}}{|z-z'|^\rho}\right| \gtrsim s^{\frac{\rho}{\alpha}} \left\Vert \frac{\Gamma^h (0,x,s,\cdot)}{\bar{p}_{\alpha}(s,\cdot-x)}\right\Vert_{\B_{\infty,\infty}^\rho}.\label{DEF_TILDE_G_H_RHO}
	\end{equation}
	Plugging this and \eqref{triangular-stable-estimates} into $\Delta_3$ along with \eqref{CTR_BESOV_BH} yields
	\begin{align*}
		|\Delta_3| &\lesssim \bar{p}_{\alpha}\left(t,y-x\right) \int_{h}^{t} \tilde{g}_{h,\rho}(\tau_s^h)\Vert b(s,\cdot)\Vert_{\B_{p,q}^\beta} \\ & \qquad \qquad \qquad \times \frac{t^{\frac{\beta}{\alpha}}}{s^{\frac{\rho}{\alpha}}(t-\tau_s^h)^{\frac{1}{\alpha}}} \left[\frac{1}{(\tau_s^h)^{\frac{d}{\alpha p}}} + \frac{1}{(t-\tau_s^h)^{\frac{d}{\alpha p}}}\right]\left[1+\frac{t^{\frac{\zeta}{\alpha}}}{(\tau_s^h)^{\frac{\zeta}{\alpha}}}+\frac{t^{\frac{\zeta}{\alpha}}}{(t-\tau_s^h)^{\frac{\zeta}{\alpha}}}\right]\d s.
	\end{align*} 
	\textcolor{black}{Note} that, on the considered time interval, $s\asymp \tau_s^h$, yielding 
	\begin{align}
		|\Delta_3| \lesssim &\bar{p}_{\alpha}\left(t,y-x\right) \textcolor{black}{\|b\|_{L^r-\B_{p,q}^\beta}}\sup_{s\in (h,T]} \tilde{g}_{h,\rho}(s) \notag\\
		&\times \textcolor{black}{t^{\frac{\beta+\zeta}{\alpha}}}\left(\int_0^t \frac{1}{s^{\frac{\rho \textcolor{black}{r'}}{\alpha}}(t-s)^{\frac{\textcolor{black}{r'}}{\alpha}}}  \left[\frac{1}{s^{\frac{d}{\alpha p}}} + \frac{1}{(t-s)^{\frac{d}{\alpha p}}}\right]^{\textcolor{black}{r'}}\left[\frac{1}{s^{\frac{\zeta}{\alpha}}}+\frac{1}{(t-s)^{\frac{\zeta}{\alpha}}}\right]^{\textcolor{black}{r'}}\d s\right)^{\frac {1}{r'}}\nonumber\\
		\lesssim&\bar{p}_{\alpha} (t,y-x)t^{\frac{\gamma-\beta-\rho}{\alpha}}  \sup_{s\in (h,T]} \tilde{g}_{h,\rho}(s),\label{delta3-supnorm}
	\end{align} 
	\textcolor{black}{where we used Lemma \ref{lm:sing_int} with $\aa=\rho/\alpha,\bb=1/\alpha,\cc=d/(\alpha p),\dd=\zeta/\alpha $ for the last inequality}.
	Gathering \eqref{delta2-supnorm} and \eqref{delta3-supnorm}, \textcolor{black}{using as well} \eqref{derivatives-palpha} , we get recalling that $\rho$ can be chosen so that $\gamma-\beta-\rho>0$,
	\begin{equation}
		\label{maj-supnorm} \left\Vert \frac{\Gamma^h(0,x,t,\cdot)}{\bar{p}_{\alpha}(t,\cdot-x)}\right\Vert_{L^\infty} \lesssim 1+ T^{\frac{\gamma-\beta-\rho}{\alpha}}\sup_{s\in (h,T]} \tilde{g}_{h,\rho}(s).
	\end{equation}
	\textbf{Control of the H\"older modulus\\}
	
	We will now control the $\rho$-H\"older modulus of $\Gamma^h (0,x,t,\cdot)/\bar{p}_{\alpha}(t,\cdot-x)$ in the diagonal regime, i.e. for $(y,y)\in (\R^d)^2$ such that $|y-y'|\leq t^{\frac{1}{\alpha}}$ (since otherwise the required control is trivial). Similarly to the proof of the main theorem, let us write, using \eqref{holder-space-palpha} \textcolor{black}{(which readily extends to $\bar p_\alpha $ instead of $p_\alpha $)},
	\begin{align}
		\left|\frac{\Gamma^h(0,x,t,y)}{\bar p_\alpha(t,y-x)}-\frac{\Gamma^h(0,x,t,y')}{\bar p_\alpha(t,y'-x)}\right| &\leq \left|\frac{\Gamma^h(0,x,t,y)-\Gamma^h(0,x,t,y')}{\bar p_\alpha(t,y-x)}\right| + \Gamma^h(0,x,t,y')\left|\textcolor{black}{\frac{1}{\bar p_\alpha(t,y-x)}-\frac 1{\bar p_\alpha(t,y'-x)}}\right|\nonumber\\
		&\lesssim\left|\frac{\Gamma^h(0,x,t,y)-\Gamma^h(0,x,t,y')}{\bar p_\alpha(t,y-x)}\right| + \frac{|y-y'|^\rho}{t^{\frac{\rho}{\alpha}}} \left\Vert\frac{\Gamma^h(0,x,t,\cdot)}{\bar p_\alpha(t,\cdot-x)}\right\Vert_{L^\infty}.\label{triangu-holdermod}
	\end{align}
	The error expansion for $|\Gamma^h(0,x,t,y)-\Gamma^h(0,x,t,y')|$ writes
	\begin{align*}
		\Gamma^h (0,x,t,y) -\Gamma^h(0,x,t,y')&=p_\alpha(t,y-x)-p_\alpha(t,y'-x)\\
		&\qquad -\int_0^h \E_{0,x}\left[\mathfrak b_h(s,X^h_{\tau_s^h})\cdot(\nabla_y  p_\alpha(t-s,y-X^h_s)-\nabla_{y'}  p_\alpha(t-s,y'-X^h_s))\right]\d s\\
		&\qquad - \int_{h}^{t}\E_{0,x}\left[\mathfrak b_h(s,X^h_{\tau_s^h})\cdot(\nabla_y  p_\alpha(t-s,y-X^h_s)-\nabla_{y'}  p_\alpha(t-s,y'-X^h_s))\right]\d s \\
		&=: \Delta_{1}(y,y')+\Delta_2 (y,y')+\Delta_3 (y,y'),
	\end{align*}
\textcolor{black}{where with a slight abuse of notation we do not emphasize the dependence of these quantities on $x,t$ and \textit{a priori} on $h$ (we will actually prove that the estimates are uniform w.r.t. this last parameter)}.
	For $\Delta_1$, using \eqref{holder-space-palpha}, we have
	\begin{align}
		|\Delta_{1}\textcolor{black}{(y,y')}| &= |p_\alpha(t,y-x)-p_\alpha(t,y'-x)|\lesssim \frac{|y-y'|^\rho}{t^{\frac{\rho}{\alpha}}} \bar p_\alpha(t,y-x).\label{delta1-holdermod}
	\end{align}
	For $\Delta_2$, using \eqref{CTR_PONCTUEL_BH},  \eqref{DRIFT_TO_NEGLECT} and \eqref{holder-space-palpha},
	\begin{align}
		|\Delta_2\textcolor{black}{(y,y')}| &=\left| \int_0^h \int \mathfrak b_h(s,x)\cdot p_\alpha \left(s,z-x-\int_0^s \mathfrak{b}_h(r,x)\d r \right)[\nabla_y  p_\alpha(t-s,y-z)-\nabla_{y'} p_\alpha(t-s,y'-z)]\d z\d s\right|\nonumber\\
		& \lesssim \int_0^h \int s^{-\frac{d}{\alpha p}+\frac{\beta}{\alpha}}\Vert b(s,\cdot)\Vert_{\B_{p,q}^\beta} \bar{p}_{\alpha}\left(s,z-x\right) \frac{|y-y'|^\rho}{(t-s)^{\frac{\rho+1}{\alpha}}} [\bar{p}_{\alpha}(t-s,y-z)+\bar{p}_{\alpha}(t-s,y'-z)]\d z\d s \nonumber \\
		& \lesssim \bar{p}_{\alpha}(t,y-x)|y-y|^\rho  \Vert b\Vert_{L^r-\B_{p,q}^\beta}\left(\int_0^h s^{-\frac{dr'}{\alpha p}+\frac{\beta r'}{\alpha}}  (t-s)^{-\frac{r'(\rho+1)}{\alpha}}  \d s\right)^{\frac{1}{r'}}\nonumber\\
		& \lesssim\bar{p}_{\alpha} (t,y-x) |y-y|^\rho h^{1-\frac{1}{r}-\frac{d}{\alpha p}+\frac{\beta}{\alpha}}t^{-\frac{1+\rho}{\alpha}}\nonumber\\
		&\lesssim \bar{p}_{\alpha} (t,y-x)|y-y|^\rho h^{\frac{\gamma}{\alpha}}h^{\frac{1}{\alpha}-\frac{\beta}{\alpha}}t^{-\frac{1+\rho}{\alpha}}\lesssim \bar{p}_{\alpha} (t,y-x)\frac{|y-y|^\rho }{t^{\frac{\rho}{\alpha}}}h^{\frac{\gamma}{\alpha}}t^{-\frac{\beta}{\alpha}}.\label{delta2-holdermod}
	\end{align}
	For $\Delta_3$, using the harmonicity of the stable kernel, for $\rho>-\beta$,
	\begin{align*}
		&|\Delta_3\textcolor{black}{(y,y')}| \\
		 =&\bigg|\int_{h}^{t}\int \Gamma^h (0,x,\tau_s^h,z)\mathfrak b_h(s,z)\\ &  \quad \cdot \left[\nabla_y  p_\alpha \left(t-\tau_s^h,y-z-\int_{\tau_s^h}^{s}\mathfrak{b}_h(r,z)\d r\right)-\nabla_{y'}  p_\alpha \left(t-\tau_s^h,y'-z-\int_{\tau_s^h}^{s}\mathfrak{b}_h(r,z)\d r\right)\right]\d z\d s\bigg|\\
		\lesssim& \int_{h}^{t} \left\Vert \frac{\Gamma^h (0,x,\tau_s^h,\cdot)}{\bar{p}_{\alpha}(\tau_s^h,\cdot-x)}\right\Vert_{\B_{\infty,\infty}^\rho} \Vert \mathfrak b_h(s,\cdot)\Vert_{\B_{p,q}^\beta} \\ 
		&   \times \left\Vert \textcolor{black}{\bar p_\alpha}(\tau_s^h,\cdot-x) \left[\nabla_y  p_\alpha \left(t-\tau_s^h,y-\cdot-\int_{\tau_s^h}^{s}\mathfrak{b}_h(r,\cdot)\d r\right)-\nabla_{y'}  p_\alpha \left(t-\tau_s^h,y'-\cdot-\int_{\tau_s^h}^{s}\mathfrak{b}_h(r,\cdot)\d r\right)\right]\right\Vert_{\B_{p',q'}^{-\beta}}\!\!\!\!\d s.
	\end{align*}
	\textcolor{black}{Similarly to the computations performed to prove \eqref{besov-estimate-gammah-stable-PERTURB_DRIFT_SENSI_HOLDER}}  (in which we again take $p_\alpha$ instead of $\Gamma^h$) and the definition of $\tilde g_{h,\rho}(s)$ in \eqref{DEF_TILDE_G_H_RHO} along with \eqref{CTR_BESOV_BH} yields, for $\zeta\in (-\beta,1],$
	\begin{align}
		|\Delta_3\textcolor{black}{(y,y')}| &\lesssim \bar{p}_{\alpha}\left(t,y-x\right) \int_{h}^{t} \tilde{g}_{h,\rho}(\tau_s^h)\Vert b(s,\cdot)\Vert_{\B_{p,q}^\beta} \nonumber\\ & \qquad \qquad \qquad \times \frac{t^{\frac{\beta}{\alpha}}|y-y'|^\rho \textcolor{black}{(1+h^{\frac{\gamma-\beta}{\alpha}})}}{s^{\frac{\rho}{\alpha}}(t-\tau_s^h)^{\frac{1+\rho}{\alpha}}} \left[\frac{1}{(\tau_s^h)^{\frac{d}{\alpha p}}} + \frac{1}{(t-\tau_s^h)^{\frac{d}{\alpha p}}}\right]\left[\frac{t^{\frac{\zeta}{\alpha}}}{(\tau_s^h)^{\frac{\zeta}{\alpha}}}+\frac{t^{\frac{\zeta}{\alpha}}}{(t-\tau_s^h)^{\frac{\zeta}{\alpha}}}\right]\d s\nonumber\\
		&\lesssim\bar{p}_{\alpha} (t,y-x)|y-y'|^\rho t^{\frac{\gamma-\beta-2\rho}{\alpha}}  \sup_{s\in (h,T]} \tilde{g}_{h,\rho}(s),\label{delta3-holdermod}
	\end{align} 
	using the fact that $s\asymp \tau_s^h$ on the considered time interval and \textcolor{black}{Lemma \ref{lm:sing_int} with $\aa=\rho/\alpha,\bb=(1+\rho)/\alpha,\cc=d/(\alpha p),\dd=\zeta/\alpha $ for the last inequality}. We then have, plugging \eqref{delta1-holdermod},  \eqref{delta2-holdermod} and \eqref{delta3-holdermod} into \eqref{triangu-holdermod},
	\begin{equation}
		\left|\frac{\Gamma^h(0,x,t,y)}{\bar p_\alpha(t,y-x)}-\frac{\Gamma^h(0,x,t,y')}{\bar p_\alpha(t,y'-x)}\right| \lesssim \frac{|y-y'|^\rho}{t^{\frac{\rho}{\alpha}}} T^{\frac{\textcolor{black}{\gamma-\beta-\rho}}{\alpha}} \sup_{s\in (h,T]} \tilde{g}_{h,\rho}(s),
	\end{equation}
	which, together with \eqref{maj-supnorm} and \textcolor{black}{recalling that $\gamma-\beta-\rho $ can be chosen to be positive},  concludes, \textcolor{black}{through a circular argument}, the proof of \eqref{ineq-density-scheme}, \textcolor{black}{since it gives that $\tilde{g}_{h,\rho}(s)$ is bounded for $s\in (h,T] $. This in turn gives \eqref{aronson-gammah} and \eqref{holder-forward-gammah} on the corresponding time set. The extension to $s\in (0,h)$ is direct since in that case the statement follows from a direct computation on the densities.}\hfill $\square $

\subsection{Heat-kernel bounds for the diffusion: proof of \eqref{ineq-density-diff}.}
\label{PROOF_FOR_DENS_DIFF} We actually extend here the result established in \cite{Fit23} which stated Hölder continuity in the forward spatial variable for $\rho\in (-\beta,-\beta+\frac \gamma2) $. We actually extend it here to $\rho\in (-\beta,-\beta+ \gamma) $. This is crucial for our convergence analysis since this regularity partly rules the convergence rate, see in particular the former term $\Delta_3 $ in the error analysis. It seems as well coherent that we obtain for the diffusion the same forward spatial regularity that we proved in the previous section for the Euler scheme. The control for the supremum norm of the quotient follows from \cite{Fit23}. It thus remains to refine the Holder norm control. Namely, it holds that 
\begin{align}\label{BD_SUP_DIFF}
\left\|\frac{\Gamma(0,x,t,\cdot)}{\bar p_\alpha(t,\cdot-x)}\right\|_{L^\infty}\lesssim C. 
\end{align}
Let us now write inthe diagonal regime, i.e. for $(y,y)\in (\R^d)^2$ such that $|y-y'|\leq t^{\frac{1}{\alpha}}$ (since otherwise the required control is trivial), similarly to \eqref{triangu-holdermod},
	\begin{align}
		\left|\frac{\Gamma(0,x,t,y)}{\bar p_\alpha(t,y-x)}-\frac{\Gamma(0,x,t,y')}{\bar p_\alpha(t,y'-x)}\right| &\leq \left|\frac{\Gamma(0,x,t,y)-\Gamma(0,x,t,y')}{\bar p_\alpha(t,y-x)}\right| + \Gamma(0,x,t,y')\left|\textcolor{black}{\frac{1}{\bar p_\alpha(t,y-x)}-\frac 1{\bar p_\alpha(t,y'-x)}}\right|\nonumber\\
		&\lesssim\left|\frac{\Gamma(0,x,t,y)-\Gamma(0,x,t,y')}{\bar p_\alpha(t,y-x)}\right| + \frac{|y-y'|^\rho}{t^{\frac{\rho}{\alpha}}} \left\Vert\frac{\Gamma(0,x,t,\cdot)}{\bar p_\alpha(t,\cdot-x)}\right\Vert_{L^\infty}.\label{triangu-holdermod_DIFF}
	\end{align}
	From the Duhamel expansion, the expansion for $\Gamma(0,x,t,y)-\Gamma(0,x,t,y')$ writes
	\begin{align*}
		\Gamma (0,x,t,y) -\Gamma(0,x,t,y')&=p_\alpha(t,y-x)-p_\alpha(t,y'-x)\\
		&\qquad - \int_{0}^{t}\E_{0,x}\left[b(s,X_{s})\cdot(\nabla_y  p_\alpha(t-s,y-X_s)-\nabla_{y'}  p_\alpha(t-s,y'-X_s))\right]\d s \\
		&=: \Delta_{1}(y,y')+\Delta_2 (y,y'),
	\end{align*}
\textcolor{black}{where with a slight abuse of notation we do not emphasize the dependence of these quantities on $x,t$.}.
	The term $\Delta_1$ was already controlled for the scheme and leads to the bound \eqref{delta1-holdermod}.
	For $\Delta_2$, write for $\rho>-\beta$,
	\begin{align*}
		&|\Delta_2\textcolor{black}{(y,y')}| \\
		 =&\bigg|\int_{0}^{t}\int \Gamma (0,x,s,z) b(s,z) \cdot \left[\nabla_y  p_\alpha \left(t-s,y-z\right)-\nabla_{y'}  p_\alpha \left(t-s,y'-z\right)\right]\d z\d s\bigg|\\
		\lesssim& \int_{0}^{t} \left\Vert \frac{\Gamma(0,x,s,\cdot)}{\bar{p}_{\alpha}(s,\cdot-x)}\right\Vert_{\B_{\infty,\infty}^\rho} \Vert b(s,\cdot)\Vert_{\B_{p,q}^\beta} 
		\left\Vert \textcolor{black}{\bar p_\alpha}(s,\cdot-x) \left[\nabla_y  p_\alpha \left(t-s,y-\cdot\right)-\nabla_{y'}  p_\alpha \left(t-s,y'-\cdot\right)\right]\right\Vert_{\B_{p',q'}^{-\beta}}\!\!\!\!\d s.
	\end{align*}

We would get  similarly to \eqref{big-lemma-2} 
that for $\zeta\in (-\beta,1]$,
\begin{align*}
	&	\left\Vert \textcolor{black}{\bar p_\alpha}(s,\cdot-x) \left[\nabla_y  p_\alpha \left(t-s,y-\cdot\right)-\nabla_{y'}  p_\alpha \left(t-s,y'-\cdot\right)\right]\right\Vert_{\B_{p',q'}^{-\beta}}\\
		\lesssim &t^{\frac\beta\alpha}\frac{|y-y'|^\rho}{(t-s)^{\frac{1+\rho}{\alpha}}}\left[\frac{1}{s^{\frac{d}{\alpha p}}} + \frac{1}{(t-s)^{\frac{d}{\alpha p}}}\right]\left[\frac{t^{\frac{\zeta}{\alpha}}}{s^{\frac{\zeta}{\alpha}}}+\frac{t^{\frac{\zeta}{\alpha}}}{(t-s)^{\frac{\zeta}{\alpha}}}\right].
\end{align*}
	Setting now for $s\in (0,T]$,
	\begin{equation}
		\tilde{g}_{\rho}(s):=\left\Vert \frac{\Gamma(0,x,s,\cdot)}{\bar{p}_{\alpha}(s,\cdot-x)}\right\Vert_{L^\infty} + s^{\frac{\rho}{\alpha}} \sup_{z\neq z'\in (\R^d)^2} \left| \frac{\frac{\Gamma(0,x,s,z)}{\bar{p}_{\alpha}(s,z-x)}-\frac{\Gamma(0,x,s,z')}{\bar{p}_{\alpha}(s,z'-x)}}{|z-z'|^\rho}\right| \gtrsim s^{\frac{\rho}{\alpha}} \left\Vert \frac{\Gamma^h (0,x,s,\cdot)}{\bar{p}_{\alpha}(s,\cdot-x)}\right\Vert_{\B_{\infty,\infty}^\rho},\label{DEF_TILDE_G_RHO}
	\end{equation}
we therefore get:
	\begin{align}
		|\Delta_2\textcolor{black}{(y,y')}| &\lesssim \bar{p}_{\alpha}\left(t,y-x\right) \int_{0}^{t} \tilde{g}_{\rho}(s)\Vert b(s,\cdot)\Vert_{\B_{p,q}^\beta} \nonumber\\ & \qquad \qquad \qquad \times \frac{t^{\frac{\beta}{\alpha}}|y-y'|^\rho }{s^{\frac{\rho}{\alpha}}(t-s)^{\frac{1+\rho}{\alpha}}} \left[\frac{1}{s^{\frac{d}{\alpha p}}} + \frac{1}{(t-s)^{\frac{d}{\alpha p}}}\right]\left[\frac{t^{\frac{\zeta}{\alpha}}}{s^{\frac{\zeta}{\alpha}}}+\frac{t^{\frac{\zeta}{\alpha}}}{(t-s)^{\frac{\zeta}{\alpha}}}\right]\d s\nonumber\\
		&\lesssim\bar{p}_{\alpha} (t,y-x)|y-y'|^\rho t^{\frac{\gamma-\beta-2\rho}{\alpha}}  \sup_{s\in (h,T]} \tilde{g}_{h,\rho}(s),\label{delta3-holdermod_APP}
	\end{align} 
	using the fact that $s\asymp \tau_s^h$ on the considered time interval and \textcolor{black}{Lemma \ref{lm:sing_int} with $\aa=\rho/\alpha,\bb=(1+\rho)/\alpha,\cc=d/(\alpha p),\dd=\zeta/\alpha $ with 
	$$r'(\bb+\cc+\dd)<1 \iff 1+\rho+\frac dp+\zeta<\alpha-\frac \alpha r \iff \rho<\Big(\alpha-1+2\beta-\frac dp-\frac \alpha r\Big)-2\beta-\zeta=\gamma-2\beta-\zeta.$$ 
	for the last inequality. Setting $\zeta=-\beta+\eta $ for an arbitrary $\eta>0 $, we see that any $\rho<\gamma-\beta $ can fulfill the condition}. We then have, plugging \eqref{delta1-holdermod} and \eqref{delta3-holdermod_APP} into \eqref{triangu-holdermod_DIFF},
	\begin{equation}
		\left|\frac{\Gamma(0,x,t,y)}{\bar p_\alpha(t,y-x)}-\frac{\Gamma(0,x,t,y')}{\bar p_\alpha(t,y'-x)}\right| \lesssim \frac{|y-y'|^\rho}{t^{\frac{\rho}{\alpha}}} T^{\frac{\textcolor{black}{\gamma-\beta-\rho}}{\alpha}} \sup_{s\in (h,T]} \tilde{g}_{\rho}(s),
	\end{equation}
	which, together with \eqref{BD_SUP_DIFF} and \textcolor{black}{recalling that $\gamma-\beta-\rho $ can be chosen to be positive},  concludes the proof of \eqref{ineq-density-diff}.\hfill $\square $

\subsection{Time sensitivity of the heat-kernel: proof of \eqref{holder-time-gamma}.}
\label{HK_TIME_APP}

Let us assume $t'\ge t $ and $0\le t'-t \le \frac{t}2$ and write from the Duhamel representation \eqref{duhamel-Diff} of the density that:
\begin{align}
\Gamma(0,x,t,y)-\Gamma(0,x,t',y)
			=& p_\alpha(t,y-x)-p_\alpha(t',y-x)\notag\\
			&-\int_{0}^{ t}\E_{t_k,x}\left[b(s,X_s)\cdot\Big(\nabla_y  p_\alpha(t-s,y-X_s)-\nabla_y  p_\alpha(t'-s,y-X_s)\Big)\right]\d s\notag\\
			&+\int_{t}^{ t'}\E_{t_k,x}\left[b(s,X_s)\cdot\Big(\nabla_y  p_\alpha(t'-s,y-X_s)\Big)\right]\d s.\label{THE_DIFF-BIS}
\end{align}
Thus, from \eqref{holder-time-palpha} we get:
\begin{align*}
|\Gamma(0,x,t,y)-\Gamma(0,x,t',y)|
			\lesssim& \Big(\frac{t'-t}{t}\Big)^{\frac{\gamma-\varepsilon}{\alpha}}\bar p_\alpha(t,y-x)\\
			&+\int_{0}^{ t}\Big|\int \Gamma(0,x,s,z) b(s,z)\Big(\nabla_y  p_\alpha(t'-s,y-z)-\nabla_y  p_\alpha(t-s,y-z)\Big)\d z \Big| \d s\\
			&+\int_{t}^{ t'}\Big|\int \Gamma(0,x,s,z) b(s,z)\nabla_y  p_\alpha(t'-s,y-z)\d z \Big| \d s.
\end{align*}
Using now the product rule \eqref{PR} and \eqref{ineq-density-diff}, \eqref{besov-estimate-stable} (taking therein $\zeta\in (-\beta,1) $)
, we get for $\rho>-\beta$,
\begin{align*}
&|\Gamma(0,x,t,y)-\Gamma(0,x,t',y)|\lesssim   \Big(\frac{t'-t}{t}\Big)^{\frac{\gamma-\varepsilon}{\alpha}}\bar p_\alpha(t,y-x)\\
	&+\int_0^t \Vert b (s,\cdot)\Vert_{\B_{p,q}^\beta} \left\Vert \frac{\Gamma (0,x,s,\cdot)}{\bar{p}_{\alpha}(s,\cdot-x)}\right\Vert_{\B_{\infty,\infty}^\rho} \Vert \bar{p}_{\alpha}(s,\cdot-x)\Big(\nabla_y p_{\alpha} (t-s,y-\cdot)-\nabla_y p_{\alpha} (t'-s,y-\cdot)\Big)\Vert_{\B_{p',q'}^{-\beta}}\d s\nonumber\\
	&+\int_t^{t'} \Vert b (s,\cdot)\Vert_{\B_{p,q}^\beta} \left\Vert \frac{\Gamma (0,x,s,\cdot)}{\bar{p}_{\alpha}(s,\cdot-x)}\right\Vert_{\B_{\infty,\infty}^\rho} \Vert \bar{p}_{\alpha}(s,\cdot-x)\nabla_y p_{\alpha} (t'-s,y-\cdot)\Vert_{\B_{p',q'}^{-\beta}}\d s\\
	&\lesssim 	\bar{p}_{\alpha} (t,x-y)\Bigg(\Big(\frac{t'-t}{t}\Big)^{\frac{\gamma-\varepsilon}{\alpha}} +\int_{0}^{t}	\Vert b (s,\cdot)\Vert_{\B_{p,q}^\beta}s^{-\frac \rho\alpha}	\frac{(t'-t)^{\frac{\gamma-\varepsilon}\alpha}}{(t-s)^{\frac{\gamma-\varepsilon+1}{\alpha}}} (t')^{\frac{\beta}{\alpha}}\left[ \frac{1}{s^{\frac{d }{\alpha  p}}}+\frac{1}{(t-s)^{\frac{d }{\alpha  p}}} \right] \left[\frac{(t')^{\frac{\zeta}{\alpha}}}{s^{\frac{\zeta }{\alpha}}}+\frac{(t')^{\frac{\zeta}{\alpha}}}{(t-s)^{\frac{\zeta }{\alpha}}}  \right]\d s	\\
	&+\int_{t}^{t'}	\Vert b (s,\cdot)\Vert_{\B_{p,q}^\beta}	s^{-\frac \rho\alpha}\frac{1}{(t'-s)^{\frac{1}{\alpha}}} (t')^{\frac{\beta}{\alpha}}\left[ \frac{1}{s^{\frac{d }{\alpha  p}}}+\frac{1}{(t'-s)^{\frac{d }{\alpha  p}}} \right] \left[\frac{(t')^{\frac{\zeta}{\alpha}}}{s^{\frac{\zeta }{\alpha}}}+\frac{(t')^{\frac{\zeta}{\alpha}}}{(t'-s)^{\frac{\zeta }{\alpha}}}  \right]\d s\Bigg)\\
&\lesssim 	\bar{p}_{\alpha} (t,x-y)\Bigg(\Big(\frac{t'-t}{t}\Big)^{\frac{\gamma-\varepsilon}{\alpha}} \\
&+\Vert b \Vert_{L^r-\B_{p,q}^\beta}(t')^{\frac{\beta+\textcolor{black}{\zeta}}{\alpha}}(t'-t)^{\frac{\gamma-\varepsilon}\alpha}\Big( \int_{0}^{t}		\frac{s^{-\frac \rho\alpha r'}}{(t-s)^{\frac{\gamma-\varepsilon+1}{\alpha} r'}} \left[ \frac{1}{s^{\frac{d}{\alpha  p}}}+\frac{1}{(t-s)^{\frac{d}{\alpha  p}}} \right]^{\textcolor{black}{r'}} \left[  \frac{1}{s^{\frac{\zeta}{\alpha}}}+\frac{1}{(t-s)^{\frac{\zeta }{\alpha}}}  \right]^{\textcolor{black}{r'}}\d s\Big)^{\frac 1{r'}}	\\
	&+\Vert b \Vert_{L^r-\B_{p,q}^\beta} (t')^{\frac{\beta-\rho+\textcolor{black}{\zeta}}{\alpha}} \Big(\int_{t}^{t'}		\frac{1}{(t'-s)^{\frac{r'}{\alpha}}} \left[ \frac{1}{s^{\frac{d }{\alpha  p}}}+\frac{1}{(t'-s)^{\frac{d}{\alpha  p}}} \right]^{\textcolor{black}{r'}} 
	\frac{1
	}{(t'-s)^{\frac{\textcolor{black}{r'}\zeta }{\alpha}}}  
	\d s\Big)^{\frac 1{r'}}\Bigg)	
\end{align*}
recalling that, since $t'-t\le t/2 $, $t'/s\lesssim 1, \ s\in [t,t'] $
. \textcolor{black}{Hence, applying Lemma \ref{lm:sing_int} with $\aa=\rho/\alpha,\bb=(\gamma+1-\varepsilon)/\alpha,\cc=d/(\alpha p),\dd= \zeta/\alpha$ 
 for the integral on $[0,t] $ and 
direct integration for the integral on $[t,t'] $, we obtain}:
\begin{align}
&|\Gamma(0,x,t,y)-\Gamma(0,x,t',y)|\notag\\
&\lesssim \bar{p}_{\alpha} (t,x-y)\Bigg(\Big(\frac{t'-t}{t}\Big)^{\frac{\gamma-\varepsilon}{\alpha}} +\Vert b \Vert_{L^r-\B_{p,q}^\beta}t^{\frac{\beta}{\alpha}}\Big[(t'-t)^{\frac{\gamma-\varepsilon}\alpha}t^{1-\frac 1 r-\frac \rho\alpha-\frac{\gamma-\varepsilon+1}{\alpha}-\frac{d}{p\alpha}}+t^{\frac \zeta\alpha\textcolor{black}{-\frac \rho\alpha}}(t'-t)^{1-\frac 1r-\frac 1\alpha-\frac d{p\alpha}-\frac \zeta\alpha}\Big]\Bigg)\notag\\
&\lesssim \bar{p}_{\alpha} (t,x-y)\Bigg(\Big(\frac{t'-t}{t}\Big)^{\frac{\gamma-\varepsilon}{\alpha}} +\Vert b \Vert_{L^r-\B_{p,q}^\beta} t^{-\frac{\beta+\rho}\alpha+\frac \varepsilon \alpha}(t'-t)^{\frac{\gamma-\varepsilon}\alpha}\Bigg).\notag
\end{align}
We thus get:
\begin{align}
&\left\|\frac{\Gamma(0,x,t,\cdot)-\Gamma(0,x,t',\cdot)}{\bar{p}_{\alpha} (t',x-\cdot)}\right\|_{L^\infty}\lesssim \Big(\frac{t'-t}{t}\Big)^{\frac{\gamma-\varepsilon}{\alpha}},
\label{CTR_DIFF_TEMPS}
\end{align}
provided $\beta+\rho<\gamma $, which can always be achieved taking $\rho=-\beta+\eta $ for $\eta<\gamma $.
It now remains to control the  thermic part  of the Besov norm, i.e. with the notation of \eqref{HEAT_CAR}  the quantity
\begin{align*}
\mathcal T_{\infty,\infty}^\rho\Big(\frac{{\Gamma(0,x,t,\cdot)-\Gamma(0,x,t',\cdot)}}{\bar p_\alpha(t',\cdot-x)}\Big):=\sup_{v\in [0,1]}v^{\textcolor{black}{1-\frac{\rho}{\alpha}}}\left\|\partial_v p_\alpha \star\frac{(\Gamma(0,x,t,\cdot)-\Gamma(0,x,t',\cdot))}{\bar p_\alpha(t',\cdot-x)}\right\|_{L^\infty}.
\end{align*}
Write now from the expansion \eqref{THE_DIFF-BIS}:
\begin{align}
&\mathcal T_{\infty,\infty}^\rho\Big(\frac{
\Gamma(0,x,t,\cdot)-\Gamma(0,x,t',\cdot)}{\bar p_\alpha(t',\cdot-x)}\Big)\notag\\
			\le & \mathcal T_{\infty,\infty}^\rho\Big(
\frac{p_\alpha(t,\cdot-x)-p_\alpha(t',\cdot-x)}{\bar p_\alpha(t',\cdot-x)}\Big)\notag\\
			&+\mathcal T_{\infty,\infty}^\rho\Big(\frac{1}{\bar p_\alpha(t',\cdot-x)}\int_{0}^{ t}\int \Gamma(0,x,s,z) b(s,z)\Big(\nabla  p_\alpha(t'-s,\cdot-z)-\nabla  p_\alpha(t-s,\cdot-z)\Big)\d z  \d s\Big)\notag\\
			&+\mathcal T_{\infty,\infty}^\rho\Big(\frac{1}{\bar p_\alpha(t',\cdot-x)}\int_{t}^{ t'}\int \Gamma(0,x,s,z) b(s,z)\nabla  p_\alpha(t'-s,\cdot-z)\d z  \d s\Big)=:\sum_{i=1}^3 \mathcal T_{\infty,\infty,i}^\rho(x,t,t').
			\label{DECOUP_HOLD_T}
\end{align}
Write:
\begin{align*}
T_{\infty,\infty,1}^\rho(x,t,t')=\sup_{v\in [0,1]}v^{1-\frac{\rho}\alpha}\left\|\partial_v p_\alpha(v,\cdot)\star \frac{(p_\alpha(t,\cdot-x)-p_\alpha(t',\cdot-x))}{\bar p_\alpha(t',\cdot-x)}\right\|_{L^\infty}.
\end{align*}
Recall that $t,t'$ are assumed to be small and that, in that setting, $t$ appears as a natural cutting level in the study of the thermic part of the norm. Namely, for $v\ge t$ we readily get from \textcolor{black}{\eqref{holder-time-palpha}}
\begin{align*}
\left\|\partial_v p_\alpha(v,\cdot)\star \frac{(p_\alpha(t,\cdot-x)-p_\alpha(t',\cdot-x))}{\bar p_\alpha(t',\cdot-x)} \right\|_{L^\infty}\lesssim \frac{|t-t'|^{\frac{\gamma-\varepsilon}\alpha}}{t^{\frac{\gamma-\varepsilon}\alpha}}\frac{1}{v}.
\end{align*} 
In particular:
\begin{align}
\sup_{v\in [t,1]}v^{1-\frac{\rho}\alpha}\left\|\partial_v p_\alpha(v,\cdot)\star \frac{(p_\alpha(t,\cdot-x)-p_\alpha(t',\cdot-x))}{\bar p_\alpha(t',\cdot-x)} \right\|_{L^\infty}\lesssim \frac{|t-t'|^{\frac{\gamma-\varepsilon}\alpha}}{t^{\frac{\gamma-\varepsilon}\alpha}}\frac{1}{t^{\frac \rho\alpha}}.\label{COUPE_1_MAIN_HOLDER}
\end{align} 
Write now for $v\in [0,t] $, for all $y\in \R^d$:
\begin{align*}
&\left|\partial_v p_\alpha(v,\cdot)\star \frac{(p_\alpha(t,\cdot-x)-p_\alpha(t',\cdot-x))}{\bar p_\alpha(t',\cdot-x)}(y)\right|\\
=&\left|\int \partial_v p_\alpha(v,y-z)\Big(\frac{(p_\alpha(t,z-x)-p_\alpha(t',z-x))}{\bar p_\alpha(t',z-x)}-\frac{(p_\alpha(t,y-x)-p_\alpha(t',y-x))}{\bar p_\alpha(t',y-x)} \Big)\d z\right|\\
\lesssim &\frac 1 v\Big(\int\bar p_\alpha(v,y-z){\mathbb 1}_{|y-z|>t^{\frac 1\alpha}}2\left\|\frac{(p_\alpha(t,\cdot-x)-p_\alpha(t',\cdot-x))}{\bar p_\alpha(t',\cdot-x)}\right\|_{L^\infty}\d z\\
&+\int\bar p_\alpha(v,y-z){\mathbb 1}_{|y-z|\le t^{\frac 1\alpha}}\frac{|t-t'|\frac{\gamma-\varepsilon}{\alpha}}{t^{\frac{\gamma-\varepsilon}{\alpha}}}\frac{|y-z|^{\rho}}{t^{\frac{\rho}{\alpha}}}\d z\Big)
\end{align*} 
using a spatial Taylor expansion and \eqref{holder-time-palpha} (recalling as well that the diagonal regime holds for the \textit{non thermic} densities) for the last inequality. Hence, from \eqref{CTR_DIFF_TEMPS},
\begin{align*}
&\left|\partial_v p_\alpha(v,\cdot)\star \frac{(p_\alpha(t,\cdot-x)-p_\alpha(t',\cdot-x))}{\bar p_\alpha(t',\cdot-x)}(y)\right|\\
\lesssim & \frac{1}{v}\frac{1}{t^{\frac \rho\alpha}}\frac{|t-t'|\frac{\gamma-\varepsilon}{\alpha}}{t^{\frac{\gamma-\varepsilon}{\alpha}}}\int\bar p_\alpha(v,y-z) |y-z|^\rho dz\lesssim \frac{1}{v^{1-\frac{\rho}{\alpha}}}\frac{1}{t^{\frac \rho\alpha}}\frac{|t-t'|\frac{\gamma-\varepsilon}{\alpha}}{t^{\frac{\gamma-\varepsilon}{\alpha}}}.
\end{align*}
Thus,
$$\sup_{v\in [0,t]}v^{1-\frac{\rho}\alpha}\|\partial_v p_\alpha(v,\cdot)\star \frac{(p_\alpha(t,\cdot-x)-p_\alpha(t',\cdot-x))}{\bar p_\alpha(t',\cdot-x)} \|_{L^\infty}\lesssim \frac{|t-t'|^{\frac{\gamma-\varepsilon}\alpha}}{t^{\frac{\gamma-\varepsilon}\alpha}}\frac{1}{t^{\frac \rho\alpha}}, $$
which together with \eqref{COUPE_1_MAIN_HOLDER} yields 
\begin{align}
T_{\infty,\infty,1}^\rho(x,t,t')\lesssim \frac{|t-t'|^{\frac{\gamma-\varepsilon}\alpha}}{t^{\frac{\gamma-\varepsilon}\alpha}}\frac{1}{t^{\frac \rho\alpha}}.\label{CT_HOLD_T1}
\end{align}
which is precisely the expected expected bound. \\

Let us now turn to $\mathcal T_{\infty,\infty,3}^\rho(x,t,t')$:
\begin{align*}
\mathcal T_{\infty,\infty,3}^\rho(x,t,t')=\mathcal T_{\infty,\infty}^\rho\Big(\frac{1}{\bar p_\alpha(t',\cdot-x)}\int_{t}^{ t'}\int \Gamma(0,x,s,z) b(s,z)\nabla  p_\alpha(t'-s,\cdot-z)\d z  \d s\Big). 
\end{align*}
We proceed with the same previous dichotomy for the time variable:
\begin{trivlist}
\item[-]For $v\in [t,1] $ write:
\begin{align*}
&\left\|\partial_v p_\alpha(v,\cdot)\star\Bigg(\frac{1}{\bar p_\alpha(t',\cdot-x)}   \int_{t}^{ t'}\int \Gamma(0,x,s,z) b(s,z)\nabla  p_\alpha(t'-s,\cdot-z)\d z  \d s\Bigg)\right\|_{L^\infty}\\
\lesssim &\frac{C}{v}\int_{t}^{t'} \left\|\frac{1}{\bar p_\alpha(t',\cdot-x)}\int \Gamma(0,x,s,z) b(s,z)\nabla  p_\alpha(t'-s,\cdot-z)\d z \right\|_{L^\infty} \d s\\
\lesssim &\frac{1}{v}\int_{t}^{t'} \|b(s,\cdot)\|_{\B_{p,q}^\beta} \sup_{y\in \R^d}\frac{1}{\bar p_\alpha(t',\cdot-y)}\left\|\frac{\Gamma(0,x,s,\textcolor{black}{\cdot})}{\bar p_\alpha(s,x,\cdot)} \right\|_{\B_{\infty,\infty}^\rho}\left\|\bar p_\alpha(s,x,\cdot)\nabla  p_\alpha(t'-s,y-\cdot)\right\|_{\B_{p',q'}^{-\beta}}\d s\\
\lesssim &\frac{1}{v}\int_{t}^{t'} \|b(s,\cdot)\|_{\B_{p,q}^\beta} s^{-\frac \rho\alpha} 				\frac{1}{(t'-s)^{\frac{1}{\alpha}}} (t')^{\frac{\beta}{\alpha}}\left[ \frac{1}{s^{\frac{d }{\alpha  p}}}+\frac{1}{(t'-s)^{\frac{d }{\alpha  p}}} \right] \left[\frac{(t')^{\frac{\zeta}{\alpha}}}{s^{\frac{\zeta }{\alpha}}}+\frac{{t'}^{\frac{\zeta}{\alpha}}}{(t'-s)^{\frac{\zeta }{\alpha}}}  \right]\d s,
\end{align*}
where we used \eqref{ineq-density-diff} and \eqref{besov-estimate-stable} for the last inequality, where $\rho,\zeta>-\beta $. We get:
\begin{align*}
&\left\|\partial_v p_\alpha(v,\cdot)\star\Bigg(\frac{1}{\bar p_\alpha(t',\cdot-x)}   \int_{t}^{ t'}\int \Gamma(0,x,s,z) b(s,z)\nabla  p_\alpha(t'-s,\cdot-z)\d z  \d s\Bigg)\right\|_{L^\infty}\\
\lesssim& \frac{1}v\|b\|_{L^r-\B_{p,q}^\beta}t^{-\frac{\rho}{\alpha}+\frac \beta\alpha+\frac \zeta\alpha}\left(\int_t^{t'} \frac{1}{(t'-s)^{\frac{r'}{\alpha}}} \left[ \frac{1}{(t'-s)^{\frac{d r'}{\alpha  p}}} \right] \left[  \frac{1}{(t'-s)^{\frac{\zeta r'}{\alpha}}}\right]\d s\right)^{\frac 1{r'}}\\
\lesssim& \frac{1}v\|b\|_{L^r-\B_{p,q}^\beta}t^{-\frac{\rho}{\alpha}} (t'-t)^{1-\frac 1{r}-(\frac 1\alpha+\frac{d}{p\alpha}+\frac{\zeta}{\alpha}) }=\frac{1}v\|b\|_{L^r-\B_{p,q}^\beta}t^{-\frac{\rho}{\alpha}} (t'-t)^{\frac{\gamma-\varepsilon}{\alpha} +\frac{-2\beta-\zeta}{\alpha}+\frac{\varepsilon}\alpha }.
\end{align*}
Eventually, choosing ${-2\beta-\zeta}+\varepsilon=\rho $.
\begin{align}
&\sup_{v\in [t,1]}v^{1-\frac \rho\alpha}\left\|\partial_v p_\alpha(v,\cdot)\star\Bigg(\frac{1}{\bar p_\alpha(t',\cdot-x)}   \int_{t}^{ t'}\int \Gamma(0,x,s,z) b(s,z)\nabla  p_\alpha(t'-s,\cdot-z)\d z  \d s\Bigg)\right\|_{L^\infty}\notag\\
\lesssim& (t'-t)^{\frac{\gamma-\varepsilon}\alpha}t^{-\frac \rho\alpha}.\label{T3_RHO_CUT_OFF_ALTO}
\end{align}
\item[-] For $v\in [0,t] $, write for all $y\in \R^d $:
\begin{align*}
&|\partial_v p_\alpha(v,\cdot)\star\Bigg(\frac{1}{\bar p_\alpha(t',\cdot-x)}   \int_{t}^{ t'}\int \Gamma(0,x,s,z) b(s,z)\nabla  p_\alpha(t'-s,\cdot-z)\d z  \d s\Bigg)(y)|\\
=&\left|\int_t^{t'}\int\int\partial_v p_\alpha(v,y-w)\Big(\frac{\Gamma(0,x,s,z) b(s,z)\nabla  p_\alpha(t'-s,w-z)}{\bar p_\alpha(t',w-x)}-\frac{\Gamma(0,x,s,z) b(s,z)\nabla  p_\alpha(t'-s,y-z)}{\bar p_\alpha(t',y-x)} \Big) \d z\d w\d s\right|\\
\lesssim& \int_t^{t'}\int |\partial_v p_\alpha(v,y-w)| \left| \int\frac{\Gamma(0,x,s,z) b(s,z)\nabla  p_\alpha(t'-s,w-z)}{\bar p_\alpha(t',w-x)}\right.\\
&\left.-\frac{\Gamma(0,x,s,z) b(s,z)\nabla  p_\alpha(t'-s,y-z)}{\bar p_\alpha(t',y-x)}  \d z\right| \d w \d s\\
\lesssim& \int_t^{t'}\|b(s,\cdot)\|_{B_{p,q}^\beta}\int |\partial_v p_\alpha(v,y-w)|\left \|\frac{\Gamma(0,x,s,\cdot)}{\bar p_\alpha(s,\cdot-x)}\right\|_{B_{\infty,\infty}^\rho}\\
&\times \left\|\frac{\bar p_\alpha(s,\cdot-x) \nabla  p_\alpha(t'-s,w-\cdot)}{\bar p_\alpha(t',w-x)}-\frac{\bar p_\alpha(s,\cdot-x)\nabla  p_\alpha(t'-s,y-\cdot)}{\bar p_\alpha(t',y-x)}  \right\|_{B_{p',q'}^{-\beta}} \d w \d s\\
\underset{\eqref{big-lemma-2}}{\lesssim}& \int_t^{t'}\|b(s,\cdot)\|_{B_{p,q}^\beta}\Bigg(\int\frac{\bar p_\alpha(v,y-w)}{v}|w-y|^\zeta \d w\Bigg)  \frac{s^{-\frac \rho\alpha} (t')^{\frac{\beta}{\alpha}}}{(t'-s)^{\frac{\zeta+1}{\alpha}}}  \left[ \frac{1}{(t'-s)^{\frac{d }{\alpha  p}}} +\frac{1}{s^{\frac{d }{\alpha  p}}}\right] \left[
\frac{(t')^{\frac{\zeta}{\alpha}}}{(t'-s)^{\frac{\zeta }{\alpha}}} + \frac{(t')^{\frac{\zeta}{\alpha}}}{s^{\frac{\zeta }{\alpha}}}
 \right] \\
\lesssim & v^{-1+\frac{\zeta}{\alpha}}\|b\|_{L^r-\B_{p,q}^\beta}(t')^{-\frac \rho\alpha+\frac\beta \alpha+\frac \zeta\alpha}\Bigg( \int_t^{t'}\frac{\d s}{(t'-s)^{r'(\frac{1+\zeta}{\alpha}+\frac{d}{\alpha p}+\frac{\zeta}{\alpha})}}\Bigg)^{\frac 1{r'}}\\
\lesssim & v^{-1+\frac{\zeta}{\alpha}}\|b\|_{L^r-\B_{p,q}^\beta}(t')^{-\frac \rho\alpha+\frac\beta \alpha+\frac \zeta\alpha}(t'-t)^{1-\frac 1r-(\frac 1\alpha+\frac d{p\alpha}+2\frac{\zeta}{\alpha})}\\
\lesssim & v^{-1+\frac{\zeta}{\alpha}}\|b\|_{L^r-\B_{p,q}^\beta}(t')^{-\frac \rho\alpha}(t'-t)^{\frac{\gamma-\varepsilon}{\alpha}+\frac{-2\zeta-2\beta+\varepsilon}{\alpha}}
\lesssim 
v^{-1+\frac{\zeta}{\alpha}}\|b\|_{L^r-\B_{p,q}^\beta}(t')^{-\frac \rho\alpha}(t'-t)^{\frac{\gamma-\varepsilon}{\alpha}}
\end{align*}
provided $-2\zeta-2\beta+\varepsilon\textcolor{black}{\ge} 0\iff \zeta	\textcolor{black}{\le} -\beta+\frac \varepsilon 2$. Together with \eqref{T3_RHO_CUT_OFF_ALTO} we eventually get:
\begin{align}
\mathcal T_{\infty,\infty,3}^\rho(x,t,t')\lesssim (t')^{-\frac \rho\alpha}(t'-t)^{\frac{\gamma-\varepsilon}{\alpha}}.\label{CT_HOLD_T3}
\end{align}
\end{trivlist}

Let us now turn to $\mathcal T_{\infty,\infty,2}^\rho(x,t,t')$
\begin{align*}
\mathcal T_{\infty,\infty,2}^\rho(x,t,t')=\mathcal T_{\infty,\infty}^\rho\Big(\frac{1}{\bar p_\alpha(t',\cdot-x)}\int_{0}^{ t}\int \Gamma(0,x,s,z) b(s,z)\Big(\nabla  p_\alpha(t'-s,\cdot-z)-\nabla  p_\alpha(t-s,\cdot-z)\Big)\d z  \d s\Big). 
\end{align*}
We proceed with the same previous dichotomy for the time variable:
\begin{trivlist}	
\item[-] For $v\in [t,1] $ write:
\begin{align*}
&\left\|\partial_v p_\alpha(v,\cdot)\star\Bigg(\frac{1}{\bar p_\alpha(t',\cdot-x)}   \int_{0}^{ t}\int \Gamma(0,x,s,z) b(s,z)\Big(\nabla  p_\alpha(t'-s,\cdot-z)-\nabla  p_\alpha(t-s,\cdot-z)\Big)\d z  \d s\Bigg)\right\|_{L^\infty}\\
\lesssim &\frac{C}{v}\int_{0}^{t} \left\|\frac{1}{\bar p_\alpha(t',\cdot-x)}\int \Gamma(0,x,s,z) b(s,z)\Big(\nabla  p_\alpha(t'-s,\cdot-z)-\nabla  p_\alpha(t-s,\cdot-z)\Big)\d z \right\|_{L^\infty} \d s\\
\lesssim &\frac{1}{v}\int_{0}^{t} \|b(s,\cdot)\|_{\B_{p,q}^\beta} \sup_{y\in \R^d}\frac{1}{\bar p_\alpha(t',y-x)}\left\|\frac{\Gamma(0,x,s,\textcolor{black}{\cdot})}{\bar p_\alpha(s,\cdot-x)} \right\|_{\B_{\infty,\infty}^\rho}\left\|\bar p_\alpha(s,\cdot-x)\Big(\nabla  p_\alpha(t'-s,y-\cdot)-\nabla  p_\alpha(t-s,y-\cdot)\Big)\right\|_{\B_{p',q'}^{-\beta}}\d s\\
\lesssim &\frac{1}{v}\int_{0}^{t} \|b(s,\cdot)\|_{\B_{p,q}^\beta} s^{-\frac \rho\alpha} 	
(t-t')^{\frac{\gamma-\varepsilon}\alpha}\frac{t^{\frac{\beta}{\alpha}} }{(t-s)^{\frac{1}{\alpha}+\frac{\gamma-\varepsilon}\alpha}} \left[\frac{1}{s^{\frac{d }{\alpha  p}}}+ \frac{1}{(t-s)^{\frac{d }{\alpha  p}}} \right] \left[ \frac{t^{\frac{\zeta}{\alpha}}}{s^{\frac{\zeta }{\alpha}}}+\frac{t^{\frac{\zeta}{\alpha}}}{(t-s)^{\frac{\zeta }{\alpha}}}  \right]
\d s,
\end{align*}
using \eqref{besov-estimate-gammah-stable-sensi-holder-time} (with $p_\alpha$ in place of $\Gamma^h$) for the last inequality. Thus,
\begin{align*}
&\left\|\partial_v p_\alpha(v,\cdot)\star\Bigg(\frac{1}{\bar p_\alpha(t',\cdot-x)}   \int_{0}^{ t}\int \Gamma(0,x,s,z) b(s,z)\Big(\nabla  p_\alpha(t'-s,\cdot-z)-\nabla  p_\alpha(t-s,\cdot-z)\Big)\d z  \d s\Bigg)\right\|_{L^\infty}\\
\lesssim&\frac{1}{v}(t-t')^{\frac{\gamma-\varepsilon}\alpha}\|b\|_{L^r-\B_{p,q}^\beta}t^{\frac{\beta+\zeta}\alpha}
 \Bigg(\int_0^t s^{-r'\frac \rho\alpha} 	
\frac{1}{(t-s)^{r'(\frac{1}{\alpha}+\frac{\gamma-\varepsilon}\alpha)}} \left[\frac{1}{s^{\frac{d }{\alpha  p}}}+ \frac{1}{(t-s)^{\frac{d }{\alpha  p}}} \right]^{r'} \left[\frac{1}{s^{\frac{\zeta }{\alpha}}}+\frac{1}{(t-s)^{\frac{\zeta }{\alpha}}}  \right]^{r'}\Bigg)^{\frac 1{r'}}\\
\lesssim&\frac{1}{v}(t-t')^{\frac{\gamma-\varepsilon}\alpha}\|b\|_{L^r-\B_{p,q}^\beta}t^{\frac{\beta+\zeta}\alpha}t^{1-\frac 1{r}-\big(\textcolor{black}{\frac \rho\alpha}+\frac 1\alpha+\frac{\gamma-\varepsilon}\alpha+\frac{d}{\alpha p}+\frac \zeta\alpha\big)}\le \frac{1}{v}(t-t')^{\frac{\gamma-\varepsilon}\alpha}\|b\|_{L^r-\B_{p,q}^\beta}t^{-\frac{\textcolor{black}{\rho+\beta-\varepsilon}}{\alpha}},
\end{align*}
which yields, \textcolor{black}{taking $\rho+\beta-\varepsilon\le 0 $},
\begin{align}
&\sup_{v\in [t,1]}v^{1-\frac \rho\alpha}\left\|\partial_v p_\alpha(v,\cdot)\star\Bigg(\frac{1}{\bar p_\alpha(t',\cdot-x)}   \int_{0}^{ t}\int \Gamma(0,x,s,z) b(s,z)\Big(\nabla  p_\alpha(t'-s,\cdot-z)-\nabla  p_\alpha(t-s,\cdot-z)\Big)\d z  \d s\Bigg)\right\|_{L^\infty}\notag\\
\le& (t-t')^{\frac{\gamma-\varepsilon}\alpha}\|b\|_{L^r-\B_{p,q}^\beta}t^{-\frac{\rho}{\alpha}}.\label{T2_RHO_CUT_OFF_ALTO}
\end{align}
\item[-] For $v\in [0,t] $:
write for all $y\in \R^d $:
\begin{align*}
&|\partial_v p_\alpha(v,\cdot)\star\Bigg(\frac{1}{\bar p_\alpha(t',\cdot-x)}   \int_{0}^{ t}\int \Gamma(0,x,s,z) b(s,z)\Big(\nabla  p_\alpha(t'-s,\cdot-z)-\nabla  p_\alpha(t-s,\cdot-z)\Big)\d z  \d s\Bigg)(y)|\\
=&\left|\int_0^{t-2|t'-t|}\int\int\partial_v p_\alpha(v,y-w)\Big(\frac{\Gamma(0,x,s,z) b(s,z)\Big(\nabla  p_\alpha(t'-s,w-z)-\nabla  p_\alpha(t-s,w-z)\Big)}{\bar p_\alpha(t',w-x)}\right.\\
&-\left.\frac{\Gamma(0,x,s,z) b(s,z)\Big(\nabla  p_\alpha(t'-s,y-z)-\nabla  p_\alpha(t-s,y-z)\Big)}{\bar p_\alpha(t',y-x)} \Big) \d z\d w\d s\right|\\
&+\left|\int_{t-2|t'-t|}^{\textcolor{black}{t}}\int\int\partial_v p_\alpha(v,y-w)\Big(\frac{\Gamma(0,x,s,z) b(s,z)\nabla  p_\alpha(t'-s,w-z)}{\bar p_\alpha(t',w-x)}-\frac{\Gamma(0,x,s,z) b(s,z)\nabla  p_\alpha(t'-s,y-z)}{\bar p_\alpha(t',y-x)} \Big) \d z\d w\d s\right|\\
&+\left|\int_{t-2|t'-t|}^{\textcolor{black}{t}}\int\int\partial_v p_\alpha(v,y-w)\Big(\frac{\Gamma(0,x,s,z) b(s,z)\nabla  p_\alpha(t-s,w-z)}{\bar p_\alpha(t',w-x)}-\frac{\Gamma(0,x,s,z) b(s,z)\nabla  p_\alpha(t-s,y-z)}{\bar p_\alpha(t',y-x)} \Big) \d z\d w\d s\right|\\
=:&(T_1+T_2+T_3)(v,t,t',x,y).
\end{align*}
Note that the terms $(T_2+T_3)(v,t,t',x,y) $ can be handled just as we did before for the lower cut in the thermic variable for $\mathcal T_{\infty,\infty,\textcolor{black}{3}}^\rho(x,t,t')$. On the other hand, \textcolor{black}{extending \eqref{big-lemma-2} to take as well into account the time derivative we get}
:
\begin{align*}
&T_1(v,t,t',x,y)\\
\lesssim&\left|\int_0^1 \d \lambda\int_0^{t-2|t'-t|}\int\int\partial_v p_\alpha(v,y-w)\Big(\frac{\Gamma(0,x,s,z) b(s,z)\partial_u \nabla  p_\alpha(u,w-z)|_{u=t+\lambda (t'-t)-s}}{\bar p_\alpha(t',w-x)}\right.\\
&-\left.\frac{\Gamma(0,x,s,z) b(s,z)\partial_u\nabla  p_\alpha(u,y-z)|_{u=t+\lambda (t'-t)-s}}{\bar p_\alpha(t',y-x)} \Big) \d z\d w\d s\right| (t'-t)\\
&\lesssim \int_0^1 \d \lambda \int_0^{t-2|t'-t|}\|b(s,\cdot)\|_{\B_{p,q}^\beta}\int |\partial_v p_\alpha(v,y-w)|\left \|\frac{\Gamma(0,x,s,\cdot)}{\bar p_\alpha(s,\cdot-x)}\right\|_{\B_{\infty,\infty}^\rho}\\
&\times \left\|\frac{\bar p_\alpha(s,\cdot-x) \partial_u\nabla  p_\alpha(u,w-\cdot)|_{u=t+\lambda(t'-t)-s}}{\bar p_\alpha(t',w-x)}-\frac{\bar p_\alpha(s,\cdot-x)\partial_u\nabla  p_\alpha(u,y-\cdot)|_{u=t+\lambda(t'-t)-s}}{\bar p_\alpha(t',y-x)}  \right\|_{\B_{p',q'}^{-\beta}}\!\!\! \d w \d s (t'-t)\\
\underset{\eqref{big-lemma-2}}{\lesssim}
& \int_0^{t-2|t'-t|}\|b(s,\cdot)\|_{\B_{p,q}^\beta}s^{-\frac \rho\alpha}\Bigg(\int\frac{\bar p_\alpha(v,y-w)}{v}|w-y|^\zeta \d w\Bigg)  \frac{ (t')^{\frac{\beta}{\alpha}}}{(t'-s)^{\frac{\textcolor{black}{1+\zeta}}{\alpha}+1}}  \left[ \frac{1}{(t'-s)^{\frac{d }{\alpha  p}}} +\frac{1}{s^{\frac{d }{\alpha  p}}}\right]\\
&\times  \left[(t')^{\frac{\zeta}{\alpha}}\left( \frac{1}{(t'-s)^{\frac{\zeta }{\alpha}}} + \frac{1}{s^{\frac{\zeta }{\alpha}}}\right) \right] (t'-t)\d s\\
\lesssim & v^{-1+\frac{\zeta}{\alpha}}\|b\|_{L^r-\B_{p,q}^\beta}(t')^{\frac\beta \alpha+\frac \zeta\alpha}(t'-t)^{\frac{\gamma-\varepsilon}\alpha}\Bigg( \int_0^{t-2|t'-t|}\frac{\d s}{s^{r'\frac\rho \alpha}(t'-s)^{r'(\frac{1+\zeta}{\alpha}+\frac{d}{\alpha p}+\frac{\zeta}{\alpha}+\frac{\gamma-\varepsilon}\alpha)}}\Bigg)^{\frac 1{r'}}\\
\lesssim &v^{-1+\frac{\zeta}{\alpha}}\|b\|_{L^r-\B_{p,q}^\beta}(t')^{\frac\beta \alpha+\frac \zeta\alpha}(t'-t)^{\frac{\gamma-\varepsilon}\alpha}t^{1-\frac 1r-(\frac 1\alpha+\frac d{p\alpha}+2\frac{\zeta}{\alpha}+\frac{\gamma-\varepsilon}{\alpha})-\frac \rho\alpha}\\
\lesssim & v^{-1+\frac{\zeta}{\alpha}}\|b\|_{L^r-\B_{p,q}^\beta}(t')^{\frac\beta \alpha+\frac \zeta\alpha}(t'-t)^{\frac{\gamma-\varepsilon}{\alpha}} t^{\frac{-2\zeta-2\beta+\varepsilon}{\alpha}-\frac \rho\alpha}
\lesssim 
v^{-1+\frac{\zeta}{\alpha}}\|b\|_{L^r-\B_{p,q}^\beta}(t')^{-\frac \rho\alpha}(t'-t)^{\frac{\gamma-\varepsilon}{\alpha}},
\end{align*}
provided $-\frac\beta \alpha-\frac \zeta\alpha+\frac{\varepsilon}{2\alpha}>0$ for the above integral to converge
. These computations, together with \eqref{T2_RHO_CUT_OFF_ALTO} eventually yields:
$$\mathcal T_{\infty,\infty,\textcolor{black}{2}}^\rho(x,t,t')\lesssim (t')^{-\frac \rho\alpha}(t'-t)^{\frac{\gamma-\varepsilon}{\alpha}},
 $$
 which together with \eqref{CT_HOLD_T3}, \eqref{CT_HOLD_T1} and \eqref{DECOUP_HOLD_T} gives the claim.
\end{trivlist}

\bibliographystyle{alpha}
\bibliography{ar}

\begin{thebibliography}{CdRJM25}

\bibitem[ABM20]{ABM20}
Siva Athreya, Oleg Butkovsky, and Leonid Mytnik.
\newblock Strong existence and uniqueness for stable stochastic differential
  equations with distributional drift.
\newblock {\em Ann. Probab.}, 48(1):178--210, 2020.

\bibitem[BJ22]{BJ20}
Oumaima Bencheikh and Benjamin Jourdain.
\newblock Convergence in total variation of the {E}uler--{M}aruyama scheme
  applied to diffusion processes with measurable drift coefficient and additive
  noise.
\newblock {\em SIAM Journal on Numerical Analysis}, 60(4):1701--1740, 2022.

\bibitem[BT96]{BT96}
Vlad Bally and Denis Talay.
\newblock The law of the euler scheme for stochastic differential equations.
\newblock {\em Probability Theory and Related Fields}, 104(1):43--60, 1996.

\bibitem[CdRJM25]{CdRJM22}
Paul-\'{E}ric Chaudru~de Raynal, Jean-Francois Jabir, and St\'{e}phane Menozzi.
\newblock Multidimensional stable driven {McK}ean-{V}lasov {SDE}s with
  distributional interaction kernel -- a regularization by noise perspective.
\newblock {\em {S}tochastics and Partial {D}ifferential {E}quations},
  13--1:367--420, 2025.

\bibitem[CdRM22]{CdRM22}
Paul-\'{E}ric Chaudru~de Raynal and St{\'e}phane Menozzi.
\newblock {On multidimensional stable-driven stochastic differential equations
  with {B}esov drift}.
\newblock {\em Electronic Journal of Probability}, 27(none):1 -- 52, 2022.

\bibitem[CIP23]{CIP23}
{Luis Mario} {Chaparro Jáquez}, Elena Issoglio, and Jan Palczewski.
\newblock Convergence rate of numerical scheme for {SDE}s with a distributional
  drift in {B}esov space.
\newblock {\em ArXiv 2309.11396}, 2023.

\bibitem[CZZ21]{CZZ21}
Zhen-Qing Chen, Xicheng Zhang, and Guohuan Zhao.
\newblock Supercritical {SDE}s driven by multiplicative stable-like {L}\'{e}vy
  processes.
\newblock {\em Trans. Amer. Math. Soc.}, 374(11):7621--7655, 2021.

\bibitem[DD16]{DD16}
Fran{\c{c}}ois Delarue and Roland Diel.
\newblock Rough paths and 1d {SDE} with a time dependent distributional drift:
  application to polymers.
\newblock {\em Probability Theory and Related Fields}, 165(1):1--63, June 2016.

\bibitem[DGI22]{dAGI}
Tiziano {De Angelis}, Maximilien Germain, and Elena Issoglio.
\newblock A numerical scheme for stochastic differential equations with
  distributional drift.
\newblock {\em Stochastic Processes and their Applications}, 154:55--90, 2022.

\bibitem[Fit24]{Fit23}
Mathis Fitoussi.
\newblock Heat kernel estimates for stable-driven {SDE}s with distributional
  drift.
\newblock {\em Potential Analysis}, 61--3:431--461, 2024.

\bibitem[FJM25]{FJM24}
Mathis Fitoussi, Benjamin Jourdain, and St{\'e}phane Menozzi.
\newblock {Weak well-posedness and weak discretization error for stable-driven
  SDEs with Lebesgue drift}.
\newblock {\em IMA Journal of Numerical Analysis}, page draf079, 2025.
\newblock https://doi.org/10.1093/imanum/draf079.

\bibitem[FM25]{FM24}
Mathis Fitoussi and Stephane Menozzi.
\newblock Weak error on the densities for the {E}uler scheme of stable additive
  {SDE}s with {H}{\"o}lder drift.
\newblock {\em Stochastic Processes and Applications}, Vol 190:\ paper
  $\#104736 $, 21pp, 2025.

\bibitem[GHR25]{GHR25}
Ludovic Goudenège, El~Mehdi Haress, and Alexandre Richard.
\newblock Numerical approximation of {SDE}s with fractional noise and
  distributional drift.
\newblock {\em Stochastic Processes and their Applications}, 181:104533, 2025.

\bibitem[HLL24]{hao:le:ling:24}
Zimo Hao, Khoa L{\^{e}}, and Chengcheng Ling.
\newblock Quantitative approximation of stochastic kinetic equations: from
  discrete to continuum.
\newblock {\em arXiv:2409.05706}, 2024.

\bibitem[HLL26]{hao:le:ling:26}
Zimo Hao, Khoa Lê, and Chengcheng Ling.
\newblock Weak approximation of kinetic {SDE}s: closing the criticality gap.
\newblock {\em arXiv 2602.18411}, 2026.

\bibitem[Hol24]{Hol22}
Teodor Holland.
\newblock A note on the weak rate of convergence for the {E}uler-{M}aruyama
  scheme with {H}\"older drift.
\newblock {\em Stoch. Proc and Appl.}, 174, paper \#104379, 2024.

\bibitem[HW26]{hao:wu:26}
Zimo Hao and Mingyan Wu.
\newblock Euler–maruyama scheme for $\alpha $-stable {SDE} with
  distributional drift.
\newblock {\em arXiv:2604.07757}, 2026.

\bibitem[JM24]{JM211}
Benjamin Jourdain and St{\'e}phane Menozzi.
\newblock {Convergence Rate of the Euler-Maruyama Scheme Applied to Diffusion
  Processes with $L^q$ -- $L^\rho$ Drift Coefficient and Additive Noise}.
\newblock {\em Annals of Applied Probability}, 34--1b:1663--1697, January 2024.

\bibitem[JM26]{jour:meno:26}
Benjamin Jourdain and Stéphane Menozzi.
\newblock Weak error for {SDE}s with additive stable noise and singular drift:
  choose the test function in the same space as the drift!
\newblock {\em arXiv 2604.20323}, 2026.

\bibitem[KM02]{KM02}
Valentin Konakov and Enno Mammen.
\newblock {Edgeworth type expansions for Euler schemes for stochastic
  differential equations}.
\newblock {\em Monte Carlo Methods Appl.}, 8--3:271--285, 2002.

\bibitem[KM10]{KM10}
Valentin Konakov and Stéphane Menozzi.
\newblock Weak error for stable driven stochastic differential equations:
  Expansion of the densities.
\newblock {\em Journal of Theoretical Probability}, 24-2:554--578, 2010.

\bibitem[KM17]{KM17}
Valentin Konakov. and Stéphane Menozzi.
\newblock Weak error for the {E}uler scheme approximation of diffusions with
  non-smooth coefficients.
\newblock {\em Electr. Journal of Proba.}, 22:paper \# 46, 47 p., 2017.

\bibitem[Kol00]{Kol00}
Vassili Kolokoltsov.
\newblock Symmetric stable laws and stable-like jump-diffusions.
\newblock {\em Proc. London Math. Soc. (3)}, 80(3):725--768, 2000.

\bibitem[KS22]{KS21}
Franziska Kühn and René~L. Schilling.
\newblock Convolution inequalities for {B}esov and {T}riebel--{L}izorkin
  spaces, and applications to convolution semigroups.
\newblock {\em Studia Mathematica}, 262:93--119, 2022.

\bibitem[L{\^e}20]{Le20}
Khoa L{\^e}.
\newblock {A stochastic sewing lemma and applications}.
\newblock {\em Electronic Journal of Probability}, 25(none):1 -- 55, 2020.

\bibitem[LL25]{LL22}
Khoa L\^e and Chengcheng Ling.
\newblock Taming singular stochastic differential equations: A numerical
  method.
\newblock {\em Annals of Probability}, 53--5:1764--1824, 2025.

\bibitem[LR02]{lema:02}
Pierre-Gilles Lemari\'{e}-Rieusset.
\newblock {\em Recent developments in the Navier-Stokes problem}.
\newblock CRC Press, 2002.

\bibitem[MP91]{MP91}
Remigijus Mikulevicius and Eckhard Platen.
\newblock Rate of convergence of the {E}uler approximation for diffusion
  processes.
\newblock {\em Mathematische Nachrichten}, 151:233--239, 1991.

\bibitem[MP26]{meno:pagl:26}
Stéphane Menozzi and Stefano Pagliarani.
\newblock On heat kernel estimtes for brownian sdes with distributional drift.
\newblock {\em arXiv, 2605.18505}, 2026.

\bibitem[PvZ23]{PvZ22}
Nicolas Perkowski and Willem van Zuijlen.
\newblock Quantitative heat-kernel estimates for diffusions with distributional
  drift.
\newblock {\em Potential Analysis}, 59--2:731--752, 2023.
\newblock https://doi.org/10.1007/s11118-021-09984-3.

\bibitem[Sat99]{Sat99}
Ken-{i}ti Sato.
\newblock {\em L\'{e}vy Processes and Infinitely divisible Distributions}.
\newblock Cambridge University Press, 1999.

\bibitem[Saw18]{Sawano18}
Yoshihiro Sawano.
\newblock {\em Theory of Besov spaces}.
\newblock Springer, 2018.

\bibitem[Tri06]{Tri06}
Hans Triebel.
\newblock {\em Theory of Function Spaces III}.
\newblock Monographs in Mathematics. Birkhäuser Basel, 2006.

\bibitem[TT90]{TT90}
Denis Talay and Luciano Tubaro.
\newblock Expansion of the global error for numerical schemes solving
  sto\-chastic differential equations.
\newblock {\em Stoch. Anal. and App.}, 8-4:94--120, 1990.

\bibitem[Wat07]{Wa07}
Toshiro Watanabe.
\newblock Asymptotic estimates of multi-dimensional stable densities and their
  applications.
\newblock {\em Transactions of the American Mathematical Society},
  359(6):2851--2879, 2007.

\bibitem[XZ20]{XZ20}
Longjie Xie and Xicheng Zhang.
\newblock Ergodicity of stochastic differential equations with jumps and
  singular coefficients.
\newblock {\em Ann. Inst. Henri Poincar\'{e} Probab. Stat.}, 56(1):175--229,
  2020.

\end{thebibliography}
\end{document}